\documentclass[11pt]{article}
\usepackage{amsmath,amssymb,fullpage}

\usepackage[english]{babel} 
\usepackage[latin1]{inputenc} 
\usepackage[T1]{fontenc}
\usepackage{amsmath}
\usepackage[all,dvips]{xy}

\usepackage{tikz}
\usepackage{tikz,fullpage}
\usetikzlibrary{positioning, arrows.meta}
\usetikzlibrary{matrix}

\usepackage{enumerate}
\usepackage{mathtools}

\usepackage{graphicx}
\usepackage{color}
\definecolor{vert}{rgb}{0.02,0.4,0.10}
\usepackage{hyperref}
\hypersetup{
	colorlinks=true,                
	breaklinks=true,                
	linkcolor= vert,                
	citecolor= blue                
}

\usepackage{graphicx}
\graphicspath{ {images/} }
\usepackage{pgf,tikz,pgfplots}
\pgfplotsset{compat=1.14}
\usepackage{mathrsfs}
\usetikzlibrary{arrows}

\usepackage{makeidx}
\usepackage{pstricks}

\providecommand{\otherindexspace}[1]{}

\usepackage{authblk}
\usepackage{amsthm}

\newtheorem{theorem}{Theorem}[section]

\newtheorem{lemma}[theorem]{Lemma}
\newtheorem{proposition}[theorem]{Proposition}
\newtheorem{remark}[theorem]{Remark}

\newtheorem{corollary}[theorem]{Corollary}

\DeclareMathOperator{\supp}{supp}

\numberwithin{equation}{section}

\def\vp{\varepsilon}

\def\cal#1{\mathcal{#1}}

\def \R{\mathbb {R}}
\def \N{\mathbb{N}}
\def \Z{\mathbb{Z}}
\def \Q{\mathbb{Q}}
\def \C{\mathbb{C}}
\def \U{\mathbb{U}}

\def\E{\mathbb{E}}
\def\P{\mathbb{P}}
\def\lb{[\![}
\def\rb{]\!]}

\usepackage{fancyhdr}

\usepackage{graphics}
\usepackage{graphicx}

\DeclareGraphicsExtensions{.eps,.bmp,.jpg,.pdf,.mps,.png,.gif}

\makeatletter
\def\titre{\@title}
\makeatother

\title{Distances on the $\mathrm{CLE}_4$, critical Liouville quantum gravity and $3/2$-stable maps}
\author{Emmanuel Kammerer \thanks{CMAP, \'Ecole polytechnique, Institut Polytechnique de Paris, 91120 Palaiseau, France, \textsf{emmanuel.kammerer@polytechnique.edu}
}}
\date{\today}
\begin{document}

\maketitle

\begin{abstract}
The purpose of this article is threefold. First, we show that when one explores a conformal loop ensemble of parameter $\kappa=4$ ($\mathrm{CLE}_4$) on an independent $2$-Liouville quantum gravity ($2$-LQG) disk, the surfaces which are cut out are independent quantum disks. 
To achieve this, we rely on approximations of the explorations of a $\mathrm{CLE}_4$: we first approximate the $\mathrm{SLE}_4^{\langle \mu \rangle}(-2)$ explorations for $\mu \in \R$ using explorations of the $\mathrm{CLE}_\kappa$ as $\kappa \uparrow 4$ and then we approximate the uniform exploration by letting $\mu \to \infty$. Second, we describe the relation between the so-called natural quantum distance and the conformally invariant distance to the boundary introduced by Werner and Wu. Third, we establish the scaling limit of the distances from the boundary to the large faces of $3/2$-stable maps and relate the limit to the $\mathrm{CLE}_4$-decorated $2$-LQG. 
\end{abstract}

\begin{figure}[h]
	\centering
	\includegraphics[scale=0.48]{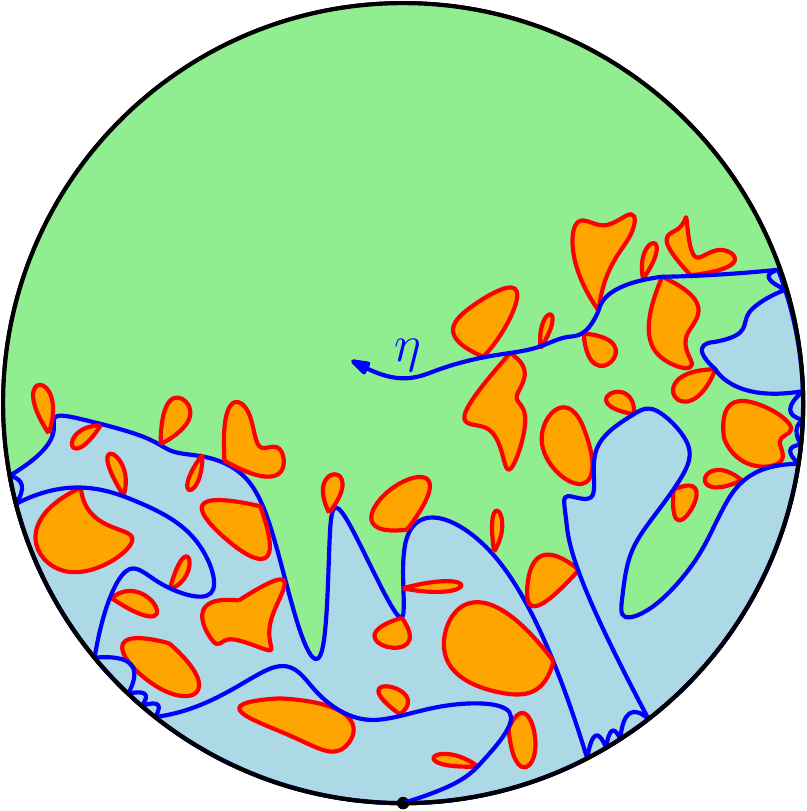}
	\hspace{1.2cm}
	\includegraphics[scale=0.48]{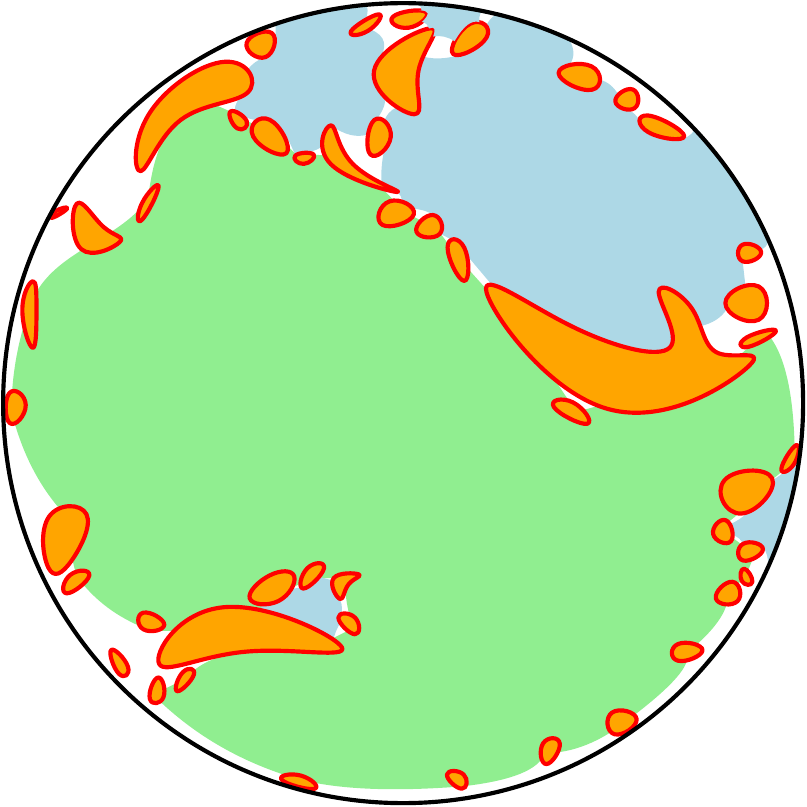}
	\caption{Left: illustration of the $\mathrm{SLE}^{\langle \mu \rangle}_4(-2)$ exploration of the $\mathrm{CLE}_4$ ensemble. Right: illustration of the uniform exploration of the $\mathrm{CLE}_4$ ensemble. The blue curve $\eta$ is the trunk of the exploration. The unexplored region is in green, the domains encircled by loops are in orange and the cut out domains are in blue.}
	\label{tronc et explo unif}
\end{figure}

\tableofcontents

\section{Introduction}
\subsection{Motivation}
This article is motivated by the study of random $3/2$-stable maps and aims to be a first step in the study of their scaling limit when their perimeter goes to $\infty$. {Informally, a $3/2$-stable map is a random map such that every face of degree $k\ge1$ has a weight of order $k^{-2}$, see Subsection \ref{section background} for the precise definition.} Even though we proved in \cite{Kam23} that $3/2$-stable maps, equipped with the graph distance on the dual map, do not satisfy a scaling limit in the usual sense of Gromov-Hausdorff or Gromov-Prokhorov, we show here that at a smaller scale their limiting behavior is related to the conformal loop ensemble $\mathrm{CLE}_4$, where we measure the lengths of the loops using an independent critical Liouville quantum gravity. In order to make sense of the scaling limit, one can order the faces of a $3/2$-stable map in the non-increasing order of their degrees. Similarly, we order the loops of the $\mathrm{CLE}_4$ on the unit disk $\mathbb{D}$ in the non-increasing order of quantum boundary lengths, where the boundary $\partial \mathbb{D}$ is also viewed as a loop. We conjecture that the scaling limit of the $3/2$-stable maps is the set of the loops of the $\mathrm{CLE}_4$ equipped with a distance $\mathrm{d}_{\mathrm{Lamperti}}$ which is defined as a Lamperti transform of the so-called quantum distance introduced in \cite{AHPS21}, {in the sense that the rescaled infinite matrix of distances with the above ordering converges in distribution for the product topology.} {Such a random metric space which appears in random conformal geometry and which is proven to be the scaling limit of random planar maps has not been uncovered since the introduction of the Brownian map \cite{MM06, Mie13, LG13, MS20, MS21, MS21b}. For instance, until now, on the one hand, the Liouville quantum gravity metric of parameter $\gamma\neq \sqrt{8/3}$ defined in \cite{GM21} has not yet been identified as the scaling limit of some natural model of random planar maps and on the other hand, the scaling limit of stable maps of \cite{LGM11,Mar18,CMR25} have not yet been constructed using Liouville quantum gravity and conformal loop ensembles.} However, the quantum distance of \cite{AHPS21} is so far only defined as a distance from the loops to the boundary. In this work, we establish the scaling limit for the distances to the boundary (which provides in particular the tightness for this topology) and we give some insight on the relation between several distances to the boundary in the $\mathrm{CLE}_4$.

\subsection{Background}\label{section background}
The conformal loop ensemble $\mathrm{CLE}_4$ is a random collection of disjoint simple loops included in some simply connected domain. It was introduced and studied in \cite{S09} as part of the one parameter family of simple conformal loop ensembles $\mathrm{CLE}_{\kappa}$ for $\kappa \in (8/3,8)$, whose loops are no longer simple when $\kappa >4$. In \cite{S09}, Sheffield also defined a family of branching explorations parametrized by some ``drift'' $\mu \in \R$ of the $\mathrm{CLE}_4$ starting from the boundary and discovering the loops, called the $\mathrm{SLE}^{\langle \mu \rangle}_4(-2)$ exploration. According to \cite{MSW17} this exploration can be described using a random simple continuous curve called the trunk of the exploration which hits the loops it discovers (see Figure \ref{tronc et explo unif}). When the trunk hits the boundary of the explored region, the unexplored domain is split into two unexplored regions. One can choose for example to continue the exploration of the region which contains a particular point which is called a target point. The whole branching exploration can then be seen as a coupling of the explorations with target points. In \cite{WW13}, Werner and Wu introduced a uniform exploration of the $\mathrm{CLE}_4$ which can be seen as the limit of the $\mathrm{SLE}^{\langle \mu \rangle}_4(-2)$ exploration as $\mu \to \infty$. Informally, it consists in discovering the loop next to a random point chosen according to the harmonic measure on the boundary of the explored region and then updating this random point. The uniform exploration of the $\mathrm{CLE}_4$ is also a branching exploration, given that the unexplored domain is regularly split into two unexplored domains. The time parametrization of the uniform exploration in \cite{WW13} is interpreted as a distance to the boundary. In an {unfinished} work, Sheffield, Watson and Wu {developed a strategy to prove that}
this distance from the boundary to the loops of the $\mathrm{CLE}_4$, which will be denoted by $\mathrm{d}_{\mathrm{WW}}$, 
can be extended to a conformally invariant distance between all the loops of the  $\mathrm{CLE}_4$, including the boundary of the domain. However, we will not use this result in this work.

More recently, the $\mathrm{CLE}_\kappa$ was shown to interact nicely with an independent Liouville quantum gravity surface of parameter $\sqrt{\kappa}$ on the same simply connected domain ($\sqrt{\kappa}$-LQG). An LQG disk parametrized by a simply connnected domain can be defined as a well-chosen version of a Gaussian free-field on this domain which provides a way to measure quantum lengths of some curves and quantum areas of sub-domains. In \cite{MSW22}, the law of the quantum boundary length of the unexplored region along the exploration of the $\mathrm{CLE}_\kappa$ for $\kappa<4$ parametrized by the quantum length of the trunk is shown to behave as some branching Markov process called a growth-fragmentation process, which was introduced by Bertoin in \cite{Ber17}. For the case $\kappa =4$, building on the rich interplay between $\mathrm{CLE}$'s, LQG's, and Brownian motion, and particularly on \cite{DMS21}, Aru, Holden, Powell and Sun identified in \cite{AHPS21} a coupling between a uniform $\mathrm{CLE}_4$ exploration, an independent $2$-LQG disk and a Brownian excursion in the upper half-plane.  Thanks to \cite{AdS22}, this coupling actually enables to describe the quantum boundary length of the explored region along the uniform exploration, parametrized by a  so-called quantum natural distance to the boundary, as a growth-fragmentation process.

Independently, Boltzmann $3/2$-stable maps were first studied in \cite{BCM} in the infinite volume setting. In this introduction, we will only talk about finite $3/2$-stable maps. Recall that a planar map $\mathfrak{m}$ is a connected planar graph embedded in the sphere seen up to orientation preserving homeomorphism, with a distinguished oriented edge $\vec{e}_r$. The face on the right of the root edge is called the root face and is written $f_r$. The number of edges counted with multiplicity around a face $f$ is called its degree and is denoted by $\deg(f)$. We focus on bipartite planar maps, which only have faces with even degrees. For all $\ell \in \N$, let us denote by $\mathcal{M}^{(\ell)}$ the set of bipartite planar maps of perimeter $2\ell$, i.e. such that the degree of the root face is $2 \ell$.
Let ${\bf q} = (q_k)_{k\ge 0}$ be a non-zero sequence of non-negative real numbers. The Boltzmann weight of a map $\mathfrak{m} \in  \mathcal{M}^{(\ell)}$ is defined by
$$
w_{\bf q} (\mathfrak{m}) = \prod_{f \in \mathrm{Faces}(\mathfrak{m})\setminus \left\{f_r\right\}} q_{\deg(f)/2}.
$$
When the \textit{partition function} 
$
W^{(\ell)}\coloneqq \sum_{\mathfrak{m} \in \mathcal{M}^{(\ell)}} w_{\bf q}(\mathfrak{m})
$ is finite, we say that the weight sequence $\bf q$ is admissible. The Boltzmann probability measure on $\mathcal{M}^{(\ell)}$ is then defined by 
$
\P^{(\ell)}(\mathfrak{m}) = {w_{\bf q}(\mathfrak{m})}/{W^{(\ell)}}
$ 
for all $\mathfrak{m} \in \mathcal{M}^{(\ell)}$. The associated expectation is written $\E^{(\ell)}$. {The geometry of random maps of law $\P^{(\ell)}$ heavily depends on the asymptotic behaviour of the partition function $W^{(\ell)}$ as $\ell \to  \infty$. By calculating the partition function and using Laplace's method, one can show that
	$$
	W^{(\ell)} = O(c_{\bf q}^\ell \ell^{-3/2}) \qquad \text{and} \qquad W^{(\ell)} = \Omega(c_{\bf q}^\ell \ell^{-5/2})
	$$
	as $\ell \to \infty$ for some constant $c_{\bf q}>0$. See Section 5.3 of \cite{StFlour}. As in the lecture notes \cite{StFlour}, we say that the weight sequence $\bf q$ is subcritical if there exists a constant $p_{\bf q} > 0$ such that $W^{(\ell)} \sim \frac{p_{\bf q}}{2} c_{\bf q}^{\ell+1} \ell^{-3/2}$ when $\ell \to \infty$. The exponent $-3/2$ is reminiscent of the asymptotics obtained in the enumeration of trees. The sequence $\bf q$ is said critical generic if there exists $p_{\bf q} > 0$ such that $W^{(\ell)} \sim \frac{p_{\bf q}}{2} c_{\bf q}^{\ell+1} \ell^{-5/2}$ as $\ell \to \infty$. When $q_k = 0$ for $k$ large enough, i.e.\@ when the face degrees are bounded, and $\bf q$ is admissible, then either $\bf q$ is subcritical, or it is generic critical. 
	
	To obtain intermediate exponents, we have to favor the appearance of high degrees by choosing $\bf q$ so that the weights decay ``slowly enough''. We say that $\bf q$ is \textit{critical non-generic of type $a$} with $3/2 < a < 5/2$ if there exists $p_{\bf q} > 0$ such that
	\begin{equation}\label{eq critique de type deux}
		W^{(\ell)} \mathop{\sim}\limits_{\ell \to \infty} \frac{p_{\bf q}}{2} c_{\bf q}^{\ell+1} \ell^{-a}.
		\end{equation}}
{Such a sequence $\bf q$ does exist and an explicit example is given by Lemma 6.1 of \cite{BC}.} {Intuitively, \eqref{eq critique de type deux} can be obtained by carefully choosing $\bf q$ with $q_k \sim p_{\bf q} c_{\bf q}^{-k+1} k^{-a}$ and such a polynomial decay favors the appearance of faces with high degrees. A random map of law $\P^{(\ell)}$ is called an \textit{$(a-1/2)$-stable map} of perimeter $2\ell$, for instance since the number of vertices satisfies a scaling limit towards a random variable related to a $1/(a-1/2)$-stable random variable as shown in Proposition 10.4 of \cite{StFlour}.} In what follows, ${\bf q}$ will be assumed \textbf{critical non-generic of type ${\mathbf{a=2}}$}. 

The distances that we will study on these maps are the graph distance on the dual map, called the \textit{dual graph distance}, obtained by exchanging the vertices with the faces, written $\mathrm{d}^\dagger_{\mathrm{gr}}$, and the \textit{first passage percolation distance} which is obtained by putting i.i.d.\@ parameter $1$ exponential random lengths on each edge of the dual map, {written $\mathrm{d}^\dagger_{\mathrm{fpp}}$}. {The study of the distance $\mathrm{d}^\dagger_{\mathrm{fpp}}$ is simpler than $\mathrm{d}^\dagger_{\mathrm{gr}}$ and often gives an intuition on the behaviour of $\mathrm{d}^\dagger_{\mathrm{gr}}$.} {These distances behave in a dramatically different way from the (primal) graph distance on the map since faces with high degrees play the role of hubs. The scaling limit of stable maps equipped with the primal graph distance was first studied in \cite{LGM11}, then it was recently constructed in \cite{CMR25} and it is conjectured to be related to a quantum metric on the CLE carpet/gasket which has not been defined yet.}

The main technique in \cite{BCM} and \cite{Kam23} to study these distances is the peeling exploration first introduced in \cite{B16}. It consists in a step by step Markovian exploration of the map where at each step, we ``peel'' an edge on the boundary of the unexplored region and discover what is behind it. In this setting, we can also observe the evolution {of} the perimeter of the unexplored region (which is the number of edges on its boundary). 
\subsection{Main results}\label{sous-section resultats}

Our main contributions can be informally summarized as follows:
\begin{itemize}
	\item We describe the law of the unexplored regions during the uniform exploration (resp.\@ $\mathrm{SLE}^{\langle \mu \rangle}_4(-2)$ exploration) of the $\mathrm{CLE}_4$-decorated $2$-LQG disk when it is parametrized by the quantum natural distance $\mathrm{d}_\mathrm{q}$ (resp.\@ by the quantum length of the trunk of the $\mathrm{SLE}^{\langle \mu \rangle}_4(-2)$). In particular, we show that for all $t\ge 0$, conditionally on the quantum boundary lengths of the loops which are discovered before time $t$ and conditionally on the quantum boundary lengths of the unexplored regions, the unexplored regions are independent $2$-LQG disks (with prescribed boundary lengths). See Theorem \ref{prop de Markov disques quantiques}.
	\item We relate the distance $\mathrm{d}_\mathrm{WW}$ to the Lamperti transform of the quantum natural distance $\mathrm{d}_\mathrm{q}$ with some continuous increasing process with stationary increments along a particular branch of the uniform exploration. See Theorem \ref{th Lamperti}.
	\item We establish the scaling limit of the distances to the root in $3/2$-stable maps. More precisely, we show that the scaling limit of the degrees of large faces corresponds to the quantum boundary lengths of the loops of the $\mathrm{CLE}_4$ {(this scaling limit can actually be seen as a simple consequence of the results of \cite{BBCK}, \cite{AHPS21} and \cite{AdS22})}, while the scaling limit of the distances from these faces to the boundary corresponds to the Lamperti transform of $\mathrm{d}_\mathrm{q}$. However the quantum distance $\mathrm{d}_\mathrm{q}$ is not the scaling limit of some distance in $3/2$-stable maps, but corresponds to the scaling limit of the time of exploration in the peeling explorations of the maps. See Theorem \ref{gros théorème}.
\end{itemize}
Let us already state our main results more precisely even if the exact definitions will only arise in the next sections.

\paragraph{Exploration of $\mathrm{CLE}_4$ and LQG.}
Recall that the uniform exploration of the $\mathrm{CLE}_4$ is a branching exploration, in the sense that the unexplored domain is gradually split into different unexplored regions which are explored in a Markovian way. {See Subsection \ref{sous-section explo unif} for a definition of the uniform exploration.} {Consider this uniform exploration on top of an independent $2$-LQG disk (or $2$-quantum disk), which is a simply connected domain equipped with a quantum area measure and a quantum boundary length measure which are both constructed using a Gaussian free field {(see Subsection \ref{sous-section disques quantiques} for the definition of LQG disks)}.} For simplicity, we choose to write the first results for a particular branch of the exploration of the $\mathrm{CLE}_4$ on the unit disk $\mathbb{D}$ equipped with an independent $2$-LQG disk, namely the branch of the locally largest component: at each time the {region remaining to be explored} splits in two regions, we choose to explore the region which has the largest quantum boundary length and the other domain is called a cut out domain {(see Subsection \ref{sous-section composante localement la plus grande} for the definition of the exploration branch of the locally largest component)}. See the right side of Figure \ref{tronc et explo unif}. We parametrize the uniform exploration by the quantum natural distance to the boundary and the $\mathrm{SLE}^{\langle \mu \rangle}_4(-2)$ exploration by the quantum length of the trunk. {See Subsection \ref{sous-section rappels CLE}  for the definition of the $\mathrm{SLE}^{\langle \mu \rangle}_4(-2)$ exploration.} These parametrizations are characterized as follows: if we denote by $(Y_t)_{t\ge 0}$ (resp. $(Y^\mu_t)_{t\ge 0}$) the càdlàg process of the quantum boundary length of the locally largest component along the uniform exploration (resp. $\mathrm{SLE}^{\langle \mu \rangle}_4(-2)$ exploration), then we have for all $t\ge 0$ the limit in probability
$$
\lim_{\vp \to 0}  \vp \# \{ s \in [0,t], \ \lvert\Delta Y_s \rvert \in [\vp, 2\vp]\}
=\lim_{\vp \to 0}  \vp \# \{ s \in [0,t], \ \lvert\Delta Y^\mu_s \rvert \in [\vp, 2\vp]\}
=t,
$$
where for any càdlàg process $(X_t)_{t\ge 0}$ we write $\Delta X_t \coloneqq X_t-X_{t-}$ for all $t\ge 0$. The positive jumps of $(Y_t)_{t\ge 0}$ or $(Y^\mu_t)_{t\ge 0}$ occur when the exploration draws a loop while the negative jumps happen when a part of the unexplored region is cut out by the exploration. 

Let us recall the positive self-similar Markov process studied in \cite{AdS22}, which corresponds to a particular case in a family of such processes introduced in \cite{BBCK} for $\theta=1$. Let $\Lambda_1$ be the image by $x \mapsto \log x$ of the measure
$
{1 / (\pi(x(1-x))^2)} {\bf 1}_{x \in(1/2,1)\cup(1,\infty)}  dx.
$
We denote by $\xi$ the Lévy process 
with no Brownian part, with drift $-2/\pi$ and L\'evy measure $\Lambda_1$.
{More precisely, the Laplace exponent $\Psi : \R_+ \to (-\infty,\infty]$ of $\xi$, which is characterized by $\E(\exp(q\xi(t))) = \exp(t\Psi(q))$ for all $t,q\ge 0$, is given by
$$
\forall q \ge 0, \qquad \qquad \Psi(q) = \frac{-2}{\pi} q + \int_\R (e^{qy}-1+q(1-e^y))\Lambda_1(dy).
$$}
 If $\alpha \in \R$, we define the Lamperti time-substitution
\begin{equation}\label{Lamperti}
	\tau(t) =  \inf \left\{ r \ge 0 \ ; \ \int_0^r e^{-\alpha \xi(s)} ds \ge t \right\}.
\end{equation}
For all $x>0$, we write $P_x^{(\alpha)}$ the distribution of the time-changed process
$
X^{(\alpha)}(t) = x e^{\xi(\tau({tx^\alpha}))}
$
with the convention that $X^{(\alpha)}(t)  = \partial$ for $t\ge \zeta \coloneqq x^{-\alpha} \int_0^\infty e^{-\alpha \xi(s)} ds$, where $\partial$ is a cemetery point. Then $X^{(\alpha)}$ is a positive self-similar Markov process of index $\alpha$, in the sense that for every $x>0$, the law of $(xX^{(\alpha)}(x^\alpha t))_{t\ge 0}$ under $P_1^{(\alpha)}$ is $P_x^{(\alpha)}$. By the definition of $\xi$, we can see that either $X^{(\alpha)}$ is absorbed a.s.\@ at the cemetery point $\partial$ after a finite time, or it converges a.s.\@ exponentially fast to zero. In the case $\alpha=-1$, we have $\zeta < +\infty$. 

{The process $X^{(\alpha)}$ is the building block of a branching process: when $X^{(\alpha)}$ performs a negative jump at time $t\ge 0$, one can interpret this jump as the splitting event of a ``cell'' of size $X^{(\alpha)}(t-)$ into two cells of sizes $X^{(\alpha)}(t)$ and $X^{(\alpha)}(t-)-X^{(\alpha)}(t)$. Then one can run another self-similar process with the same law as $X^{(\alpha)}$ starting from $X^{(\alpha)}(t-)-X^{(\alpha)}(t)$. We will not delve into these branching processes until Subsection \ref{sous-section cartes et CLE}.}

{\begin{remark}\label{remarque taux de sauts}The positive self-similar Markov process $X^{(\alpha)}$ can also be described using its jump kernel: at time $t\ge0$ the process $X^{(\alpha)}$ makes a jump of size $y\in \R^*$ from $X^{(\alpha)}(t)=x$ according to the jump kernel
		$$
		K(x,dy)= x^{\alpha} \frac{1}{\pi ((1+y/x)y/x)^2} \frac{dy}{x} {\bf 1}_{y/x-1 \in (1/2,1)\cup (1,\infty)} = \frac{x^{3+\alpha}}{\pi y^2(x+y)^2} dy {\bf 1}_{y>-x/2 \text{ and } y \neq 0} dy.
		$$
		This is the way closely related self-similar Markov processes are described in Paragraph 1.1.3 of \cite{MSW22}. 
		\end{remark}
 }

Our first main result is the analogue of Theorem 1.2 of \cite{MSW22} in the case $\kappa=4$ for the $\mathrm{SLE}_4^{\langle \mu \rangle}(-2)$ exploration and for the uniform exploration.

\begin{theorem}\label{prop de Markov disques quantiques}
Let $\mu \in \R$. The processes $(Y_t)_{t\ge 0}$ and $(Y^\mu_t)_{t\ge 0}$ have the same law as $(X^{(-1)}(\pi t))_{t\ge 0}$. 
Moreover, conditionally on $(Y_s)_{0\le s \le t}$ (resp. $(Y^\mu_s)_{0\le s \le t}$), if we denote by $(x_i)_{i\ge 1}$ the sizes of the jumps of $(Y_s)_{0\le s \le t}$ (resp. $(Y^\mu_s)_{0\le s \le t}$) up to time $t$ ranked in the non-increasing order of absolute value, then the corresponding domains which are encircled by loops, associated with the positive jumps, or which are cut out, associated with the negative jumps, 
and the unexplored region are independent $2$-quantum disks of boundary lengths $|x_i|$'s and $Y_t$ (resp. $Y^\mu_t$).
\end{theorem}

Note that the fact that $(Y_t)_{t\ge 0}$ has the same law as $(X^{(-1)}(\pi t))_{t\ge 0}$ can be seen as a direct consequence of Theorem 5.5 of \cite{AHPS21} and Theorem 1.1 of \cite{AdS22}. {This is not the case for the statement on the quantum disks since the half-plane Brownian excursion of \cite{AHPS21} does not characterize the pair $\mathrm{CLE}_4$/$2$-quantum disk.} Besides, the above theorem {extends} Theorem 1.1 of the recent article \cite{AG23} {which focuses on the regions encircled by loops}. {Actually, in the proof of the above result, we will use a few results of \cite{AG23}.}

\paragraph{Quantum natural distance and the conformally invariant distance.} 
For all $u\ge 0$, we denote by $C_u$ the unexplored region at time $u$ when following the locally largest component in the uniform exploration of \cite{WW13} parametrized by the distance $\mathrm{d}_\mathrm{WW}$. {We stress that this time parametrization is \textit{different from the time parametrization by the quantum natural distance} used in the previous paragraph, which was denoted by $t$. In general, in this paper the letter $t\ge 0$ is related to the quantum natural distance.} One can extend the definition of the quantum natural distance from loops to the boundary in order to measure also the quantum natural distance $\mathrm{d}_\mathrm{q}(\partial \mathbb{D}, C_u)$ from the unexplored region $C_u$ to the boundary, so that if $u$ is a time of discovery of a loop $\mathcal{L}$, then $\mathrm{d}_\mathrm{q}(\partial \mathbb{D}, C_u)$ corresponds to the quantum distance $\mathrm{d}_\mathrm{q}(\partial \mathbb{D}, \mathcal{L})$ from the boundary to $\mathcal{L}$. See Subsection \ref{sous-section composante localement la plus grande} for the precise definition. Let $(\sigma(t))_{t\ge 0}$ be the right-continuous inverse of $u \mapsto \mathrm{d}_\mathrm{q}(\partial \mathbb{D}, C_u)$. The next result, a consequence of Theorem \ref{prop de Markov disques quantiques}, gives some insight on $(\sigma(t))_{t\ge 0}$.

\begin{theorem}\label{th Lamperti}
	The process $(\sigma(t))_{t\ge 0}$ is continuous and increasing. Moreover, the continuous increasing process $S$ characterized by
$$
	\forall t \ge 0, \qquad \sigma(t) = S\left(\int_0^t \frac{1}{Y_s} ds\right)
$$
	has stationary increments.
\end{theorem}

{If one identifies $Y$ with the process $(X^{(-1)}(\pi t))_{t\ge 0}$ using Theorem \ref{prop de Markov disques quantiques}, then one can also define the process $S$ by setting $S(t) = \sigma(\tau^{-1}(\pi t)/\pi)$, where $\tau^{-1}$ is the inverse bijection of $\tau$ defined in \eqref{Lamperti} with $\alpha=-1$, which can also be written $\tau^{-1}(t)= \int_0^t e^{\xi(s)} ds$ for all $t\ge 0$.} 
The above result motivates the introduction of another distance to the boundary that we first define here only for the loops which are discovered by the branch of the locally largest component (see \eqref{def d Lamperti} for its definition for all the loops of the $\mathrm{CLE}_4$): for every loop $\mathcal{L}$ which is discovered by the exploration of the locally largest component, we set $\mathrm{d}_{\mathrm{Lamperti}} (\partial \mathbb{D}, \mathcal{L})\coloneqq \int_0^{\mathrm{d}_{\mathrm{q}}(\partial \mathbb{D}, \mathcal{L})} {ds}/{Y_s} $. Then we have $\mathrm{d}_{\mathrm{WW}}(\partial \mathbb{D}, \mathcal{L})=S(\mathrm{d}_\mathrm{Lamperti}(\partial \mathbb{D}, \mathcal{L}))$. However, we do not believe that $S$ has independent increments. The independence of the increments would entail that $S$ is a deterministic linear function, hence $\mathrm{d}_\mathrm{Lamperti}$ and $\mathrm{d}_{\mathrm{WW}}$ would be equal up to some multiplicative constant and Subsection \ref{sous-section heuristique} provides some heuristics against it.

\paragraph{$3/2$-stable maps and $\mathrm{CLE}_4$-decorated LQG.} 
Denote by $(P_*(n))_{n\ge 0}$ the perimeter process during the peeling exploration of the locally largest component (in terms of perimeter) under the law $\P^{(\ell)}$. Then it is known from Proposition 6.6 of \cite{BBCK} (which comes from Theorem 1 of \cite{BCM}) that under $\P^{(\ell)}$ the rescaled process $(P_*(\lfloor \ell t\rfloor )/\ell)_{t\ge 0}$ converges towards $(X^{(-1)}(\pi p_{\bf q} t))_{t\ge 0}$ where $p_{\bf q}$ is introduced in \eqref{eq critique de type deux} {and where $X^{(-1)}$ is defined using the Lamperti transform \eqref{Lamperti}}, so that by Theorem \ref{prop de Markov disques quantiques} the quantum length of the trunk of the $\mathrm{SLE}_4^{\langle \mu \rangle}(-2)$ exploration and the quantum natural distance corespond to the scaling limit of the time in the peeling exploration. The next results extend this relation between $3/2$-stable maps and $\mathrm{CLE}_4$-decorated LQG by establishing the scaling limit of the distances to the boundary.

By using some particular peeling algorithm, we can explore the map in a way that follows the growth of the balls centred at the root face $f_r$ or equivalently such that the distances from the edges on the boundary of the unexplored region to the root face are roughly the same. For example, for $\mathrm{d}^\dagger_{\mathrm{gr}}$, the idea is to explore the map layer by layer: at first the edges at distance $0$ from the root face, then at distance $1$, etc. This peeling algorithm is called the peeling by layers algorithm. For $\mathrm{d}^\dagger_{\mathrm{fpp}}$, we use the uniform peeling algorithm which takes at each step an edge uniformly at random on the boundary. The peeling explorations also have a branching nature since there are peeling steps when the unexplored region splits into two regions. A particular branch of the exploration, called the peeling exploration of the locally largest component, consists in choosing to explore the region which has the largest perimeter at each splitting event. See Subsection \ref{rappels épluchage} for the definition of the peeling explorations. Let $\partial\overline{\mathfrak{e}}_n$ be the boundary of the explored region during the exploration of the locally largest component at time $n\ge0$, using the appropriate peeling exploration for each distance. Then we prove that the radius of the balls during the exploration evolves in the following way:
\begin{align}
&\text{under } \P^{(\ell)}, \qquad (\mathrm{d}^\dagger_{\mathrm{fpp}} (f_r, \partial \overline{\mathfrak{e}}_{\lfloor \ell t \rfloor}))_{t\ge 0}
\mathop{\longrightarrow}\limits_{\ell \to \infty}^\mathrm{(d)}
\left(\int_0^t \frac{1}{2Y(p_{\bf q} s)} ds \right)_{t\ge 0};\label{cv rayon fpp}
\\
&\text{under } \P^{(\ell)}, \qquad (\mathrm{d}^\dagger_{\mathrm{gr}} (f_r, \partial \overline{\mathfrak{e}}_{\lfloor \ell t \rfloor})/(\log \ell))_{t\ge 0}
\mathop{\longrightarrow}\limits_{\ell \to \infty}^\mathrm{(d)} \left(\int_0^t \frac{p_{\bf q}}{2Y(p_{\bf q} s)} ds \right)_{t\ge 0};\label{cv rayon graphe}
\end{align}
for the $J_1$ Skorokhod topology. {Note that the above scaling limit can be seen as the Lamperti transform used to come back from the time parametrization of the self-similar Markov process $X^{(-1)}$ to the time parametrization of the Lévy process $\xi$ in the definition of $X^{(-1)}$. In other words, the scaling limit of the exploration time is the time parametrization of $X^{(-1)}$ while the scaling limit of the distances to the root face is the time parametrization of $\xi$.}

All our statements only focused so far on the locally largest component. But by the Markov property of Theorem \ref{prop de Markov disques quantiques} for the $\mathrm{CLE}_4$ exploration, and by the spatial Markov property of the peeling exploration, they can be extended to the whole exploration tree. For each face $f$ of a map $\mathfrak{m}$, we may define the exploration time $n(f)$ of $f$ as the number of peeling steps before discovering $f$ which occurred in the peeling exploration branch of $\mathfrak{m}$ which discovers $f$. Note that $n(f)$ depends on the choice of the peeling algorithm. See Subsection \ref{sous-section cartes et CLE} for a precise definition of $n(f)$. This enables us to state the scaling limit of the degree, the time of exploration and the distance from the boundary to the large faces of $3/2$-stable maps towards the quantum boundary length, the quantum natural distance (or quantum length of the trunk depending on the exploration) and its Lamperti transform of a $\mathrm{CLE}_4$-decorated LQG disk. For every loop $\mathcal{L}$ of the $\mathrm{CLE}_4$, recall that we denote by $\mathrm{d}_{\mathrm{q}}(\partial \mathbb{D}, \mathcal{L})$ the quantum natural distance to $\mathcal{L}$. We also denote by $\mathrm{d}_{\mathrm{Lamperti}}(\partial \mathbb{D}, \mathcal{L})$ its Lamperti transform defined by
\begin{equation}\label{def d Lamperti}
\mathrm{d}_{\mathrm{Lamperti}}(\partial \mathbb{D}, \mathcal{L}) \coloneqq  \int_0^{ \mathrm{d}_{\mathrm{q}}(\partial \mathbb{D}, \mathcal{L})} \frac{dt}{Y_{\mathcal{L}}(t)},
\end{equation}
where $(Y_{\mathcal{L}}(t))_{0 \le t \le \mathrm{d}_{\mathrm{q}}(\partial \mathbb{D}, \mathcal{L})}$ is the quantum boundary length of the unexplored region surrounding $\mathcal{L}$ parametrized by the quantum natural distance. One can define it using an exploration with a target point $z$ inside the loop. See Subsection \ref{sous-section composante localement la plus grande} for a definition of $Y_{\mathcal{L}}$ and $\mathrm{d}_{\mathrm{q}}(\partial \mathbb{D}, \mathcal{L})$.
\begin{theorem}\label{gros théorème}
For all $\ell \ge 1$, let $\mathfrak{M}^{(\ell)}$ of law $\P^{(\ell)}$. Let $(f_i^{(\ell)})_{i\ge 1}$ be the collection of faces of $\mathfrak{M}^{(\ell)}$ ranked in the non-increasing order of degree. Let $(n(f_i^{(\ell)}))_{i\ge 1}$ be their exploration times using the peeling by layers algorithm. Let $(\mathcal{L}_i)_{i\ge 1}$ be the collection of loops of the $\mathrm{CLE}_4$ ranked in the non-increasing order of quantum boundary length. Let us denote by $\nu^2_{h^2}(\mathcal{L})$ the quantum boundary length of a loop $\mathcal{L}$. Then, for the product topology,
$$
\left(\frac{\deg(f_i^{(\ell)})}{2\ell}, \frac{n(f_i^{(\ell)})}{p_{\bf q} \ell}, \frac{2\mathrm{d}^\dagger_\mathrm{gr}(f_r, f_i^{(\ell)})}{ \log \ell}\right)_{i\ge 1}
\mathop{\longrightarrow}\limits_{\ell \to \infty}^{(\mathrm{d})}
\left(\nu^2_{h^2}(\mathcal{L}_i), \mathrm{d}_{\mathrm{q}}(\partial \mathbb{D}, \mathcal{L}_i), \mathrm{d}_{\mathrm{Lamperti}}(\partial \mathbb{D}, \mathcal{L}_i)\right)_{i\ge 1}.
$$
\end{theorem}
The same theorem can be stated for the $\mathrm{SLE}_4^{\langle \mu \rangle}(-2)$ exploration, replacing the quantum natural distance by the quantum length of the trunk. For the fpp distance, the factor $\log \ell$ is replaced by a factor $1/p_{\bf q}$ and the peeling by layers exploration is replaced by the uniform peeling exploration.

{The scaling limit of the dual graph distances to the root was obtained in \cite{BCK,BBCK} in the cases $a \in (2,5/2]$, relying on the work \cite{BC} in the cases $a \in (2,5/2)$ where the scale of the distances to the root under $\P^{(\ell)}$ is $\ell^{a-2}$. See Theorem 6.8 in \cite{BBCK}. The proof of Theorem \ref{gros théorème} makes use of the same kind of branching processes as in \cite{BCK,BBCK}. However, dealing with the dual graph distance in the case $a=2$ is much more delicate than in the case $a >2$ since it involves Cauchy processes. To obtain Theorem \ref{gros théorème} for the dual graph distance, we use techniques developed in \cite{BCM} for this special case.}

\subsection{Outline}
In Section \ref{explo unif et LQG} we recall the definition of the uniform exploration of the $\mathrm{CLE}_4$ from \cite{WW13} and the definitions of $\gamma$-LQG disks for $\gamma \in ({0},2]$ from \cite{AHPS21}. The definition of the $\mathrm{SLE}^{\langle \mu \rangle}_4(-2)$ exploration is postponed to the next section.

Then, in Section \ref{section kappa vers 4}, after introducing the remaining definitions, we prove Theorem \ref{prop de Markov disques quantiques} for the $\mathrm{SLE}^{\langle \mu \rangle}_4(-2)$ exploration by approximating the exploration with some well-chosen $\mathrm{SLE}_\kappa^\beta(\kappa-6)$ exploration, using the results of \cite{Le21} and by approximating the $2$-LQG disks with $\sqrt{\kappa}$-LQG disks relying on \cite{AHPS21}. 

Next, in Section \ref{section mu vers l'infini}, so as to obtain the statement of Theorem \ref{prop de Markov disques quantiques} for the uniform exploration, we approximate the uniform exploration with the $\mathrm{SLE}^{\langle \mu \rangle}_4(-2)$ exploration as $\mu \to \infty$. This section relies partly on the same ideas as in Section \ref{section kappa vers 4}.

On the basis of Theorem \ref{prop de Markov disques quantiques}, using the Poissonian structure of the uniform exploration in \cite{WW13}, we prove Theorem \ref{th Lamperti} in Section \ref{section Lamperti}. 

In Section \ref{section cartes}, we first establish the metric growth in the infinite volume setting using results of \cite{BCM}. By absolute continuity, we prove the convergences (\ref{cv rayon fpp}) and (\ref{cv rayon graphe}). We then describe the limit of the whole exploration tree and deduce Theorem \ref{gros théorème}. Our approximation results are informally summarized by the following diagram:

\begin{center}
\noindent\begin{tikzpicture}[node distance=0.3cm and 4.4cm, on grid, auto]
    \node (A) at (0,0) {$\mathrm{CLE}_\kappa$ exploration};
    \node[right=of A] (B) {$\mathrm{SLE}_4^{\langle \mu \rangle}(-2)$};
    \node[right=of B] (C) {Uniform exploration};
    \node[right=of C] (D) {$3/2$-stable map};
		
		\node[below=of A] (A2) {{\qquad \qquad \qquad \qquad \qquad}};
    \node[below=of B] (B2) {{\qquad \qquad \qquad \, \, }};
    \node[below=of C] (C2) {{\qquad \qquad \qquad \qquad \qquad  }};
		\node[below=of D] (D2) {{\qquad \qquad \qquad \qquad }};

    \node[below=of A2] (Atext) {on a $\sqrt{\kappa}$-LQG disk};
    \node[below=of B2] (Btext) {on a $2$-LQG disk};
    \node[below=of C2] (Ctext) {on a $2$-LQG disk};
		\node[below=of D2] (Dtext) {of perimeter $\ell$};

    \draw[->, line width=1pt] (A2) -- node[below] {$\kappa \uparrow 4$} node[above] {Sec. \ref{section kappa vers 4}} (B2);
    \draw[->, line width=1pt] (B2) -- node[below] {$\mu \to \infty$} node[above] {Sec.\ref{section mu vers l'infini}} (C2);
    \draw[<->, line width=1pt] (D2) -- node[below] {$\ell \to \infty$} node[above] {Sec.\ref{section cartes}} (C2);
\end{tikzpicture}
\end{center}
\qquad

Finally in Section \ref{section discussion}, we describe the applications to $O(2)$-loop decorated planar maps and we argue why we believe that the conformally invariant distance from \cite{WW13} is not the Lamperti transform of the quantum natural distance.

\section{Preliminaries on the uniform exploration of the $\mathrm{CLE}_4$ and LQG}\label{explo unif et LQG}

In this section, we first define the uniform exploration of the $\mathrm{CLE}_4$ and then we give some background on Liouville quantum gravity and quantum disks.

\subsection{The uniform exploration of the $\mathrm{CLE}_4$}\label{sous-section explo unif}

We recall here the uniform exploration of the $\mathrm{CLE}_4$ introduced in \cite{WW13}. {Let $\mathcal{C}(\R_+,\R_+)$ be the set of continuous functions from $\R_+$ to $\R_+$.} 
Let $E$ be the space of positive excursions defined by
$$
E= \{e \in \mathcal{C}(\R_+, \R_+); \ e(0)=0 \text{ and } \tau(e) \coloneqq \sup\{s>0; \ e(s)>0\} \in (0,\infty)\}.
$$
Let $\lambda$ be the {infinite measure on $E$ called the Itô measure of positive excursions of standard Brownian motion}, which {is such that} for any $x>0$,
the measure of the subset $E_x$ of excursions of height at least $x$ is finite and if one renormalizes $\lambda$ so that it is a probabillity measure on $E_x$, then after reaching $x$ the excursion has the same law as a {standard} Brownian motion {starting from $x$} which is stopped when attaining zero. {See e.g.\@ \cite{LG10} or Chapter XII of \cite{RY05} for a definition of Itô's excursion measure of Brownian motion (see Chapter VI Section 8 of \cite{RW00} for the general theory of excursions of Markov processes). The above description can be seen as a consequence of the point (iii) in Theorem (48.1) of Chapter VI in \cite{RW00}.} 
We denote by $\mathbb{H}$ the upper half-plane. If $e\in E$ is an excursion of duration $\tau(e)$, we define the driving function $w$ on $[0,\tau(e)]$ by
$$
\forall t \in [0,\tau(e)], \qquad w_t = 2e(t) -\int_0^t \frac{ds}{e(s)}.
$$
It is shown in \cite{SW12} that for $\lambda$-almost all $e \in E$, the conformal maps $g_t$ characterized by the ODE
$$
g_0(z)=z \qquad \text{and}\qquad \forall t\ge 0, \enskip \partial_tg_t(z) = \frac{2}{g_t(z)-w_t}
$$
define a simple loop $\gamma= \gamma(e)$ in the upper half plane starting at zero, in the sense that the domain of definition of $g_t$ can be written $\mathbb{H}\setminus \gamma((0,t])$. {The above ODE, called Loewner's differential equation, is classically used to define Schramm-Loewner evolutions (SLE). Let us give an intuition for why we choose such a driving function $w$. Intuitively, the path $\gamma$ returns to $0$ when $w_t= g_t(\gamma(t)) = g_t(0)$. But letting $v_t \coloneqq g_t(0)-w_t$, we have $dv_t = \partial_tg_t(0) dt - dw_t  = (2/v_t)dt -2de(t) + dt/e(t)$ so that one can see that $v_t= -2e(t)$ (since $v_0 = 2e(0)=0$). This ``shows'' that the path $\gamma$ returns to $0$ at the end of the excursion.} 

We define the measure $\mu^0$ as the pushforward measure $\gamma_*\lambda$. For all $x\in \R$, we denote by $\mu^x$ the measure on loops rooted at $x$ obtained by translating $\mu^0$. {The measure $\mu^x$ is often called the SLE$_4$ ``bubble'' measure as in \cite{WW13} or the ``one-point pinned'' measure as in \cite{SW12}.}

Let $z\in \mathbb{H}$. Let us first define the uniform exploration targeted at $z$. Let $(\widehat{\gamma}_u)_{u\ge 0}$ be a {Poisson point process (PPP)} of intensity $M= \int_\R dx \mu^x$. {We denote by $\partial$ the cemetery point.} For every loop $\widehat{\gamma}_u$, we denote by $\widehat{f}_u$ the conformal transformation from the connected component of $\mathbb{H}\setminus \widehat{\gamma}_u$ containing $z$ onto $\mathbb{H}$ which fixes the point $z$ and such that $\widehat{f}_u'(z) \in \R_+$.  We define $\widehat{F}^z_u  = \widehat{F}_u \coloneqq\circ_{0\le v<u} \widehat{f}_v$ where the composition is done in the order of appearance of the maps $\widehat{f}_u$. This indeed defines a conformal map by p.19 of \cite{WW13}. We also denote by $\tau^z$ the first time at which the loop $\widehat{\gamma}_u$ surrounds the point $z$. We define $ (\widetilde{\gamma}^z_u)_{u\le \tau^z} \coloneqq (\widehat{F}_u^{-1}(\widehat{\gamma}_u))_{ u\le \tau^z}$.

 Lemma 8 from \cite{WW13} shows that 
 it is possible to couple two processes $(\widetilde{\gamma}^z_u)_{u\le \tau^z}$ and $(\widetilde{\gamma}^{z'}_u)_{u\le \tau^{z'}}$ so that they coincide until the time at which they disconnect $z$ from $z'$. If we perform this coupling for a dense countable subset of $\mathbb{H}$ then the union of the loops surrounding points of this subset of $\mathbb{H}$ is the $\mathrm{CLE}_4$. The processes $(\widetilde{\gamma}^z_u)_{u\le \tau^z}$ for $z \in \mathbb{H}$ can then be seen as branches of the uniform exploration. {There is also a nested version of the CLE$_4$ which is obtained recursively by putting independent CLE$_4$'s in the domains which are surrounded by loops.} If we keep on drawing loops with $(\widehat{F}_u^{-1}(\widehat{\gamma}_u))_{u\ge 0}$ after the time $\tau^z$, we then build the nested $\mathrm{CLE}_4$. Except in Sections \ref{section mu vers l'infini} and \ref{sous-section cartes deco}, we will always consider the non-nested version of the $\mathrm{CLE}_4$ and it will simply be referred {to} the $\mathrm{CLE}_4$.

For future use, we come back to the unit disk $\mathbb{D}$ via a conformal transformation $\Phi: \mathbb{H} \rightarrow \mathbb{D}$ and we define the domain $D_u\subseteq \mathbb{D}$ of the unexplored points (during the uniform exploration of the $\mathrm{CLE}_4$) at time $u\ge 0$ by $D_u = \Phi(\left\{ z\in \mathbb{H};\  \tau^z >u\right\})$. 
We also define for all $u\ge 0$ and $z \in \mathbb{D}$ the domain $\widetilde{D}^z_u\subseteq \mathbb{D}$ consisting in the unexplored points at time $u$ in the uniform exploration of the nested $\mathrm{CLE}_4$ targeted at $z$ by $\widetilde{D}^z_u= \Phi((F_u^{\Phi^{-1}(z)})^{-1}(\mathbb{H}) )$. 
The choice of $\Phi$ does not change the law of the exploration by Proposition 9 of \cite{WW13}. {The time $u\ge 0$ at which a loop $\mathcal{L}$ is discovered can be seen as a distance from the loop to the boundary and we also denote it by $\mathrm{d}_{\mathrm{WW}}(\partial \mathbb{D}, \mathcal{L})$.}

\subsection{Critical Liouville quantum gravity and quantum disks}\label{sous-section disques quantiques}

We recall here the definitions of the free boundary Gaussian free field (GFF) and of the $\gamma$-quantum disks for $\gamma \in ({0},2]$ following \cite{AHPS21}{, which are the surfaces that are expected to appear as scaling limits of discrete models such as random maps coupled with statistical mechanics models}. {The $\gamma$-quantum disks were first introduced in \cite{HRV18} and in \cite{DMS21}.} See also \cite{She07} for more details on GFF and \cite{BP} for more about GFF and Liouville quantum gravity. We follow exactly the definitions in Subsection 4.1 of \cite{AHPS21}.  Let $D \subset \C$ be a simply connected domain with harmonically non-trivial boundary {(i.e.\@ for any point $z\in D$, a Brownian motion started at $z$ almost surely hits $\partial D$)}. We denote by $H(D)$ the Hilbert space closure of the space of real functions $f$ of class $\mathcal{C}^\infty$ on $D$ modulo constants which have finite Dirichlet energy (i.e. $\int_D \lvert \nabla f \rvert^2 <\infty$), equipped with the Dirichlet inner product defined for any pair $f,g$ of such functions by
$$
(f\vert g)_\nabla = \frac{1}{2\pi} \int_{D} \nabla f(z) \cdot \nabla g(z) d^2 z,
$$
where $d^2z$ is the Lebesgue measure on $\C$.
Let $(f_n)_{n\ge 0}$ be a Hilbert basis of $H(D)$. Let $(X_n)_{n\ge 0}$ be i.i.d. random variables of law $\mathcal{N}(0,1)$. The \textit{free boundary GFF} on $D$ is defined as the sum
$$
h=
\sum_{n \ge 0} X_n f_n,
$$
which converges in the space of distributions modulo constants. Even though the resulting random variable $h$ lives in the space of distributions modulo constants, one can fix that constant by assigning some particular value to $h$ on an arbitrary test function. The resulting random distribution is actually in the Sobolev space $H^{-1}_\mathrm{loc}(D)$ which is the space of distributions such that their restrictions to any relatively compact open domain $U\subset D$ are in the Sobolev space $H^{-1}(U)$. We endow the space $H^{-1}_\mathrm{loc}(D)$ with the corresponding limit topology which is metrizable and separable, such that a sequence converges in $H^{-1}_\mathrm{loc}(D)$ if and only if the restrictions to any relatively compact open domain $U$ converge in $H^{-1}(U)$.

Next, we recall that a $\gamma$-LQG surface for $\gamma\in ({0},2]$ can be seen as an equivalence class of domains $D\subset \C$ equipped with a variant of the GFF $h$ (a free-boundary GFF plus a continuous random function on $\overline{D}$) under conformal maps. To that instance of the GFF, we associate an area and a boundary length measures $\mu^\gamma_h$ and $\nu^\gamma_h$.  
We define the \textit{quantum area and quantum boundary length measures} $\mu^\gamma_h$ and $\nu^\gamma_h$ by the limits in probability: for $\gamma=2$, for all bounded open set $A \subseteq D$ we set
$$
\mu^2_h (A)= \lim_{\delta \to 0}  \int_A \left(-h_\delta + \log\left(\frac{1}{\delta}\right)\right)e^{2 h_\delta(z)} \delta d^2z $$
and for all bounded open set $B\subseteq \partial D$
$$
\nu^2_h (B) = \lim_{\delta \to 0} \int_B  \left(-\frac{h_\delta}{2} + \log\left(\frac{1}{\delta}\right)\right)e^{h_\delta(z)}\delta dz,
$$
while for $\gamma \in ({0},2)$, for all bounded open set $A \subseteq D$, we set
$$
\mu_{h}^\gamma (A)= \lim_{\delta \to 0} (8(2-\gamma))^{-1} \int_A \exp \left( \gamma h_\delta(z)\right) \delta^{\gamma^2/2} d^2z $$
and for all bounded open set $B\subseteq \partial D$
$$
\nu_{h}^\gamma (B) = \lim_{\delta \to 0} (8(2-\gamma))^{-1}\int_B  \exp\left(\frac{\gamma}{2} h_\delta(z)\right) \delta^{\gamma^2/4} dz,
$$
where $d^2z$ is the Lebesgue measure on $\C$, $dz$ is the Lebesgue measure on $\R$ when $D$ is the upper half-plane $\mathbb{H}$ and $h_\delta(z)$ is the circular mean of $h$ on the circle of center $z$ and radius $\delta$.

We say that two pairs $(D,h)$ and $(\widetilde{D},\widetilde{h})$, where $D,\widetilde{D}$ are domains and $h,\widetilde{h}$ are distributions, are equivalent (as $\gamma$-quantum surfaces) if there exists a conformal map $\varphi: D\rightarrow \widetilde{D}$ such that 
\begin{equation}\label{covariance}
	h= \widetilde{h}\circ \varphi + Q_\gamma \log |\varphi'|\qquad \text{where} \qquad Q_\gamma = 2/\gamma+ \gamma/2. 
\end{equation}
If $h, \widetilde{h}$ are absolutely continuous with respect to a GFF plus a continuous function, then $\mu^\gamma_{\widetilde{h}} = \varphi_*\mu^\gamma_h$ and $\nu^\gamma_{\widetilde{h}} =\varphi_* \nu^\gamma_h$.

Informally, a $\gamma$-quantum disk (also called $\gamma$-LQG disk) corresponds to the ``domain encircled by a $\mathrm{CLE}_{\gamma^2}$ loop with the restriction of an {independent} GFF {on $\mathbb{C}$}''. A quantum disk has a finite boundary length and a finite volume. More precisely, our unit-boundary length $\gamma$-quantum disks are defined as in Definition 4.1 for $\gamma \in ({0},2)$ in \cite{AHPS21} and Definition 4.3 for $\gamma= 2$ from \cite{AHPS21}. {They correspond up to a constant to the original definition in \cite{DMS21}.} Let us summarize their definition although we will not rely on it since we will use directly the results of \cite{AHPS21}. The definition is easier to state when we parametrize the disk by $\mathcal{S}\coloneqq \R \times (0,\pi) \subset \C$, the infinite strip. Let $h_1^\gamma$ be a function on $\mathcal{S}$ defined for all $s \in \R$ and $y \in (0,\pi)$ by $h_1^\gamma(s+ i y)= \mathcal{B}^\gamma_s$, where $\mathcal{B}^\gamma$ is defined by:
\begin{enumerate}[(i)]
	\item For $\gamma \in ({0},2)$, the function $(\mathcal{B}^\gamma_s)_{s \ge 0}$ has the same law as $B_{2s}-{(2/\gamma - \gamma/2)s}$ conditioned to stay negative for all $s\ge 0$, where $(B_s)_{s \ge 0}$ is a standard Brownian motion, while $(\mathcal{B}^2_s)_{s\ge 0}$ is $(-\sqrt{2})$ times a Bessel process of dimension $3$ (which corresponds in some sense to $(B_{2s})_{s\ge 0}$ conditioned to stay negative by a famous result of Williams).
	\item The process $(\mathcal{B}^\gamma_{-s})_{s \ge 0}$ is independent of $(\mathcal{B}^\gamma_s)_{s \ge 0}$ and has the same law as $(\mathcal{B}^\gamma_s)_{s \ge 0}$.
\end{enumerate}
Let $h_2$ be the orthogonal projection of a free boundary GFF on $\mathcal{S}$ on the closure in $H(\mathcal{S})$ of the functions of mean zero along the vertical lines of $\mathcal{S}$. Set $h^\gamma_s = h_1^\gamma+h_2$ and let $\widehat{h}^\gamma$ be a random distribution which has the law of $h^\gamma_s - (2/\gamma) \log \nu^\gamma_{h^\gamma_s}(\partial \mathcal{S})$ reweighted by $\nu^\gamma_{h^\gamma_s}(\partial \mathcal{S})^{(4/\gamma^2)-1}$. Then the \textit{$\gamma$-quantum disk} can be defined as the equivalence class of $(\mathcal{S}, \widehat{h}^\gamma)$. We denote by $(\mathbb{D},h^\gamma)$ the parametrization by the unit disk $\mathbb{D}$ of the unit boundary length $\gamma$-quantum disk.

One can define a marked quantum disk by sampling a marked boundary point according to the quantum boundary length measure, but we will not do this in this work. For $b>0$, we define the $b$-boundary length $\gamma$-quantum disk to be the quantum surface parametrized by $(\mathbb{D},h^\gamma+(2/\gamma)\log b)$. Note that $\nu^\gamma_{h^\gamma+(2/\gamma)\log b} = b\nu^\gamma_{h^\gamma}$ and $\mu^\gamma_{h^\gamma+(2/\gamma)\log b}= b^2 \mu^\gamma_{h^\gamma}$.

\subsection{The quantum natural distance and the locally largest component}\label{sous-section composante localement la plus grande}

Let $(\mathbb{D}, h^2)$ be a parametrization of a $2$-quantum disk by $\mathbb{D}$ and consider an independent uniform exploration of a $\mathrm{CLE}_4$ on $\mathbb{D}$. For all $z \in \Q^2 \cap \mathbb{D}$, for all $u \in [0,\tau^{z})$, let $D_u^{z}$ be the connected component of $D_u$ containing $z$. {The time $u$ is also interpreted as a distance from $D_u^z$ to the boundary of the disk and is written $u=\mathrm{d}_{\mathrm{WW}}(\partial \mathbb{D}, D_u^z)$.} For all $z \in \Q^2 \cap \mathbb{D}$, using the restriction of $h^2 $ to $D_u^{z}$, one can define the quantum boundary length 
$$
	\mathcal{Y}^{z}_u\coloneqq \nu^2_{h^2_{\vert D_u^{z}}} (\partial D_u^{z})
$$
 of $D_u^{z}$ since it is encircled by $\mathrm{SLE}_4$-type curves. {Note that, as explained in the previous subsection, in the above expression the right-hand side is defined as a limit in probability. Still, we can make sense of the jump of $\mathcal{Y}^z$ at time $u\ge 0$ as the random variable $\Delta \mathcal{Y}^z_u =\nu^2_{h^2_{\vert D_u^{z}}} (\partial D_u^{z})- \nu^2_{h^2_{\vert D^z_{u-}}}(\partial D^z_{u-})$. Actually, as we will see in the rest of this subsection, the process $\mathcal{Y}^z$ has a càdlàg version and $\Delta \mathcal{Y}^z_u$ is indeed the jump of $\mathcal{Y}^z$ at time $u$.}

 {
 	We recall the following result of \cite{AHPS21}, which comes from Equation (5.2) in Theorem 5.5 of \cite{AHPS21}.
 	\begin{lemma}\label{lemme cv distance quantique naturelle}(consequence of Equation (5.2) in \cite{AHPS21})
 		For all $z \in \mathbb{D}$, the quantity
 		$$
 		  \vp \# \{ s \in [0,\tau^z], \ \lvert\Delta \mathcal{Y}^z_s \rvert \in [\vp, 2\vp]\}
 		  $$
 		  converges in probability as $\vp \to 0$ to a quantity called the quantum natural distance to the loop surrounding $z$.
 	\end{lemma}
 	For all $u\in [0,\tau^z)$ such that $\widehat{\gamma}_u \neq \partial$, following Remark 2.3 of \cite{MSW22}, we define the \textit{quantum natural distance} from $\partial \mathbb{D}$ to $D^z_u$ as the limit in probability
 	\begin{equation}\label{eq distance quantique target}
 		\mathrm{d}_\mathrm{q} (\partial \mathbb{D}, D^{z}_u) = \lim_{\vp \to 0}  \vp \# \{ s \in [0,u], \ \lvert\Delta \mathcal{Y}^z_s \rvert \in [\vp, 2\vp]\}.
 	\end{equation}
 	The above convergence holds thanks to Lemma \ref{lemme cv distance quantique naturelle} and we extend the definition of $u\mapsto \mathrm{d}_\mathrm{q} (\partial \mathbb{D}, D^{z}_u)$ into a right-continuous non-decreasing function on the interval $[0,\tau^z]$ thanks to the fact that the set of times $u\ge 0$ such that $\widehat{\gamma}_u\neq \partial$, i.e.\@ the set of times $u\ge 0$ of discovery of a loop, is a.s.\@ dense. The above equality will actually hold for all $u \in [0,\tau^z]$ since we will see shortly that $u\mapsto \mathrm{d}_\mathrm{q} (\partial \mathbb{D}, D^{z}_u)$ is continuous.

 	Let $\mathcal{L}^z$ be the loop surrounding $z$. We set 
 	$$
 	\mathrm{d}_\mathrm{q} (\partial \mathbb{D}, \mathcal{L}^z)\coloneqq \lim_{u \uparrow \tau^z}\mathrm{d}_\mathrm{q} (\partial \mathbb{D}, D^{z}_{u})=\lim_{\vp \to 0}  \vp \# \{ s \in [0,\tau^z], \ \lvert\Delta \mathcal{Y}^z_s \rvert \in [\vp, 2\vp]\},
 	$$
 	which is a limit in probability thanks to Lemma \ref{lemme cv distance quantique naturelle}.

 	We denote by $(Y^z_t)_{ 0\le t < \mathrm{d}_\mathrm{q} (\partial \mathbb{D}, \mathcal{L}^z)}$ the process $\mathcal{Y}^z$ reparametrized by the right-continuous inverse $(\sigma^z(t))_{t\ge 0}$ of $u \mapsto \mathrm{d}_\mathrm{q} (\partial \mathbb{D}, D^{z}_u) $. We have the following result from \cite{AHPS21} which describes the jumps of the process $Y^z$:
 	\begin{lemma}\label{lemme AHPS sauts positifs sauts negatifs}(consequence of Theorem 5.5 of \cite{AHPS21})
 		The process $Y^z$ has a càdlàg version. Each positive jump of $(Y^z(t))_{ 0\le t < \mathrm{d}_\mathrm{q} (\partial \mathbb{D}, \mathcal{L}^z)}$ at some time $t\ge 0$ corresponds to the quantum boundary length of the loop discovered at time $\sigma^z(t)$ by $(D^z_u)_{0 \le u < \tau^z}$ (computed using the restriction of $h^2$ to the domain encircled by the loop). Moreover, each negative jump of $(Y^z(t))_{ 0\le t < \mathrm{d}_\mathrm{q} (\partial \mathbb{D}, \mathcal{L}^z)}$ at some time $t\ge 0$ corresponds to the splitting at time $u= \sigma^z(t)$ of $D_{u-}^{z}$ into two components $D_u^{z}$ and $D_{u-}^{z} \setminus D_u^{z}$ and the absolute value of the size of the negative jump corresponds to the quantum boundary length of the ``cut out region'' $D_{u-}^{z} \setminus D_u^{z}$. Besides, the set of times $t\ge 0$ of positive jumps of $(Y^z(t))_{ 0\le t < \mathrm{d}_\mathrm{q} (\partial \mathbb{D}, \mathcal{L}^z)}$ is dense.
 	\end{lemma}

 	The above lemma could also be seen as a consequence of our convergence results of CLE explorations and actually we will not rely on it for the proof of the convergence. Note that the times of positive jumps of $(Y^z_t)_{0\le t< \mathrm{d}_\mathrm{q} (\partial \mathbb{D}, \mathcal{L}^z)}$ are the quantum distances from $\partial \mathbb{D}$ to the corresponding loops which are discovered at those times. Notice also that if $z,z' \in \Q^2 \cap \mathbb{D}$ are encircled by the same loop $\mathcal{L}$, then $\mathcal{Y}^z=\mathcal{Y}^{z'}$, $Y^z=Y^{z'}$ and we can set $Y^{\mathcal{L}}\coloneqq Y^z$.
 	
 	Thanks to Lemma \ref{lemme AHPS sauts positifs sauts negatifs}, we see that the quantum distances from the boundary to the loops discovered by $(D^z_u)_{0 \le u < \tau^z}$ are distinct, so that $u\mapsto \mathrm{d}_{\mathrm{q}}(\partial \mathbb{D}, D^z_u)$ is increasing. But by density of the set of times of positive jumps of $(Y^z(t))_{ 0\le t < \mathrm{d}_\mathrm{q} (\partial \mathbb{D}, \mathcal{L}^z)}$ one can see that $\sigma^z$ is increasing. Hence, $u\mapsto \mathrm{d}_{\mathrm{q}}(\partial \mathbb{D}, D^z_u)$ is also continuous. As a result, $(\mathcal{Y}^{z}_u)_{0\le u < \tau^{z}}$ is càdlàg, each positive jump {$(\mathcal{Y}^{z}_u)_{0\le u < \tau^{z}}$} corresponds to the discovery of a loop of the $\mathrm{CLE}_4$ and the size of the jump is exactly the quantum boundary length of the loop while each negative jump at some time $u$ corresponds to the splitting of $D_{u-}^{z}$ into two components $D_u^{z}$ and $D_{u-}^{z} \setminus D_u^{z}$ and the size of the negative jump corresponds to the quantum boundary length of $D_{u-}^{z} \setminus D_u^{z}$.
 }

{Now, let us introduce the exploration of the locally largest component in terms of quantum boundary length, which is no longer an exploration towards a fixed point.} Informally, at each discovery of a loop, it goes on exploring the exterior of the loop and at each splitting event, it explores the domain which has the largest quantum boundary length. Let us first define formally this locally largest component. Let $(z_j)_{j\ge 1}$ be an enumeration of $\mathbb{D} \cap \Q^2$.

The locally largest component $(C_u)_{u\ge 0}$ is defined as follows. Let
$$
\tau_1= \tau^{z_1} \wedge \inf \{u\ge 0; \ \mathcal{Y}^{z_1}_{u}<\mathcal{Y}^{z_1}_{u-}/2\}
$$
be the first time $u$ at which $z_1$ is encircled by a loop or $\mathcal{Y}^{z_1}_{u}<\mathcal{Y}^{z_1}_{u-}/2$. For all $u<\tau_1$, we set $C_u= D^{z_1}_u$. We then define by induction a sequence $(\tau_k)_{k\ge 1}$ and a subsequence $(z_{j_k})_{k\ge 1}$ by setting for all $k\ge 1$:
\begin{itemize}
	\item If $\tau_k = \tau^{z_{j_k}}$, then $\tau_k$ is the time at which a loop $\mathcal{L}$ encircles $z_{j_k}$. Let $j_{k+1}$ be the first integer $j>j_k$ for which $z_j\in D_{\tau_k-}^{z_{j_k}}$ and such that $z_j$ is not encircled by $\mathcal{L}$. 
	\item If $\tau_k = \inf \{u\ge 0; \ \mathcal{Y}^{z_{j_k}}_{u}<\mathcal{Y}^{z_{j_k}}_{u-}/2\}$, then $\tau_k $ is the first splitting time at which $z_{j_k}$ lies in the component with smallest quantum boundary length. Let $j_{k+1}$ be the first integer $j>j_k$ for which $z_j\in D_{\tau_k-}^{z_{j_k}} \setminus D_{\tau_k}^{z_{j_k}} $.
\end{itemize}
In both cases we set
$$
\tau_{k+1} = \tau^{z_{j_{k+1}}} \wedge  \inf \{u\ge 0; \ \mathcal{Y}^{z_{j_{k+1}}}_{u}<\mathcal{Y}^{z_{j_{k+1}}}_{u-}/2\}
\qquad
\text{and}
\qquad
\forall u \in [\tau_k, \tau_{k+1}{)}, \ C_u= D^{z_{j_{k+1}}}_u.
$$
{On the event $\lim_{k\to \infty}\tau_k<\infty$, for all $u\ge \lim_{k\to \infty}\tau_k$, by convention we set $C_u = \emptyset$.}
Notice that by Lemma 8 of \cite{WW13}, the process $(C_u)_{u\ge 0}$ does not depend on the choice of the enumeration $(z_j)_{j\ge 1}$. For all {$u \in [0,\lim_{k\to \infty}\tau_k)$}, we set $$\mathcal{Y}(u)\coloneqq \nu^2_{h^2_{\vert C_u}} (\partial C_u),$$
{and by convention $\mathcal{Y}(u)=0$ for all $u \ge \lim_{k\to \infty}\tau_k$ when $\lim_{k\to \infty}\tau_k<\infty$.}
By applying  \eqref{eq distance quantique target} to the $z_{j_k}$'s for all ${k\ge 1}$, {for all $u \in [0,\lim_{k\to \infty}\tau_k)$,} the quantum natural distance from $\partial \mathbb{D}$ to $C_u$ is characterized as the limit in probability
\begin{equation}\label{eq distance quantique}
\mathrm{d}_\mathrm{q} (\partial \mathbb{D}, C_u) = \lim_{\vp \to 0}  \vp \# \{ s \in [0,u], \ \lvert\Delta \mathcal{Y}(s) \rvert \in [\vp, 2\vp]\}.
\end{equation}
It is clear that $u \mapsto \mathrm{d}_\mathrm{q} (\partial \mathbb{D}, C_u)$ is  non-decreasing. Moreover, it is adapted with respect to the natural filtration of $\mathcal{Y}$. We denote by $\sigma$ the right-continuous inverse of $u \mapsto \mathrm{d}_\mathrm{q} (\partial \mathbb{D}, C_u)$ and for all $t\ge 0$ we set $Y_t= \mathcal{Y}({\sigma(t)})$. {By combining Theorem 5.5 of \cite{AHPS21} with Theorem 1.1 of \cite{AdS22}, one can directly see that $(Y_t)_{t\ge 0}$ has the same law as the self-similar Markov process $(X^{(-1)}(\pi t))_{t\ge 0}$ defined in Section \ref{sous-section resultats}. However, we will recover this result by transferring the results of \cite{MSW22} which hold for $\kappa<4$ by letting $\kappa \uparrow 4$ and then taking another limit.}

\section{The convergence $\kappa \uparrow 4$}\label{section kappa vers 4}
In this section, we prove Theorem \ref{prop de Markov disques quantiques} for the $\mathrm{SLE}^{\langle \mu \rangle}_4(-2)$ exploration. To this end, we approximate the $\mathrm{SLE}_4^{\langle \mu \rangle}(-2)$ explorations of the $\mathrm{CLE}_4$ decorated $2$-quantum disk by explorations of a $\mathrm{CLE}_\kappa$-decorated $\sqrt{\kappa}$-quantum disk as $\kappa \uparrow 4$, relying on the convergence results of \cite{Le21} and \cite{AHPS21} so as to transfer the analogous property of \cite{MSW22} {which describes the law of the quantum boundary lengths in the exploration of a CLE$_\kappa$-decorated $\sqrt{\kappa}$ quantum disk and states that, conditionally on the quantum boundary lengths, the encircled and cut out regions, as well as the unexplored region, are independent $\sqrt{\kappa}$-quantum disks}. But first, we recall the definition {of} the Carathéodory topology and the $\mathrm{SLE}_\kappa^\beta(\kappa-6)$ and $\mathrm{SLE}_4^{\langle \mu \rangle}(-2)$ explorations.

\subsection{Carath\'eodory topology}
We first provide some useful results on Carath\'eodory topology.
Let $z \in \C$. For all $n \ge 0$, let $U_n \subset \C$ be a simply connected domain. Let $U \subset \C$ be a simply connected domain containing $z$. When $z \in U_n$, let $f_n$ (resp. $f$) be the unique conformal map  from $\mathbb{D}$ to $U_n$ (resp. $U$) such that $f_n(0)=z$ (resp. $f(0)=z$) and $f'_n(0)>0$ (resp. $f'(0)>0$). We recall that $U_n$ converges to $U$ as $n\to \infty$ in the \textit{Carath\'eodory topology} if $z \in U_n$ for $n$ large enough and $f_n$ converges to $f$ uniformly on compact subsets of $\mathbb{D}$. Note that the convergence does not depend on the choice of $z\in U$ (see e.g.\@ Lemma 2.2 of \cite{BRY19}). The Carathéodory topology on the space of simply connected open domains of $\C$ containing $z$ is clearly metrizable {using the distance
\begin{equation}\label{eq distance caratheodory}
d_{\text{Carathéodory},z}(U_1,U_2) \coloneqq \sum_{j\ge 1} 2^{-j} \sup_{(1-1/j)\mathbb{D}} \vert f_1-f_2\vert,
\end{equation}
where $f_1,f_2$ are defined as explained above.}

An equivalent definition of Carathéodory convergence (see for instance Definition 2.1 of \cite{BRY19}), in fact the original one, states that $U_n$ converges to $U$ for the Carathéodory topology if and only if for every compact $K \subset U$, we have $K\subset U_n$ for all $n$ large enough and for every connected open set $V\ni z$, if $V\subset U_n$ for infinitely many $n$ then $V \subset U$. {The equivalence between the two definitions of the Carathéodory convergence comes from the Carathéodory kernel theorem (see e.g.\@ p.3 of \cite{BRY19} or Proposition 3.63 in \cite{Law05}).} From this equivalent definition, it is not hard to see that the space of simply connected open domains of $\C$ containing $z$ is separable.  Still, with this definition, notice that a sequence $U_n$ which converges to $U$ in the sense of Carath\'eodory may also converge to another simply connected open domain $V$ such that $V \cap U=\emptyset$. {For instance, let $U_n$ be the union of two disjoint non-empty open domains $U,V$ and an $\vp_n$-neighbourhood of a path $\Gamma$ connecting the domains $U$ and $V$ with $\vp_n \coloneqq 1/n$. Then one can see that $U_n\to U$ and $U_n \to V$ as $n\to \infty$ in the sense of Carathéodory.}

From this equivalent definition, one can see that for all compact $K \subset U$ the restriction of $f_n^{-1}$ to $K$ is well defined for $n$ large enough. Actually, as a consequence of Arzelà-Ascoli's theorem, the convergence in the sense of Carathéodory of $U_n$ to $U$ implies the uniform convergence of $f_n^{-1}$ towards $f^{-1}$ on every compact subset of $U$. See Proposition 2.1 of \cite{BRY19} for more details. 

\subsection{Background on $\mathrm{CLE}_\kappa$ explorations}\label{sous-section rappels CLE}

We recall here the explorations of $\mathrm{CLE}_\kappa$ for $\kappa \in (8/3,4]$ defined in \cite{S09}. We will not rely on the precise definition of the explorations and we refer to \cite{S09} for details. If $\beta \in [-1,1]$, we define the skew Bessel process $(X_u)_{u\ge 0}$ of dimension $\delta>0$ as the process such that $(|X_u|)_{u\ge 0}$ is a Bessel process of dimension $\delta$ and to each excursion of $(|X_u|)_{u\ge0}$ we assign an independent sign: the corresponding excursion of $(X_u)_{u\ge 0}$ is positive with probability $\frac{1+\beta}{2}$ and negative with probability $\frac{1-\beta}{2}$. When $\delta \neq 1$, we associate a process $(Y_u)_{u\ge 0}$ to the skew Bessel process $X$ which is characterized by Proposition 3.8 of \cite{S09}: {the process $(Y_u)_{u\ge 0}$ is the unique continuous process such that:
\begin{itemize}
	\item We have
$$
\frac{dY_u}{du} = \frac{1}{X_u} \qquad \text{on} \qquad \left\{u\ge 0; \ X_u \neq 0 \right\};
$$
\item The pair $(X,Y)$ is adapted with respect to the natural filtration of the Brownian motion appearing in the SDE which defines the Bessel process $X$;
\item The law of $(X,Y)$ is invariant under Brownian scaling;
\item If $T$ is a stopping time for $(X,Y)$ such that $X_T=0$ a.s.\@, then $(X_{T+u},Y_{T+u})_{u\ge 0}$ has the same law as $(X,Y)$.
\end{itemize}}
In the case $\delta=1$, the process $(Y_u)_{u\ge 0}$ is only defined when $\beta=0$ (and can be defined as the limit in law when $\delta \to 1$). Let $\mu \in \R$. One can obtain processes which satisfy the same properties as $(Y_u)_{u\ge 0}$ in Proposition 3.8 of \cite{S09} for $\delta=1$ by replacing $(Y_u)_{u\ge 0}$ by $(Y_u-\mu \ell^0_u)_{u\ge 0}$, where $(\ell^0_u)_{u\ge 0}$ is the local time at zero of $(X_u)_{u\ge 0}$. {The term $\mu \ell^0_u$ can be interpreted as ``a small drift each time $X$ is not making an excursion''. Actually, as explained in Section 3.2 of \cite{S09}, if one reparametrizes $Y$ with the right-continuous inverse of the local time $\ell^0$, then one obtains a $(2-\delta)$-stable Lévy process. In the case $\delta=1$, when we add this drift it remains a $1$-stable Lévy process.}

Let $\delta = 3-8/\kappa$. If $\delta \neq 1$, then let $\beta \in [-1,1]$, otherwise let $\mu \in \R$. Let $(X_u)_{u\ge 0}$ and $(Y_u)_{u\ge 0}$ be the two processes defined above associated to $\delta,\beta$ when $\delta \neq 1$ and to $\delta,\mu$ when $\delta=1$. Let $(O_u)_{u\ge 0}$ and $(W_u)_{u\ge 0}$ be the processes given by $O_0=W_0=0$ and for all $u\ge 0$,
$$
O_u = -\frac{2}{\sqrt{\kappa}} Y_u \qquad \text{and} \qquad W_u=O_u + \sqrt{\kappa}X_u.
$$
Then the \textit{chordal $\mathrm{SLE}^\beta_\kappa(\kappa-6)$} for $\kappa \in (8/3,4)$ and the \textit{chordal $\mathrm{SLE}^{\left\langle \mu\right\rangle}_4(-2)$} from $0$ to $\infty$ in $\mathbb{H}$ are defined as the growing family of closed sets $(K_u)_{u\ge 0}$ determined by the ODE
$$
g_0(z)=z \qquad \text{and} \qquad \partial_u g_u(z) =\frac{2}{g_u(z)-W_u},
$$
in the sense that the domain of definition of $g_u$ is $\mathbb{H} \setminus K_u$. 
{These processes can be seen as special cases of SLE processes with force point, with parameters $\kappa$ and $\rho =\kappa-6$. Such processes first appeared in \cite{LSW03}. See e.g. Appendix A.3 of \cite{BP} for an overview. As explained in Remark A.6 of \cite{BP}, the Bessel process $X_u$ describes the distance between the driving function $W_u$ and the marked point $O_u$. However, given that in our case $\rho \le -2$ and since the marked point $O_u$ is sometimes infinitesimally close to the driving function $W_u$, the above definition is not so simple (see Remark A.7 of \cite{BP}).}

\begin{figure}[h]
	\centering
	\includegraphics[scale=0.8]{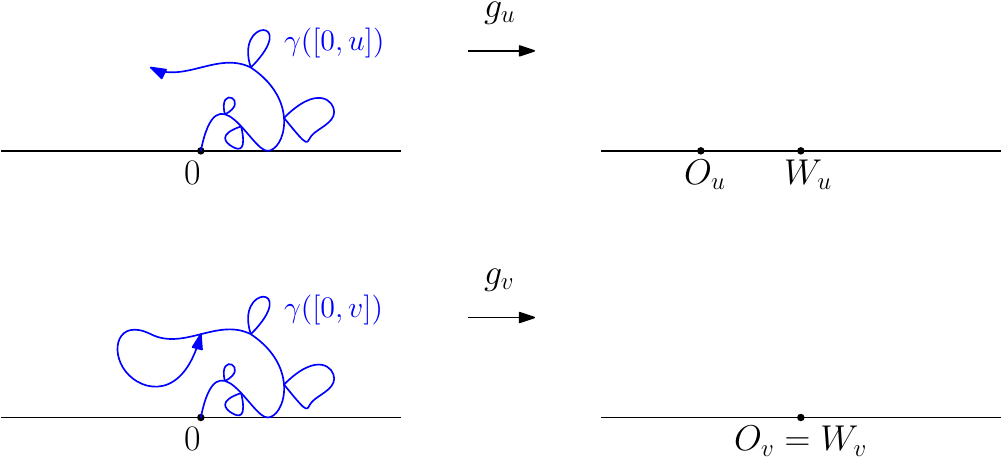}
	\caption{{Illustration of a chordal $\mathrm{SLE}_\kappa^\beta(\kappa-6)$ or $\mathrm{SLE}^{\langle \mu \rangle}_4(-2)$ from $0$ to $\infty$ in $\mathbb{H}$. Above: at a time $u$ where $\gamma$ is drawing a loop, i.e.\@ a time such that $O_u\neq W_u$. Below: at the time $v$ where $\gamma$ closes the loop, so that $O_v=W_v$.}}
	\label{SLEkappakappamoinssix}
\end{figure}

It was shown in \cite{MSW17} (in Theorem 7.4 for $\kappa \in (8/3,4)$ and in Proposition 5.3 for $\kappa =4$) that $\mathbb{H}\setminus K_u$ is the unbounded connected component of $\mathbb{H}\setminus\gamma([0,u])$ for some continuous path $\gamma$. On each excursion of $(X_u)_{u\ge 0}$, the path $\gamma$ draws a simple loop {(see Figure \ref{SLEkappakappamoinssix})}. We also recall from Proposition 3.10 of \cite{S09} that $(K_u)_{u\ge 0}$ satisfies a ``renewal property'' {and a ``conformal Markov property'' which are described as follows.
\begin{proposition}\label{prop renewal et Markov conforme}(from Proposition 3.10 of \cite{S09})
	The process $(K_u)_{u\ge 0}$ satisfies the following properties:
	\begin{itemize}
		\item \textbf{Renewal property.} Conditionally on $W_u$ up to a stopping time $T$ (for the natural filtration of $(X,Y)$) for which $O_T=W_T$ almost surely, the process $(W_{T+u}-W_T)_{u\ge 0}$ has the same law as $(W_u)_{u\ge 0}$.
		\item \textbf{Conformal Markov property.} Conditionally on $W_u$ up to some stopping time $T$ for which a.s.\@ $O_T \neq W_T$, the growing family given by the closure of $g_T(K_{T+u})$ up to the first time $u\ge 0$ such that $W_{T+u} = O_{T+u}$ has the same law (up to a time-change) as a chordal $\mathrm{SLE}_\kappa$ in $\mathbb{H}$ from $W_T$ to $O_T$.
	\end{itemize}
\end{proposition} 
The renewal property can also be rephrased as follows: if $T$ is a stopping time at which a.s.\@ the exploration is not drawing a loop (i.e.\@ such that $X_T=0$ a.s.), then conditionally on $(K_u)_{0 \le u \le T}$, the process $(g_T(K_{T+u}))_{u\ge 0}$ has the same law as $(K_u)_{u\ge 0}$.}

A chordal $\mathrm{SLE}^\beta_\kappa(\kappa-6)$ (or $\mathrm{SLE}^{\left\langle \mu\right\rangle}_4 (-2)$ when $\kappa=4$) from $a$ to $b$ which are on the boundary of some simply connected domain $D \subset \C$ is just the image of the above $\mathrm{SLE}^\beta_\kappa(\kappa-6)$ (or $\mathrm{SLE}^{\left\langle \mu\right\rangle}_4 (-2)$ when $\kappa=4$) by a conformal map sending $0$ to $a$ and $\infty$ to $b$. One can define the trunk $\eta$ of $\gamma$, a chordal $\mathrm{SLE}^\beta_\kappa(\kappa-6)$ (or  $\mathrm{SLE}^{\left\langle \mu\right\rangle}_4 (-2)$ when $\kappa =4$) 
as the path $\gamma$ where we removed the loops. In Theorem 7.4 of \cite{MSW17}, the trunk is shown to be an $\mathrm{SLE}_{\kappa'}$-type curve for $\kappa'=16/\kappa$. The path $\gamma$ can be decomposed into the trunk $\eta$ and the loops it traces along each excursion of $X$, see \cite{MSW17} for details {and the left-hand side of Figure \ref{tronc et explo unif} for an illustration.}

One can also define a \textit{radial} version of the $\mathrm{SLE}^{\beta}_\kappa (\kappa-6)$ (or $\mathrm{SLE}^{\langle \mu \rangle}_4 (-2)$ for $\kappa=4$) from $-i$ to $0$ in the unit disk $\mathbb{D}$, which is characterized by Proposition 3.13 of \cite{S09} using the chordal $\mathrm{SLE}^{\beta}_\kappa (\kappa-6)$ (or $\mathrm{SLE}^{\langle \mu \rangle}_4 (-2)$ for $\kappa=4$). As explained in \cite[p.~11]{WW13} one can define this radial version relying only on chordal $\mathrm{SLE}$'s as follows. Consider $\gamma_1$ a chordal $\mathrm{SLE}^{\beta}_\kappa (\kappa-6)$ (or $\mathrm{SLE}^{\langle \mu \rangle}_4 (-2)$ for $\kappa=4$) from $-i$ to $i$, up to the first time $u_1$ at which it finishes to draw a loop intersecting the circle of radius $1/2$ centred at $-i$. Then, let $F_1$ be the conformal transformation from $D_{u_1}^0$ to $\mathbb{D}$ that keeps $0$ fixed and maps $\gamma_1(u_1)$ to $-i$ and we repeat the same procedure again: growing $\gamma_2$ another independent chordal $\mathrm{SLE}^{\beta}_\kappa (\kappa-6)$ (or $\mathrm{SLE}^{\langle \mu \rangle}_4 (-2)$ for $\kappa=4$) from $-i$ to $i$ until it draws a loop intersecting the circle of radius $1/2$ centred at $-i$, looking at its preimage by $F_1$ etc. We stop when we draw a loop surrounding $0$. The radial $\mathrm{SLE}^{\beta}_\kappa (\kappa-6)$ (or $\mathrm{SLE}^{\langle \mu \rangle}_4 (-2)$ for $\kappa=4$) targeted at $z\in \mathbb{D}$ can be defined in a similar way.

{
	Two radial $\mathrm{SLE}^\beta_\kappa (\kappa-6)$ (or $\mathrm{SLE}^{\left\langle \mu\right\rangle}_4 (-2)$ when $\kappa=4$) are related by the following target-invariance property.
\begin{proposition}\label{prop target invariance}(Proposition 3.14 of \cite{S09})
	For all $z_1,z_2 \in \mathbb{D}$, one can couple two radial $\mathrm{SLE}^\beta_\kappa (\kappa-6)$ (or $\mathrm{SLE}^{\left\langle \mu\right\rangle}_4 (-2)$ when $\kappa=4$), $\gamma_1$ starting from $-i$ and targeted at $z_1$ and $\gamma_2$ from $-i$ to some $z_2$ so that they coincide up to some time-change until the time that $z_1$ and $z_2$ fail to lie in the same connected component of $\mathbb{D} \setminus \gamma_1([0,u])$. 
\end{proposition}
}
By performing this coupling with targets in a countable dense subset of $\mathbb{D}$, one can thus build the branching $\mathrm{SLE}^\beta_\kappa(\kappa-6)$ (or $\mathrm{SLE}^{\left\langle \mu\right\rangle}_4 (-2)$ when $\kappa=4$) starting at $-i$ in $\mathbb{D}$, hence exploring all the loops of a $\mathrm{CLE}_\kappa$. It is also called the exploration tree. The radial and chordal $\mathrm{SLE}$'s can be seen as ``branches'' of this exploration tree. If we continue the radial explorations after the time they draw a loop surrounding their targets, then they explore the nested $\mathrm{CLE}_\kappa$. Except in Sections \ref{section mu vers l'infini} and \ref{sous-section cartes deco}, we will always consider the non-nested version of the $\mathrm{CLE}_\kappa$ and it will simply be referred as the $\mathrm{CLE}_\kappa$.
\subsection{Carathéodory convergence of the chordal exploration}
We now describe an approximation result of the $\mathrm{SLE}^{\langle \mu \rangle}_4(-2)$ exploration in the sense of Carathéodory based on the results of \cite{Le21}. Let $\mu \in \R$. If $\kappa \in (8/3,4)$, let $(\gamma^\kappa(u))_{u\ge 0}$ be a chordal $\mathrm{SLE}^{\mu((4/\kappa)-1)}_{\kappa}(\kappa-6)$ in $\mathbb{D}$ from $-i$ to $i$. For all $u\ge 0$, if $X(u)=0$, then let $\upsilon^\kappa(u)=u$ and otherwise let $\upsilon^\kappa(u)$ be the time of the beginning of the excursion $X$ which is being drawn at time $u$. Let $(\gamma^4(u))_{u\ge 0}$ be a chordal $\mathrm{SLE}^{\langle \mu \rangle}_{4}(-2)$ in $\mathbb{D}$ from $-i$ to $i$. Let $\upsilon^4(u)$ for $u\ge 0$ be defined as before using the corresponding driving function. For all $\kappa \in (8/3,4]$, for all $u\ge 0$, for all $z \in \mathbb{D} $, let $D^{z,\kappa}_u$ be the connected component of $\mathbb{D}\setminus \gamma^\kappa([0,\upsilon^\kappa(u)])$ containing $z$.  
Note that we choose to restrict the definition of the $D^{z,\kappa}_u$'s to the times $\upsilon^\kappa(u)$ since the relation with quantum disks that we will recall in the next subsection is only valid for these times. Then a consequence of the results of \cite{Le21} is the following proposition.
\begin{proposition}\label{prop cv CLE}
We have the convergence in law
\begin{equation}\label{cvCLE}
	\left(D^{z,\kappa}_u\right)_{u\ge 0, z \in \mathbb{Q}^2 \cap \mathbb{D}}
	\enskip \mathop{\longrightarrow}\limits_{\kappa \uparrow 4}^{(\mathrm{d})} \enskip 
	(D^{z,4}_u)_{u\ge 0, z \in \mathbb{Q}^2 \cap \mathbb{D}} 
\end{equation}
in terms of finite-dimensional distributions for the Carathéodory topology.
\end{proposition}

{The rest of this subsection is devoted to the proof of the above proposition. The outline of the proof is as follows: Proposition 3.10 of \cite{Le21} gives the convergence of the driving functions of the exploration. Next, Proposition 4.47 of \cite{Law05} enables to deduce the convergence at the level of conformal maps. Finally, some technical work remains to deduce the convergence of the cut out and encircled regions in the desired topology.}

We fix $\mu \in \R$. For all $\kappa \in (8/3, 4]$, we denote by $(W^\kappa_u)_{u\ge 0}$, $(O^\kappa_u)_{u\ge 0}$, $(g^\kappa_u)_{u\ge 0}$ and $(K^\kappa_u)_{u\ge 0}$ the processes with which the $\mathrm{SLE}^{\mu((4/\kappa)-1)}_{\kappa}(\kappa-6)$ for $\kappa <4$ and $\mathrm{SLE}^{\langle \mu \rangle}_4(-2)$ for $\kappa=4$ is defined as described above. Let $\kappa_n$ be an increasing sequence tending to $4$. By Proposition 3.10 of \cite{Le21} we have the convergence of the driving functions:
\begin{equation}\label{eq cv pilote kappa}
(W^{\kappa_n},O^{\kappa_n})\mathop{\longrightarrow}\limits_{n \to \infty}^{(\mathrm{d})} (W^4, O^4)
\end{equation}
for the topology of uniform convergence on compact sets. By Skorokhod's representation theorem, we may suppose {in the rest of the subsection} that the above convergence is almost sure. In particular, by Proposition 4.47 of \cite{Law05}, for all $u_0>0$, for all $\vp>0$, we get a.s.\@ the convergence
\begin{equation}\label{eq cv g unif}
g^{\kappa_n} \mathop{\longrightarrow}\limits_{n\to \infty} g^4
\end{equation}
uniformly on the set $[0,u_0]\times \left\{z \in \mathbb{H}; \ d(z,K^4_{u_0})\ge \vp\right\}$. {As a result, we obtain the Carathéodory convergence of the unexplored region: a.s.\@ for all $u \ge 0$, {for any sequence of real numbers $u_n\ge0$ converging to $u$},
\begin{equation}\label{eq cv region inexploree}
	\mathbb{H} \setminus K^{\kappa_n}_{u_n}
	\mathop{\longrightarrow}\limits_{n \to \infty}
	\mathbb{H} \setminus K^4_u.
\end{equation}
} In particular, for all $u \ge 0$, we have a.s.
\begin{equation}\label{eq cv H moins K}
\mathbb{H} \setminus K^{\kappa_n}_{u}
\mathop{\longrightarrow}\limits_{n \to \infty}
\mathbb{H} \setminus K^4_u
\qquad
\text{and}
\qquad
\mathbb{H} \setminus K^{\kappa_n}_{\upsilon^{\kappa_n}(u)}
\mathop{\longrightarrow}\limits_{n \to \infty}
\mathbb{H} \setminus K^4_{\upsilon^4(u)}
.
\end{equation}

We then look into the domains encircled by loops. {The next lemma states the convergence of the domains encircled by the loops which are drawn by excursions of height larger than $\vp$.}
{
\begin{lemma}\label{lemme cv domaine entoure par une boucle}
	Let $\vp>0$ and $i\ge 1$. Let $\mathcal{B}^n_{i,\vp}$ (resp. $\mathcal{B}_{i,\vp}$) be the domain encircled by the loop which is drawn by the $i$-th excursion of $W^{\kappa_n}-O^{\kappa_n}$ (resp.\@ $W^4-O^4$) of height larger than $\vp$. Then 
	$$
	\mathcal{B}^n_{i,\vp} \mathop{\longrightarrow}\limits_{n\to \infty}^{\mathrm{a.s.}}
	\mathcal{B}_{i,\vp}
	$$
	in the sense of Carathéodory.
\end{lemma}}
{In order to prove the above lemma, let us introduce some notation.} Let $\vp>0$. For all $i\ge 1$, for all $\kappa \in (8/3,4]$, let $e_{i,\vp}^\kappa$ be the $i$-th excursion of $W^\kappa-O^\kappa$ of height larger than $\vp$ and let $S_{i,\vp}^\kappa<T_{i,\vp}^\kappa$ be its starting time and end time. First notice that by (\ref{eq cv pilote kappa}) we get readily the uniform convergence of the excursions $e_{i,\vp}^{\kappa_n}$'s to the $e_{i,\vp}^4$'s and the convergence of their endpoints $S_{i,\vp}^{\kappa_n},T_{i,\vp}^{\kappa_n}$ to $S^4_{i,\vp},T_{i,\vp}^4$ as $n\to \infty$. {Similarly, for all $\kappa \in (8/3,4]$, let $R^{\kappa}_{i,\vp}$ be the $i$-th time at which $W^{\kappa}-O^{\kappa}$ first reaches $\vp$ after hitting zero. Then the $R^{\kappa_n}_{i,\vp}$'s converge a.s.\@ to the $R^4_{i,\vp}$'s as $n \to \infty$. Let us prove the following convergences:}
\begin{lemma}\label{lemme auxiliaire cv domaine entoure par une boucle}
	We have the following convergences:
	\begin{itemize}
		\item \textbf{Convergence of the beginning of the loop.} For all $u \in (S^4_{i,\vp}, T^4_{i,\vp})$,
		\begin{equation}\label{eq cv Hausdorff debut}
			(g^{\kappa_n}_{S^{\kappa_n}_{i,\vp}}\circ\gamma^{\kappa_n})( [S^{\kappa_n}_{i,\vp}, u])
			\mathop{\longrightarrow}\limits_{n\to \infty}^{(\mathrm{a.s.})}
			(g^{4}_{S^{4}_{i,\vp}}\circ\gamma^{4})( [S^{4}_{i,\vp}, u])
		\end{equation}
		with respect to the Hausdorff distance. 
		\item \textbf{Convergence of the end of the loop.} for all $\delta>0$ small enough, 
		\begin{equation}\label{eq cv Hausdorff fin}
			(g^{\kappa_n}_{S^{\kappa_n}_{i,\vp}}\circ\gamma^{\kappa_n})([R^{\kappa_n}_{i,\vp}, T^{\kappa_n}_{i,\vp}])\cap 
			V_\delta
			\mathop{\longrightarrow}\limits_{n\to \infty}^{(\mathrm{a.s.})}
			(g^4_{S^4_{i,\vp}}\circ \gamma^4)( [R^4_{i,\vp}, T^4_{i,\vp}])
			\cap
			V_\delta
		\end{equation}
		for the Hausdorff distance, where $$V_\delta\coloneqq \left\{
		z \in \mathbb{H}; \ d(z, g^4_{S^4_{i,\vp}}(\gamma^4([S^4_{\vp,i}, R^{4}_{i,\vp}]))\cup \R)>\delta
		\right\}.
		$$
	\end{itemize} 
\end{lemma}
\begin{proof}
Let $u \in (S^{4}_{i,\vp}, T^4_{i,\vp}) $ and $ \delta >0$. Applying Proposition 4.47 of \cite{Law05} to 
$(W^{\kappa_n}_{(S^{\kappa_n}_{i,\vp}+u)\wedge T^{\kappa_n}_{i,\vp}})_{u\ge 0}$ which converges uniformly to $(W^{4}_{(S^{4}_{i,\vp}+u)\wedge T^{4}_{i,\vp}})_{u\ge 0}$ as $n \to \infty$, we know that for all $n$ large enough, 
\begin{equation}\label{eq cv g excursion}
\left((v,z)\mapsto g^{\kappa_n}_{S^{\kappa_n}_{i,\vp} + v}((g^{\kappa_n}_{S^{\kappa_n}_{i,\vp}})^{-1}(z)) \right)\mathop{\longrightarrow}\limits_{n \to \infty}
\left(
(v,z) \mapsto g^4_{S^4_{i,\vp}+v}((g^{4}_{S^{4}_{i,\vp}})^{-1}(z))
\right)
\end{equation}
uniformly on $[0,u] \times \left\{z \in \mathbb{H}; \ d(z,g^4_{S^{4}_{i,\vp}}(K^4_u)) \ge \delta \right\}$. In particular, for $n$ large enough,
$$
 \left\{z \in \mathbb{H}  ; \  d(z,g^4_{S^{4}_{i,\vp}}(K^4_u)) \ge \delta \right\}\subset \mathbb{H} \setminus g^{\kappa_n}_{S^{\kappa_n}_{i,\vp}}(K^{\kappa_n}_u).
$$ 
In other words, for all $n$ large enough, the curve $(g^{\kappa_n}_{S^{\kappa_n}_{i,\vp}}\circ\gamma^{\kappa_n})( [S^{\kappa_n}_{i,\vp}, u])$ is included in the $\delta$-neighbourhood of $(g^{4}_{S^{4}_{i,\vp}}\circ\gamma^{4})( [S^{4}_{i,\vp}, u])$. Besides, let $U$ be the open ball centred at $g^4_{S^4_{i,\vp}}(\gamma^4(u))$ of radius $\delta$. Then, by choosing $\delta$ small enough, $U$ is an open connected subset of $\mathbb{H}$  such that $U\cap (\mathbb{H}\setminus g^4_{S^4_{i,\vp}}(K^4_u)) \neq\emptyset$. By the Carathéodory convergence of $\mathbb{H} \setminus g^{\kappa_n}_{S^{\kappa_n}_{i,\vp}}(K^{\kappa_n}_u)$ towards $\mathbb{H}\setminus g^4_{S^4_{i,\vp}}(K^4_u)$ (which stems from (\ref{eq cv g excursion})), we cannot have $U\subset \mathbb{H} \setminus g^{\kappa_n}_{S^{\kappa_n}_{i,\vp}}(K^{\kappa_n}_u)$ infinitely often, otherwise we would get that $\gamma^4(u) \in \mathbb{H} \setminus K^4_u$, which is not the case. Thus, for all $n$ large enough, $U \cap g^{\kappa_n}_{S^{\kappa_n}_{i,\vp}}(K^{\kappa_n}_u)\neq \emptyset$. By choosing $\delta$ small enough, by connectedness of $(g^{\kappa_n}_{S^{\kappa_n}_{i,\vp}}\circ\gamma^{\kappa_n})( [S^{\kappa_n}_{i,\vp}, u])$, we obtain that for $n$ large enough, $(g^{4}_{S^{4}_{i,\vp}}\circ\gamma^{4})( [S^{4}_{i,\vp}, u])$ is in the $\delta$-neighbourhood of $(g^{\kappa_n}_{S^{\kappa_n}_{i,\vp}}\circ\gamma^{\kappa_n})( [S^{\kappa_n}_{i,\vp}, u])$. We have thus proven \eqref{eq cv Hausdorff debut} that for all $u\in(S^{4}_{i,\vp}, T^4_{i,\vp}) $.

{Next, let us prove \eqref{eq cv Hausdorff fin}.} Recall that for all $\kappa \in (8/3,4]$, $R^{\kappa}_{i,\vp}$ is the $i$-th time at which $W^{\kappa}-O^{\kappa}$ first reaches $\vp$ after hitting zero. Then $R^{\kappa_n}_{i,\vp}$ converges to $R^4_{i,\vp}$ as $n \to \infty$. 
Moreover, for all $n\ge 0$, by the conformal Markov property {recalled in Proposition \ref{prop renewal et Markov conforme}} the curve $g^{\kappa_n}_{R^{\kappa_n}_{i,\vp}}\circ\gamma^{\kappa_n}_{\vert [R^{\kappa_n}_{i,\vp}, T^{\kappa_n}_{i,\vp}]}$ has the same law as a chordal $\mathrm{SLE}(\kappa_n)$ in $\mathbb{H}$ from $g^{\kappa_n}_{R^{\kappa_n}_{i,\vp}}(\gamma^{\kappa_n}(R^{\kappa_n}_{i,\vp}))=W^{\kappa_n}_{R^{\kappa_n}_{i,\vp}}$ to $g^{\kappa_n}_{R^{\kappa_n}_{i,\vp}}(\gamma^{\kappa_n}(T^{\kappa_n}_{i,\vp}))=O^{\kappa_n}_{R^{\kappa_n}_{i,\vp}}$. Furthermore, by taking Möbius transformations sending one of these two endpoints to $\infty$, and by applying Theorem 1.10 of \cite{KS17} (which states that the chordal $\mathrm{SLE}(\kappa_n)$ from $0$ to $\infty$ converges uniformly on compact sets to the $\mathrm{SLE}(4)$ as $n\to \infty$), we obtain the convergence
\begin{equation}\label{eq cv Hausdorff SLE}
(g^{\kappa_n}_{R^{\kappa_n}_{i,\vp}}\circ\gamma^{\kappa_n})([R^{\kappa_n}_{i,\vp}, T^{\kappa_n}_{i,\vp}])
\mathop{\longrightarrow}\limits_{n\to \infty} 
(g^4_{R^4_{i,\vp}}\circ \gamma^4)( [R^4_{i,\vp}, T^4_{i,\vp}])
\end{equation}
for the Hausdorff distance. Besides, by $(\ref{eq cv g excursion})$ and by the convergence of $R^{\kappa_n}_{i,\vp}$ to $R^4_{i,\vp}$, we know that
$$
g^{\kappa_n}_{R^{\kappa_n}_{i,\vp}} \circ (g^{\kappa_n}_{S^{\kappa_n}_{i,\vp}})^{-1}
\mathop{\longrightarrow}\limits_{n\to \infty}
g^{4}_{R^{4}_{i,\vp}} \circ (g^{4}_{S^{4}_{i,\vp}})^{-1}
$$
uniformly on compact subsets of $\mathbb{H} \setminus (g^{4}_{S^{4}_{i,\vp}}\circ\gamma^{4})( [S^{4}_{i,\vp}, R^4_{i,\vp}])$. Since the limiting function is a conformal map, it is easy to see that $g^{\kappa_n}_{S^{\kappa_n}_{i,\vp}} \circ (g^{\kappa_n}_{R^{\kappa_n}_{i,\vp}})^{-1}$ converges pointwise to $g^{4}_{S^{4}_{i,\vp}} \circ (g^{4}_{R^{4}_{i,\vp}})^{-1}$. By Montel's theorem, we deduce that
$$
g^{\kappa_n}_{S^{\kappa_n}_{i,\vp}} \circ (g^{\kappa_n}_{R^{\kappa_n}_{i,\vp}})^{-1}
\mathop{\longrightarrow}\limits_{n\to \infty}
g^{4}_{S^{4}_{i,\vp}} \circ (g^{4}_{R^{4}_{i,\vp}})^{-1}
$$
uniformly on compact subsets of $\mathbb{H}$. Combining this with (\ref{eq cv Hausdorff SLE}), we get \eqref{eq cv Hausdorff fin}.\end{proof}

We are now in position to prove the Carathéodory convergence of the domain encircled by the loop{, i.e.\@ to prove Lemma \ref{lemme cv domaine entoure par une boucle}.}
\begin{proof}[Proof of Lemma \ref{lemme cv domaine entoure par une boucle}]
For all $n \in \N$, let $\mathcal{L}_n= g^{\kappa_n}_{S^{\kappa_n}_{i,\vp}}(\gamma^{\kappa_n}([S^{\kappa_n}_{\vp, i}, T^{\kappa_n}_{\vp,i}]))$ be the loop drawn by $u \mapsto W^{\kappa_n}_{(S^{\kappa_n}_{i,\vp} + u)\wedge T^{\kappa_n}_{i,\vp}}$ and let $\mathcal{L} = g^4_{S^4_{i,\vp}}(\gamma^4([S^4_{\vp, i}, T^4_{\vp,i}]))$ be the loop drawn by $u \mapsto W^4_{(S^4_{i,\vp} + u)\wedge T^4_{i,\vp}}$. By the convergence of $g^{\kappa_n}$ and $S_{i,\vp}^{\kappa_n}$ as $n\to \infty$, to prove the Carathéodory convergence of the open domain encircled by the loop $\gamma^{\kappa_n}([S^{\kappa_n}_{\vp, i}, T^{\kappa_n}_{\vp,i}])$ towards the domain encircled by $\gamma^4([S^4_{\vp, i}, T^4_{\vp,i}])$, it suffices to prove the Carathéodory convergence of the domain $\mathcal{B}_n$ encircled by $\mathcal{L}_n$ towards the domain $\mathcal{B}$ encircled by $\mathcal{L}$. 

\begin{figure}[h]
   \centering
   \includegraphics[scale=0.8]{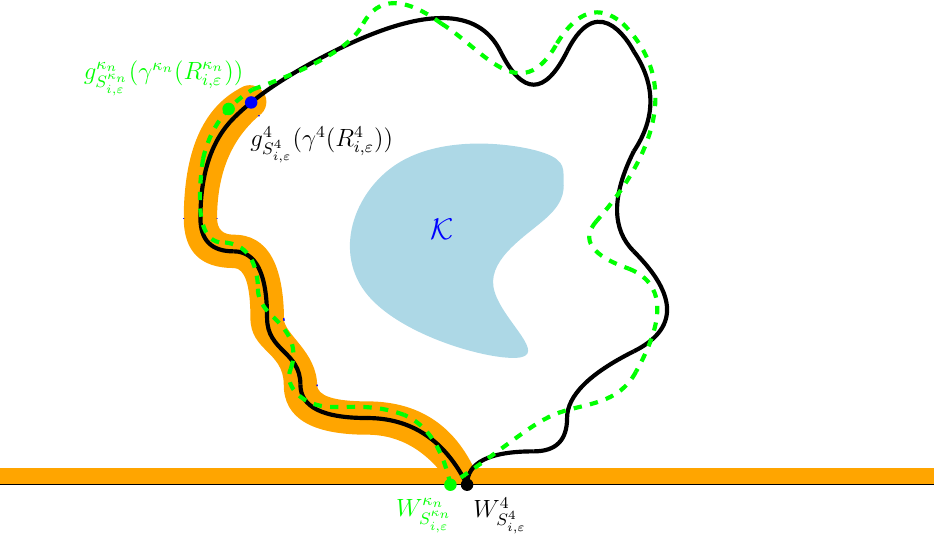}
   \caption{Illustration of the proof of the Carathéodory convergence of the domain encircled by $\mathcal{L}_n= g^{\kappa_n}_{S^{\kappa_n}_{i,\vp}}(\gamma^{\kappa_n}([S^{\kappa_n}_{\vp, i}, T^{\kappa_n}_{\vp,i}]))$, in green dashes, towards the domain encircled by $\mathcal{L} = g^4_{S^4_{i,\vp}}(\gamma^4([S^4_{\vp, i}, T^4_{\vp,i}]))$, in black. The complement of $V_\delta$ is in orange. A compact $\mathcal{K}$ encircled by $\mathcal{L}$ is drawn in blue.}
   \label{image preuve cv boucle}
\end{figure}

Let $\mathcal{K}$ be a compact subset of the open domain $\mathcal{B}$. Then we may take $\delta>0$ small enough such that $\mathcal{K}\subset V_\delta$, so that by (\ref{eq cv Hausdorff fin}), we obtain that $\mathcal{K}\cap(g^{\kappa_n}_{S^{\kappa_n}_{i,\vp}}\circ\gamma^{\kappa_n})([R^{\kappa_n}_{i,\vp}, T^{\kappa_n}_{i,\vp}]) = \emptyset$ for $n$ large enough. 
Moreover, by (\ref{eq cv Hausdorff debut}), for $u>0$ small enough, for all $n$ large enough, we have $(g^{\kappa_n}_{S^{\kappa_n}_{i,\vp}}\circ\gamma^{\kappa_n})([S^{\kappa_n}_{i,\vp}, R^{\kappa_n}_{i,\vp}+u]) \subset V_\delta$. Therefore, by the convergence of $R^{\kappa_n}_{i,\vp}$ to $R^4_{i,\vp}$ we deduce that for all $n$ large enough, $\mathcal{K}\cap \mathcal{L}_n = \emptyset$. One can actually see that $\mathcal{K}$ is encircled by $\mathcal{L}_n$ for $n$ large enough using that the endpoints of the path $(g^{\kappa_n}_{S^{\kappa_n}_{i,\vp}}\circ\gamma^{\kappa_n})([R^{\kappa_n}_{i,\vp}, T^{\kappa_n}_{i,\vp}])$ are in $\mathbb{H} \setminus V_\delta$ and applying (\ref{eq cv Hausdorff fin}) once more (see Figure \ref{image preuve cv boucle}). 

Let $z \in \mathcal{B}$. To end the proof of the Carathéodory convergence of $\mathcal{B}_n$ towards $\mathcal{B}$, it suffices to show that if $U$ is a connected open set containing $z$ such that $U\subset \mathcal{B}_n$ infinitely often, then $U\subset \mathcal{B}$. Since $U$ and $\mathcal{B}$ are open it is enough to see that $U \subset \overline{\mathcal{B}}$. Assume by contradiction that there exists $w\in U\setminus \overline{\mathcal{B}}$. Take $\delta> 0$ small enough so that $B(w, \delta) \subset U\setminus \overline{\mathcal{B}}$. In particular, $w \in V_{\delta}$. As a result, due to the Hausdorff convergences (\ref{eq cv Hausdorff debut}) and (\ref{eq cv Hausdorff fin}), the point $w$ is not encircled by $\mathcal{L}_n$ for $n$ large enough, hence a contradiction. We have thus proven that
$$
\mathcal{B}_n \mathop{\longrightarrow}\limits_{n \to \infty} \mathcal{B}
$$
in the sense of Carathéodory. {This ends the proof of the lemma.}
\end{proof}

\begin{proof}[Proof of Proposition \ref{prop cv CLE}]
{By \eqref{eq cv H moins K}, we have the convergence of the unexplored region. Moreover, by Lemma \ref{lemme cv domaine entoure par une boucle}, we have} the convergence of the domain encircled by the loop $\gamma^{\kappa_n}([S^{\kappa_n}_{\vp, i}, T^{\kappa_n}_{\vp,i}])$ for $i \ge 1$ towards the domain encircled by $\gamma^4([S^4_{\vp, i}, T^4_{\vp,i}])$. By letting $\vp \to 0$, we get the Carathéodory convergence of all the domains encircled by loops.

Finally, we focus on the domains which are cut out by the exploration. Let $u \ge 0$. Let $x \in \R$. For all $n$, let $\zeta^n_x$ be the first time that $x$ is swallowed by $(g^{\kappa_n}_u({K}^{\kappa_n}_{v+u}))_{v\ge 0}$ and let $\zeta_x$ be the first time that $x$ is swallowed by $(g^4_u(K^4_{u+v}))_{v\ge 0}$. For all $n\ge 1$, let $C^n_x$ be the interior of $g^{\kappa_n}_{u}({K}^{\kappa_n}_{u+\zeta^n_x})\setminus g^{\kappa_n}_u({K}^{\kappa_n}_{u+\zeta^n_x -})$ and let $C_x$ be the interior of $g^{4}_{u}({K}^{4}_{u+\zeta_x})\setminus g^{4}_u({K}^{4}_{u+\zeta_x -})$. By the proof of Theorem 1.2 of \cite{Le21} and by another application of Skorokhod's representation theorem, we know that $\zeta_x$ is the almost sure limit as $n\to \infty$ of $\zeta^n_x$. To be more precise, the proof of Theorem 1.2 of \cite{Le21} gives the convergence of $\zeta^n_1$, while the convergence of $\zeta^n_x$ follows by scale-invariance and {reflection-invariance.} 
As a consequence, by considering chordal $\mathrm{SLE}_{\kappa_n}^{\mu((4/\kappa_n)-1)}(\kappa_n-6)$'s and an $\mathrm{SLE}_4^{\langle \mu \rangle}(-2)$ targeted at $x$ and coupled with the chordal $\mathrm{SLE}_{\kappa_n}^{\mu((4/\kappa_n)-1)}(\kappa_n-6)$'s and $\mathrm{SLE}_4^{\langle \mu \rangle}(-2)$ targeted at $\infty$ until $x$ is separated from $\infty$  and by applying the convergence (\ref{eq cv g unif}), we obtain the a.s.\@ convergence
$$
C^n_x \mathop{\longrightarrow}\limits_{n \to \infty} C_x
$$ 
in the sense of Carathéodory. By (\ref{eq cv H moins K}), we deduce the convergence of the cut out domain
$$
(g_u^{\kappa_n})^{-1}(C^n_x)
\mathop{\longrightarrow}\limits_{n\to \infty}
(g_u^{4})^{-1}(C_x)
$$
in the sense of Carathéodory.
To conclude we just need to check that any cut out domain can be written in the form $(g_u^{4})^{-1}(C_x)$ for some $u\ge 0$ and $x \in \R$. This is straightforward by taking $u\ge 0$ such that the frontier of the image by $g_u^4$ of the cut out domain has a non-trivial intersection with $\R$.
\end{proof}

\subsection{ The $\sqrt{\kappa}$-quantum disks and $\mathrm{CLE}_\kappa$}

Here, we describe the results of \cite{MSW22} that we want to transfer to $\kappa =4$. Recall Subsection \ref{sous-section disques quantiques}. Let $\kappa \in (8/3,4]$. Let $(\mathbb{D}, h^{\sqrt{\kappa}})$ be a parametrization of a unit boundary $\sqrt{\kappa}$-quantum disk. By Lemma 4.5 of \cite{AHPS21}, we know that
\begin{equation}\label{cv disque quantique}
	(h^{\sqrt{\kappa}},\mu_{h^{\sqrt{\kappa}}}^{\sqrt{\kappa}},\nu_{h^{\sqrt{\kappa}}}^{\sqrt{\kappa}}) \mathop{\longrightarrow}\limits_{\kappa \uparrow 4} (h^2, \mu^2_{h^2},\nu^2_{h^2}),
\end{equation}
where the convergence of the first coordinate holds in the space $H^{-1}_\mathrm{loc}(\mathbb{D})$ while the convergences of the other coordinates hold with respect to the weak topology of measures. In what follows, when we speak of the convergence of some quantum disks, it always means the convergence of the GFF and of the quantum area and boundary length measures associated to the parametrization by the unit disk as in the above equation.

When we run an independent chordal $\mathrm{SLE}_\kappa^\beta(\kappa-6)$ (or $\mathrm{SLE}_4^{\langle \mu \rangle}(-2)$ for $\kappa=4$) $\gamma$ on a $\sqrt{\kappa}$-quantum disk $(\mathbb{D},h^{\sqrt{\kappa}})$, 
the trunk of $\gamma$ can be parametrized by its quantum length since it is an $\mathrm{SLE}_{\kappa'}$-type process by Theorem 7.4 of \cite{MSW17} and Remark 2.3 of \cite{MSW22}. Notice that this quantum length is defined up to some multiplicative constant. Besides, we can define the quantum boundary length of $\mathrm{SLE}_\kappa$ type curves, such as the loops drawn by $\gamma$ or the boundary of the domains cut out by $\gamma$, via a quantum boundary length measure induced by $h^{\sqrt{\kappa}}$, using the quantum zipper properties first discovered in \cite{S16}. The quantum boundary length measure on such a domain $D$ is equal to $\nu^{\sqrt{\kappa}}_{h^{\sqrt{\kappa}}_{\vert D}}$. We will see in the next proposition that they correspond exactly to the boundary length measures of some quantum disks.

We then state a result of \cite{MSW22} for $\kappa \in (8/3,4)$, which gives the law of the quantum boundary length of the unexplored region and a Markov property in terms of quantum disks. Let $L_0$ (resp. $R_0$) be the quantum boundary length along the clockwise (resp. counterclockwise) segment from $-i$ to $i$ in $\partial \mathbb{D}$. Let $\beta \in (-1,1)$. Let $\gamma$ be a chordal $\mathrm{SLE}_\kappa^\beta(\kappa-6)$ from $-i$ to $i$ taken independently from the quantum disk $(\mathbb{D}, h^{\sqrt{\kappa}})$. We parametrize its trunk $\eta$ by its quantum length. Let $T^{\kappa,\beta}$ be the total quantum length of the trunk. For each $t$ (corresponding here to the quantum length of the trunk), let $E_t$ be the event that the trunk has not yet attained $i$ before having quantum length $t$, i.e. $E_t= \{T^{\kappa, \beta } >t\}$. For all $t<T^{\kappa, \beta}$, we set $L^{\kappa,\beta}_t$ and $R^{\kappa,\beta}_t$ to be the left and right quantum boundary lengths of the unexplored region $\mathcal{D}^\kappa_t$ ($\mathcal{D}^\kappa_t$ is the connected component of $\mathbb{D} \setminus \gamma([0,\tau_t])$ containing $i$ in its closure, where $\tau_t$ is the time for $\gamma$ at which the quantum length of the trunk $\eta$ reaches $t$). More precisely, $L^{\kappa,\beta}_t$ is the quantum boundary length of the arc of $\partial\mathcal{D}^\kappa_t$  going from the tip of the trunk to $i$ in the clockwise direction and $R^{\kappa,\beta}_t$ is the quantum boundary length of the arc of $\partial\mathcal{D}^\kappa_t$  going from the tip of the trunk to $i$ in the anti-clockwise direction. Note that for all $t<T^{\kappa, \beta}$, we have $L^{\kappa,\beta}_t>0$ and $R^{\kappa,\beta}_t>0$. For all $t\ge T^{\kappa, \beta}$ we set $L^{\kappa,\beta}_t=R^{\kappa,\beta}_t=0$.  
If we write $S^{\kappa,\beta}_t= L^{\kappa,\beta}_t+R^{\kappa,\beta}_t$ for all $t\ge0$, then $S^{\kappa,\beta}_t$ is the boundary length of $\mathcal{D}^\kappa_t$ for all $t<T^{\kappa,\beta}$. Let $a= 1+4/\kappa$. Let $\widetilde{L}^{\beta,a}$ (resp. $\widetilde{R}^{\beta,a}$) be the $(a-1)$-stable Lévy process of Lévy measure $\Lambda_{L}^{\beta,a}\coloneqq\frac{1-\beta}{2}\cos(a \pi) (dx/x^a){\bf 1}_{x>0} + \frac{1}{2}(dx/|x|^a){\bf 1}_{x<0}$ (resp. $\Lambda_{R}^{\beta,a}\coloneqq\frac{1+\beta}{2}\cos(a \pi) (dx/x^a){\bf 1}_{x>0} + \frac{1}{2}(dx/|x|^a){\bf 1}_{x<0}$). The two processes $\widetilde{L}^{\beta,a}$ and $\widetilde{R}^{\beta,a}$ are taken independent. Let $\widetilde{S}^a = \widetilde{L}^{\beta,a}+\widetilde{R}^{\beta,a}$.

\begin{proposition}\label{Prop 5.1 MSW}(Proposition 5.1 of \cite{MSW22})
	Up to some linear time-change, the process $(L^{\kappa,\beta}_t,R^{\kappa,\beta}_t)_{t\ge 0}$ {has a càdlàg version which} is a Doob $h$-transform of the process $(\widetilde{L}^{\beta,a},\widetilde{R}^{\beta,a})$, in the sense that there exists a constant $c>0$ such that for all $t>0$, for all bounded measurable function $F: \mathbb{D}([0,t],\R)^2 \to \R$, for all $x,y>0$,
	\begin{align*}
	\E&\left( \left. F\left((L^{\kappa, \beta}_s,R^{\kappa, \beta}_s)_{0\le s\le t}\right){\bf 1}_{E_t} \right| (L^{\kappa,\beta}_0,R^{\kappa,\beta}_0)=(x,y) \right) \\
	&= 
	\E_{x,y} \left(\left(\frac{x+y}{\widetilde{S}^a_{ct}}\right)^a F\left((\widetilde{L}^{\beta,a}_{cs},\widetilde{R}^{\beta,a}_{cs})_{0\le s\le t}\right) {\bf 1}_{\forall s \in [0,t], \ \widetilde{L}^{\beta,a}_{cs},\widetilde{R}^{\beta,a}_{cs}>0}\right),
	\end{align*}
	where {under $\P_{x,y}$} the Lévy process $\widetilde{L}^{\beta,a}$ (resp. $\widetilde{R}^{\beta,a}$) starts at $x$ (resp. $y$). Moreover, conditionally on the quantum boundary lengths $(L^{\kappa,\beta}_s,R^{\kappa,\beta}_s)_{0\le s\le t}$, the quantum surfaces which are cut out to the left, cut out to the right, inside the left loops or inside the right loops are independent $\sqrt{\kappa}$-quantum disks of quantum boundary lengths corresponding to the {absolute values of the} sizes of the {negative and positive} jumps of $L^{\kappa,\beta}$ and $R^{\kappa,\beta}$, while $(\mathcal{D}^\kappa_t, h^{\sqrt{\kappa}}_{\vert \mathcal{D}_t})$ is an independent $\sqrt{\kappa}$-quantum disk of quantum boundary length $S^{\kappa,\beta}_t$.
\end{proposition}

When we parametrize the trunk by its quantum length, the parametrization is chosen up to a multiplicative constant, so that for our purpose we choose to take $c=1$ in the above proposition. This proposition will be a key ingredient to obtain analogous statements for $\kappa =4$.

\subsection{Joint convergence of the $\mathrm{CLE}_\kappa$ and the $\sqrt{\kappa}$-quantum disks as $\kappa \uparrow 4$}\label{sous-section kappa vers 4}
In this subsection, we combine the convergences (\ref{cvCLE}) and (\ref{cv disque quantique}) to extend Proposition \ref{Prop 5.1 MSW} to the case $\kappa=4$ for the chordal $\mathrm{SLE}_4^{\langle \mu \rangle} (-2)$. Then the statement of Theorem \ref{prop de Markov disques quantiques} for the $\mathrm{SLE}^{\langle \mu \rangle}_4(-2)$ exploration will be a direct corollary.

Let $(\mathbb{D},h^{2})$ be a unit boundary length $2$-quantum disk. Let $L_0^4$ (resp. $R_0^4$) be the quantum boundary length along the clockwise (resp. counterclockwise) segment from $-i$ to $i$ in $\partial \mathbb{D}$. Let $\mu \in \R$. Let $\gamma$ be an independent chordal {$\mathrm{SLE}_4^{\langle \mu \rangle}(-2)$} on $\mathbb{D}$ from $-i$ to $i$. We again parametrize the trunk by its quantum length and we denote by $T^{4,\mu}$ its total quantum length. For each $t\ge0$ (corresponding here to the quantum length of the trunk), let $E_t$ be the event that the trunk has not yet attained $i$ before having quantum length $t$, i.e. $E_t=\{T^{4,\mu}>t\}$. For all $t<T^{4,\mu}$, we set $L_t^{4,\mu}$ and $R_t^{4,\mu}$ to be the left and right quantum boundary lengths of the remaining to be explored region $\mathcal{D}_t^4$ ($\mathcal{D}^4_t$ is the connected component of $\mathbb{D} \setminus \gamma([0,\tau_t])$ containing $i$ in its closure, where $\tau_t$ is the time for $\gamma$ at which the quantum length of the trunk reaches $t$). Note that for all $t<T^{4,\mu}$, we have $L_t^{4,\mu}>0$ and $R_t^{4,\mu}>0$.  For all $t\ge T^{4,\mu}$, we set $L_t^{4,\mu}=R_t^{4,\mu}=0$. If we write $S^{4,\mu}_t= L^{4,\mu}_t+R^{4,\mu}_t$ for all $t\ge0$, then $S^{4,\mu}_t$ is the quantum boundary length of $\mathcal{D}^4_t$ when $t<T^{4,\mu}$. Let $\widetilde{L}^2$ and $\widetilde{R}^2$ be two independent symmetric Cauchy processes of Lévy measure $\frac{1}{2} (dx/x^2){\bf 1}_{x\neq 0} $. Let $\widetilde{S}^2 = \widetilde{L}^2+\widetilde{R}^2$.

\begin{proposition}\label{absolue continuité SLE4}
The process $(S^{4,\mu}_t)_{t\ge 0}$ {has a càdlàg version which} satisfies the following absolute continuity relation: there exists a constant $c>0$ such that for all $t>0$, for all bounded measurable function $F: \mathbb{D}([0,t], \R)^2 \to \R$, for all $x,y>0$,
\begin{equation}\label{transfo de Doob}
\E\left( \left. F\left((S^{4,\mu}_s)_{0\le s\le t}\right) {\bf 1}_{E_t} \right| (L^{4,{\mu}}_0,R^{4,{\mu}}_0)=(x,y) \right) = 
\E_{x,y} \left(\left(\frac{x+y}{\widetilde{S}^2_{ct}}\right)^2 F\left((\widetilde{S}^2_{cs})_{0\le s\le t}\right) {\bf 1}_{\forall s \in [0,t], \ \widetilde{L}^2_{cs},\widetilde{R}^2_{cs}>0}\right),
\end{equation}
where {under $\P_{x,y}$} the Cauchy process $\widetilde{L}^2$ (resp. $\widetilde{R}^2$) starts at $x$ (resp. $y$). 
Moreover, conditionally on $(S^{4,\mu}_s)_{0\le s\le t}$, the quantum surfaces which are cut out or inside discovered loops before time $t$ are independent quantum disks of quantum boundary length corresponding to {the absolute values of sizes of the negative and positive jumps} of $(S^{4,\mu}_s)_{0\le s \le t}$, while $(\mathcal{D}^4_t,h^2_{|\mathcal{D}^4_t})$ is an independent quantum disk of quantum boundary length $S^{4,\mu}_t$.
\end{proposition}

Here we also choose the multiplicative constant for the quantum length of the trunk so that $c=1$ to be consistent with our choice of $c$ for Proposition \ref{Prop 5.1 MSW}. Note that the quantum length of the trunk of the chordal $\mathrm{SLE}_4^{\langle \mu \rangle}(-2)$ can also be defined using the quantum boundary length measure since it is an $\mathrm{SLE}_4$-type curve, so that there should be a choice for $c$ such that the quantum length of the trunk and the quantum boundary lengths of the cut out and encircled domains are defined using the same multiplicative constant. However, our approach does not enable us to identify this constant.

\begin{remark}
	The above proposition is not exactly the proper counterpart of Proposition \ref{Prop 5.1 MSW} insofar as we do not describe the left and right quantum boundary lengths of the unexplored region, but only the whole quantum boundary length. To get the evolution of the left and right quantum boundary lengths, one has to identify the quantum surfaces which are cut out or encircled by a loop on the right side and on the left side of the trunk. For the cut out domains one can see from Proposition \ref{prop cv CLE} that the domains which are cut out on the right (resp. left) side of the trunk are limits of domains which are cut out on the right (resp. left) side of the trunk. For the loops, one has to look at the direction the trunk takes after hitting the loops, using the proof of Proposition \ref{prop cv CLE} once more. We do not go further in that direction since it is not the main purpose of this work. {Note that in the case $\kappa=4$, the quantum boundary lengths of the loops discovered on the right-hand side and on the left-hand side during the SLE$^{\langle \mu \rangle}_4(-2)$ exploration should have the same intensity since the corresponding Bessel process here is symmetric ($\beta=0$). The term $\mu$ only adds an ``infinitesimal rotation at each time the process is not drawing a loop''. This explains why the processes $\widetilde{L}^2$ and $\widetilde{R}^2$ have the same law contrary to the case $\beta \neq 0$ and their law does not depend on $\mu$.}
\end{remark}
Using Proposition \ref{absolue continuité SLE4}, one can deduce Theorem \ref{prop de Markov disques quantiques} for the $\mathrm{SLE}_4^{\langle \mu \rangle}(-2)$ exploration:
\begin{proof}[Proof of Theorem \ref{prop de Markov disques quantiques} for $(Y^{\mu}_t)_{t\ge 0}$ using Proposition \ref{absolue continuité SLE4}]
Using the $\mathrm{SLE}_4^{\langle \mu \rangle}(-2)$ exploration tree, one can define the branch which follows the locally largest component in terms of quantum boundary length. It can be defined through an approximation as in Subsection 5.3 of \cite{MSW22}: let $\vp>0$, for each time of the form $\vp n$ for $n \in \N$, one updates the target point for the chordal exploration and chooses the point such that the quantum boundary lengths on both sides of the unexplored region are equal. When $\vp \to 0$, this exploration approximates the branch of the locally largest component. Denote by $(Y^{\mu}_t)_{t\ge 0}$ the quantum boundary length of the locally largest component 
parametrized by the quantum length of the trunk (which is given by a concatenation of parts of trunks of chordal $\mathrm{SLE}_4^{\langle \mu \rangle}(-2)$'s). Then, using exactly the same reasoning as in Subsection 5.3 of \cite{MSW22}, we deduce from Proposition \ref{absolue continuité SLE4} the analogous result for the branch of the locally largest component, which is no more than the statement of Theorem \ref{prop de Markov disques quantiques} for $(Y^{\mu}_t)_{t\ge 0}$. {Let us still give the main idea for the identification of the law of $(Y^\mu_t)_{t\ge 0}$ for completeness.

By Proposition \ref{absolue continuité SLE4}, while $L^{4,\mu}_t,R^{4,\mu}_t>0$, as a Doob $h$-transform of a Cauchy process, the process $S^{4,\mu}$ makes a jump of size $\ell \in \R\setminus \{0\}$ at rate
$$
\frac{c}{\ell^2} \left( \frac{S_t^{4,\mu}}{S^{4,\mu}_0} \right)^2 \left( \frac{S_t^{4,\mu} + \ell}{S^{4,\mu}_0} \right)^{-2} = \frac{c}{\ell^2} \left( \frac{S^{4,\mu}_t}{S^{4,\mu}_t + \ell}\right)^2.
$$
As a result, the process $(Y^\mu_t)_{t\ge 0}$ giving the quantum boundary length of the locally largest component makes a positive (resp.\@ negative) jump of size $\ell>0$ (resp.\@ $-\ell<0$) at rates
$$
 \frac{c}{\ell^2} \left( \frac{Y^{\mu}_t}{Y^{\mu}_t + \ell}\right)^2\qquad \text{and respectively} \qquad  \frac{c}{\ell^2} \left( \frac{Y^{\mu}_t}{Y^{\mu}_t - \ell}\right)^2 {\bf 1}_{\ell<Y^\mu_t/2},
$$
hence the desired result thanks to the description of $(X^{(-1)}(t))_{t\ge 0}$ given in Remark \ref{remarque taux de sauts}.
}
\end{proof}
All this subsection and the next one are then devoted to the proof of Proposition \ref{absolue continuité SLE4}. Let $(\mathbb{D},h^{\sqrt{\kappa}})$ be a parametrization of a unit-boundary $\sqrt{\kappa}$-quantum disk. We run an independent chordal $\mathrm{SLE}^{\mu((4/\kappa)-1)}_{\kappa}(\kappa-6)$ in $\mathbb{D}$ from $-i$ to $i$. We will prove the above result by letting $\kappa \uparrow 4$. 

{
\paragraph{Short outline of proof of Proposition \ref{absolue continuité SLE4}.} Let us summarize the proof before entering the details. We first prove the convergence of the processes giving the quantum boundary length of the unexplored region. Next, we build on the convergence \eqref{cv disque quantique} of the quantum disk and on the Carathéodory convergence of the CLE$_\kappa$ exploration to obtain the convergence of the restrictions of the field $h^{\sqrt{\kappa}}$ and of the quantum area measure to the cut out, encircled or unexplored domain containing a fixed point $z \in \mathbb{D}$. After proving this, the difficulty is to see that these domains appear in the ``same order'' in the exploration. To see this, we prove the convergence of the number of cut out or encircled domains of quantum boundary length at least $\vp>0$ for all $\vp >0$. Finally, we prove the convergence of the quantum length of the trunk. We deduce Proposition \ref{absolue continuité SLE4} from Proposition \ref{Prop 5.1 MSW} by letting $\kappa\uparrow 4$ and using all the above-mentioned convergences.
}

\paragraph{}The first step towards Proposition \ref{absolue continuité SLE4} is the convergence in law of the processes involved in Proposition \ref{Prop 5.1 MSW}. 
Let $(\overline{L}^4_0,\overline{R}^4_0)=(L^4_0,R^4_0)$ (see the beginning of Subsection \ref{sous-section kappa vers 4}). Next, conditionally on $(\overline{L}^4_0,\overline{R}^4_0)=(x,y)$, define a process $(\overline{L}_t^4,\overline{R}_t^4)_{t\ge 0}$ via the same Doob $h$-transform as in Proposition \ref{Prop 5.1 MSW} for $\kappa=4$, $a=1+4/\kappa=2$ and $\beta =0$, replacing $E_t$ by $\{\overline{L}_t^4,\overline{R}_t^4>0\}$ and setting $\overline{L}_t^4=\overline{R}_t^4=0$ on the complement. {Note that we choose to write $\overline{L}^4, \overline{R}^4$ the processes appearing in the following lemma instead of $L^{4,\mu},R^{4,\mu}$ since we do not know yet whether they indeed describe the quantum boundary length of the unexplored region as in Proposition \ref{absolue continuité SLE4}.

The following lemma establishes the convergence in law of the processes giving the left and right boundary lengths in  Proposition \ref{Prop 5.1 MSW}. Its proof relies on their description as a Doob $h$-transform of Lévy processes and on the convergence of these Lévy processes.}

\begin{lemma}\label{lemme cv Skorokhod}
	We have the convergence
	\begin{equation}\label{cv Skorokhod}
		(L_t^{\kappa,\mu((4/\kappa)-1)} , R_t^{\kappa,\mu((4/\kappa)-1)} )_{t\ge0}\mathop{\longrightarrow}\limits_{\kappa \uparrow 4}^{(\mathrm{d})}(\overline{L}_t^4 ,\overline{R}_t^4)_{t\ge 0}
	\end{equation}
	in the sense of Skorokhod. 
\end{lemma}
\begin{proof}
	The convergence of $(L_0^{\kappa,\mu((4/\kappa)-1}, R^{\kappa,\mu((4/\kappa)-1}_0)$ stems from the convergence of the quantum boundary length measure in (\ref{cv disque quantique}). 
	Besides, one can easily see that for all continuous bounded function $\varphi: \R \rightarrow \R$ vanishing on a neighbourhood of zero, we have
	$$
	\int_\R \varphi (x)\Lambda_L^{\mu((4/\kappa)-1), (4/\kappa)+1}(dx) \mathop{\longrightarrow}\limits_{\kappa \uparrow 4}
	\int_\R \varphi(x) \frac{dx}{2x^2}
	$$
	and
	$$
	\int_\R \varphi (x)\Lambda_R^{\mu((4/\kappa)-1), (4/\kappa)+1}(dx) \mathop{\longrightarrow}\limits_{\kappa \uparrow 4}
	\int_\R \varphi(x) \frac{dx}{2x^2}.
	$$
	It is classical (see e.g. \cite{JS87} VII.3.6) that this implies the convergence of Lévy processes
	$$
	\left(\widetilde{L}^{\mu((4/\kappa)-1), (4/\kappa)+1}_t\right)_{t \ge 0} \mathop{\longrightarrow}\limits_{\kappa \uparrow 4}^{(\mathrm{d})}
	\left(\widetilde{L}^{ 2}_t\right)_{t \ge 0}
	\enskip
	\text{and}
	\enskip
	\left(\widetilde{R}^{\mu((4/\kappa)-1), (4/\kappa)+1}_t\right)_{t \ge 0} \mathop{\longrightarrow}\limits_{\kappa \uparrow 4}^{(\mathrm{d})}
	\left(\widetilde{R}^{ 2}_t\right)_{t \ge 0},
	$$
	for the $J_1$ Skorokhod topology. We then reason by absolute continuity. Let $t,\vp >0$. By Proposition \ref{Prop 5.1 MSW}, if $F$ is a bounded continuous function on $\mathbb{D}([0,t],\R)^2$ vanishing on $\{ (f,g); \ \exists s \in [0,t], f(s)< \vp \text{ or } g(s) < \vp \}$, we have
	
	\begin{align*}
		\E&\left( \left. F\left((L^{\kappa,\beta}_s,R^{\kappa,\beta}_s)_{0\le s\le t}\right) {\bf 1}_{E_t}\right| (L^{\kappa,\beta}_0,R^{\kappa,\beta}_0)=(x,y) \right)\\
		&= 
		\E_{x,y} \left(\left(\frac{x+y}{\widetilde{S}^{ a}_{ct}}\right)^{a} F\left((\widetilde{L}_{cs}^{\beta, a},\widetilde{R}^{\beta, a}_{cs})_{0\le s\le t}\right) {\bf 1}_{\forall s \in [0,t], \ \widetilde{L}^{\beta, a}_{cs},\widetilde{R}^{\beta, a}_{cs}>0}\right),
	\end{align*}
	where we have set $\beta=\mu((4/\kappa)-1)$ and $a=(4/\kappa)+1$. One concludes by dominated convergence.
\end{proof}

Another ingredient is the fact that the convergences (\ref{cvCLE}) and (\ref{cv disque quantique}) imply the convergence of the underlying fields and quantum area measure of the cut out, encircled and unexplored surfaces:

\begin{lemma}\label{lemme cvDQ}
	If for all $z \in \mathbb{D}\cap \Q^2$, the function $f^{z,\kappa}_u$ is the unique conformal map from $\mathbb{D}$ to $D^{z,\kappa}_u$ which sends $0$ to $z$ such that $(f^{z,\kappa}_u)'(0)>0$, then jointly with (\ref{cvCLE}) and (\ref{cv disque quantique}),
	\begin{equation}\label{cvDQ}
		\!
		\!
		\!
		\begin{pmatrix}
			f^{z,\kappa}_u \\
			h^{\sqrt{\kappa}} \circ f^{z,\kappa}_u +Q_{\sqrt{\kappa}} \log |(f^{z,\kappa}_u)'|\\
			\mu^{\sqrt{\kappa}}_{h^{\sqrt{\kappa}}} \circ f^{z,\kappa}_u
		\end{pmatrix}_{u\ge 0, z \in \mathbb{Q}^2 \cap \mathbb{D}}
		\!
		\!
		\!
		\!
		\!
		\mathop{\longrightarrow}\limits_{\kappa \uparrow 4}^{(\mathrm{d})} \qquad
		\begin{pmatrix}
			f^{z,4}_u \\
			h^{2} \circ f^{z,4}_u +Q_{2} \log |(f^{z,4}_u)'|\\
			\mu^{2}_{h^{2}} \circ f^{z,4}_u
		\end{pmatrix}_{u\ge 0, z \in \mathbb{Q}^2 \cap \mathbb{D}},
	\end{equation}
	in terms of finite dimensional distributions, where the first coordinate is endowed with the topology of the uniform convergence on compact sets, the second one with the topology of the convergence of distributions and the third one with the weak convergence of measures on  $\mathbb{D}$.
\end{lemma}

\begin{proof}
	By Skorokhod's representation theorem we can assume that the convergences (\ref{cvCLE}) and (\ref{cv disque quantique}) hold a.s.. Let $z \in \mathbb{D} \cap \mathbb{Q}^2$ and $ u\ge 0$. Let $f_\kappa=f^{z,\kappa}_u$. Then, from the Carathéodory convergence of $D^{z,\kappa}_u$ as $\kappa \uparrow 4$, we know that $f_\kappa$ converges uniformly on compact sets towards $f_4$. 
	We also know that for every compact subset $K \subset D^{z,4}_u$, 
	the conformal map $f^{-1}_\kappa$ converges uniformly on $K$ towards $f_4^{-1}$ as $\kappa \uparrow 4$. From Cauchy's formula, we know that the derivatives of $f^{-1}_\kappa$ also converge uniformly on $K$ to the derivatives of $f_4^{-1}$. Let us first show that $h^{\sqrt{\kappa}} \circ f_\kappa$ converges to $h^2 \circ f_4$ for the weak-$*$ topology. Let $\varphi : \mathbb{D} \rightarrow \R$ be a $\mathcal{C}^\infty$ function with compact support. Since $f_\kappa$ converges uniformly on compact sets to $f_4$ as $\kappa \uparrow 4$, there exists some relatively compact open subset $U\subset \mathbb{D}$ such that for $\kappa$ close enough to $4$, one has
	$
	\supp (\varphi \circ f_\kappa^{-1}) = f_\kappa(\supp \varphi) \subset U.
	$ We thus have
	\begin{align*}
		|\langle h^{\sqrt{\kappa}} \circ f_\kappa -h^2 \circ f_4, \varphi \rangle| = 
		&|\langle h^{\sqrt{\kappa}}, (\varphi \circ f_\kappa^{-1}) \cdot |(f_\kappa^{-1})'|^2 \rangle
		-
		\langle h^{2}, (\varphi \circ f_4^{-1}) \cdot |(f_4^{-1})'|^2 \rangle|
		\\
		=&|\langle h^{\sqrt{\kappa}}, (\varphi \circ f_\kappa^{-1}) \cdot |(f_\kappa^{-1})'|^2 -(\varphi \circ f_4^{-1}) \cdot |(f_4^{-1})'|^2\rangle
		\\
		&+
		\langle h^{\sqrt{\kappa}}-h^{2}, (\varphi \circ f_4^{-1}) \cdot |(f_4^{-1})'|^2 \rangle|
		\\
		\le &\|h^{\sqrt{\kappa}} \|_{H^{-1}(U)} \| (\varphi \circ f_\kappa^{-1}) \cdot |(f_\kappa^{-1})'|^2 -(\varphi \circ f_4^{-1}) \cdot |(f_4^{-1})'|^2\|_{H^1(U)} \\
		&+\|h^{\sqrt{\kappa}}-h^{2}\|_{H^{-1}(U)} \|(\varphi \circ f_4^{-1}) \cdot |(f_4^{-1})'|^2\|_{H^1(U)}.
	\end{align*}
	But we know that $\|h^{\sqrt{\kappa}}-h^{2}\|_{H^{-1}(U)} \rightarrow 0$ as $\kappa \uparrow 4$. Thus, to prove that $\langle h^{\sqrt{\kappa}} \circ f_\kappa -h^2 \circ f_4, \varphi \rangle \to 0$, it remains to check that 
	$
	\| (\varphi \circ f_\kappa^{-1}) \cdot |(f_\kappa^{-1})'|^2 -(\varphi \circ f_4^{-1}) \cdot |(f_4^{-1})'|^2\|_{H^1(\mathbb{D})} \to 0
	$,
	in other words that
	$$
	\|(\varphi \circ f_\kappa^{-1}) \cdot |(f_\kappa^{-1})'|^2 -(\varphi \circ f_4^{-1}) \cdot |(f_4^{-1})'|^2\|_{\mathrm{L}^2(\mathbb{D})}
	+
	\left\|\nabla\left((\varphi \circ f_\kappa^{-1}) \cdot |(f_\kappa^{-1})'|^2 -(\varphi \circ f_4^{-1}) \cdot |(f_4^{-1})'|^2 \right)\right\|_{\mathrm{L}^2(\mathbb{D})} 
	$$
	tends to zero as $\kappa \uparrow 4$. {The above expression indeed goes to zero as $\kappa \uparrow 4$} since $\varphi$ is $\mathcal{C}^\infty$ with compact support and $f_\kappa^{-1}$ and its derivatives converge uniformly on compact sets to $ f_4^{-1}$ and its derivatives.
	Thus $h^{\sqrt{\kappa}} \circ f_\kappa$ converges to $h^2 \circ f_4$ for the weak-$*$ topology, so $h^{\sqrt{\kappa}} \circ f_\kappa+ Q_{\sqrt{\kappa}} \log |f_\kappa'|$ converges to $h^2 \circ f_4 + Q_{2} \log |f_4'|$ for the weak-$*$ topology. The third convergence of (\ref{cvDQ}) comes easily from the uniform convergence on compact sets of $f_\kappa^{-1}$ and from the convergence of $\mu^{\sqrt{\kappa}}_{h^{\sqrt{\kappa}}}$. 
\end{proof}

\subsection{Proof of Proposition \ref{absolue continuité SLE4}}
Now, let us prove Proposition \ref{absolue continuité SLE4} using the above lemmas. 

We first introduce some notation. For all $i\ge 1$, {for all $\vp>0$}, let $\tau_{\vp,i}^\kappa$ be the time of the $i$-th jump of $(S_t^{\kappa,\mu((4/\kappa)-1)})_{t\ge 0}$ of size larger than $\vp$. Then we know that {for all $\vp>0$}, $\tau_{\vp,i}^\kappa$ and the size of the jump at time $\tau_{\vp,i}^\kappa$ converge in distribution by (\ref{cv Skorokhod}) but we do not know yet that the limit corresponds to a jump of $S^{4,\mu}$.
Henceforth, in the rest of the subsection, we denote by $\mathrm{ql}^\kappa(u)$ the quantum length of the trunk at time $u$ for the original parametrization of the chordal $\mathrm{SLE}^{\mu((4/\kappa)-1)}_\kappa(\kappa-6)$. The function $\mathrm{ql}^\kappa$ is non-decreasing and such that its limit $\lim_{u\to \infty}\mathrm{ql}^\kappa(u)$ corresponds to the total quantum length of the trunk. Since the exploration runs for all $u>0$, we have $\mathrm{ql}^\kappa(u)<\lim_{v\to \infty}\mathrm{ql}^\kappa(v)$ for all $u\ge 0$. From Remark 2.3 of \cite{MSW22}, we know that $\mathrm{ql}^\kappa$ is characterized as the limit in probability:
\begin{equation}\label{eq longueur du tronc}
\mathrm{ql}^\kappa(u)  = C(\kappa)\lim_{\vp \to 0} \vp^{4/\kappa} \# \left\{ s \in [0,u]; \ -\Delta S_{\mathrm{ql}^\kappa(s)}^{\kappa,\mu((4/\kappa)-1)}  \in [\vp,2\vp]\right\}
\end{equation}
where one can compute $C(\kappa)=(4/\kappa)/(1-1/2^{4/\kappa})$ with our choice of the multiplicative constant below Proposition \ref{Prop 5.1 MSW} (the value of $C(\kappa)$ is not important for our purpose, we only want it to converge to a positive number as $\kappa\uparrow 4$).
It is clear from (\ref{eq longueur du tronc}) that $(\mathrm{ql}^\kappa(u))_{u\ge 0}$ is adapted with respect to the natural filtration of $\big(S_{\mathrm{ql}^\kappa(u)}^{\kappa,\mu((4/\kappa)-1)}\big)_{u\ge0}$ and in particular to the filtration associated to the quantum disks discovered during the chordal exploration. For all $t\ge 0$, we denote by $\mathrm{Cut}^\kappa_t$ the domain which is cut out or encircled by a discovered loop at (quantum length) time $t$. By convention, we set $\mathrm{Cut}^\kappa_t=\partial$ where $\partial$ is a cemetery point if no such domain is encircled or cut out at (quantum length) time $t$.

{Similarly, for all $u\ge 0$, we denote by $\mathrm{ql}^4(u)$ the quantum length of the trunk at time $u$ for the original parametrization of the chordal $\mathrm{SLE}_4^{\langle \mu \rangle}(-2)$, which is the limit in probability of $\vp$ times the number of cut out regions until time $t$ of quantum boundary length in $[\vp,2\vp]$, up to a multiplicative constant. For consistency with \eqref{eq longueur du tronc}, we take the constant equal to $2$. The function $\mathrm{ql}^4$ is then characterized as the limit in probability
\begin{equation}\label{eq longueur du tronc critique}
\mathrm{ql}^4(u) = 2 \lim_{\vp \to 0} \vp \# \{v \in [0,u];  \ -\Delta S^{4,\mu}_{\mathrm{ql}^4(v)} \in [\vp,2\vp]\}.
\end{equation}
}

By Skorokhod's representation theorem, we may assume that (\ref{cvCLE}), \eqref{eq cv region inexploree}, (\ref{cv disque quantique}) and (\ref{cvDQ}) hold almost surely.

Let us denote by $b_i^\kappa$ and $(b_i^\kappa)^2 a_i^\kappa$ the quantum boundary lengths and quantum areas of the domains 
which are encircled by loops or cut out by the trunk during the whole chordal exploration, ranked in the decreasing order of quantum boundary lengths. Note that by Proposition \ref{Prop 5.1 MSW}, conditionally on the $b_i^\kappa$'s, the $a_i^\kappa$'s are i.i.d. of the same law as $\mu^{\sqrt{\kappa}}_{h^{\sqrt{\kappa}}}(\mathbb{D})$. We also have the relation
\begin{equation}\label{eq somme des aires}
	\mu^{\sqrt{\kappa}}_{h^{\sqrt{\kappa}}} (\mathbb{D}) = \sum_{i\ge 1} (b_i^\kappa)^2 a_i^\kappa .
\end{equation}
Moreover, by Lemma \ref{lemme cv Skorokhod}, the quantum boundary lengths of the cut out and encircled domains during the whole $\mathrm{SLE}^{\mu((4/\kappa)-1)}_\kappa(\kappa-6)$ chordal exploration ranked in the non-increasing order converge in distribution as $\kappa \uparrow 4$. We also have the convergence in distribution of the $a_i^\kappa$'s by (\ref{cv disque quantique}). Let $\kappa_n \uparrow 4$. We may find a subsequence again denoted by $\kappa_n$ such that the families $(b_i^{\kappa_n})_{i\ge 1}$ and {$(a_i^{\kappa_n})_{i\ge 1}$} converge a.s.\@ by Skorokhod's theorem so that we have a.s.\@ with the above convergences
\begin{equation}\label{eq cv a et b}
	\forall i\ge1, \qquad b_i^{\kappa_n} \mathop{\longrightarrow}\limits_{n\to \infty} b_i
	\qquad \text{and}
	\qquad
	a_i^{\kappa_n} \mathop{\longrightarrow}\limits_{n\to \infty} a_i.
\end{equation}
\paragraph{Step 1.} We first prove {for all $\vp>0$} the tightness of the number of cut out and encircled domains of quantum boundary length larger than $\vp$. 
{Let $u\ge 0$. For all $\vp >0$, for all $\kappa \in (8/3,4]$, let $N_\vp^\kappa(u)$ be the number of cut out domains or domains encircled by loops of quantum boundary length larger than $\vp$ before time $u$, i.e.
	$$
	N_\vp^\kappa(u)= \#\left\{t\le \mathrm{ql}^\kappa(u); \  
	\left\vert\Delta S_{t}^{\kappa,\mu((4/\kappa)-1)} \right\vert>\vp \right\}.
	$$
\begin{lemma}\label{lemme N epsilon est tendu}
For all $\vp >0$ and $u\ge 0$, the sequence $(N^{\kappa_n}_\vp(u))_{n\ge 0}$ is tight.
\end{lemma}}
\begin{proof}
{One can first see that}
\begin{equation}\label{eq tightness N}
\sum_{i\ge 1} \vp^2 a_i^{\kappa_n} {\bf 1}_{b_i^{\kappa_n}>\vp}  \le \sum_{i\ge 1} (b_i^{\kappa_n})^2 a_i^{\kappa_n} {\bf 1}_{b_i^{\kappa_n}>\vp} \le \mu^{\sqrt{\kappa_n}}_{h^{\sqrt{\kappa_n}}}(\mathbb{D}) 
\mathop{\longrightarrow}\limits^{\mathrm{a.s.}}_{n\to \infty} \mu^2_{h^2}(\mathbb{D}).
\end{equation}
{We deduce that $(\#\{i\ge 1;  \ {b_i^{\kappa_n}>\vp}\})_{n\ge 0}$ is tight, otherwise since the $a_i^{\kappa_n}$'s are i.i.d.\@ and independent from the $b_i^{\kappa_n}$'s there would exist $\eta>0$ such that for all integer $k\ge 1$, for infinitely many $n$, with probability at least $\eta$ the above sum stochastically dominates $\sum_{i=1}^k \vp^2 a_i^{\kappa_n}$ which converges to $\sum_{i=1}^k \vp^2 {a}_i$ by \eqref{eq cv a et b}, where the ${a}_i$'s are i.i.d.\@ random variables following the same law as $\mu^2_{h^2}(\mathbb{D})$ by (\ref{cv disque quantique}), which contradicts \eqref{eq tightness N} by taking $k$ large. Thus $N^{\kappa_n}_{\vp}(u)$ is tight as $N^{\kappa_n}_{\vp}(u)\le \#\{i\ge 1;  \ {b_i^{\kappa_n}>\vp}\}$.}
\end{proof}
{By Lemma \ref{lemme N epsilon est tendu},} we may take a subsequence such that jointly with the above convergences, for all $u\in \Q \cap \R_+$, {for all $\vp \in \Q\cap \R_+^*$}, we have
\begin{equation}\label{eq cv N}
N^{\kappa_n}_\vp(u) \mathop{\longrightarrow}\limits_{n \to \infty} N_\vp(u).
\end{equation}

\paragraph{Step 2.} {Let $\vp\in \Q\cap \R_+^*$}. Let us prove the Carath\'eodory convergence of the cut out and encircled domains of quantum boundary length larger than $\vp$ along some subsequence: 
{\begin{lemma}\label{lemme cut cv sous suite} Let $u \in \Q \cap \R_+$. Let $1 \le {i_0} \le N_\vp(u)$. The sequence $(\mathrm{Cut}_{\tau_{\vp,{i_0}}^{\kappa_n}}^{\kappa_{n}})_{n\ge 0}$ converges a.s.\@ along some (deterministic) subsequence which is again denoted by $(\kappa_n)_{n\ge 0}$ in the sense of Carathéodory.
\end{lemma}
}
\begin{proof}
{First, note that $\mu^{\sqrt{\kappa_n}}_{h^{\sqrt{\kappa_n}}}(\mathrm{Cut}_{\tau_{\vp,{i_0}}^{\kappa_n}}^{\kappa_{n}})$ can be written in the form $(b_{\vp,i_0}^{\kappa_n})^2 a_{\vp,i_0}^{\kappa_n}$, where $b_{\vp,i_0}^{\kappa_n}$ is the $i_0$-th jump of $S^{\kappa_n,\mu((4/\kappa_n)-1)}$ of size larger than $\vp$ (in absolute value), and where $a^{\kappa_n}_{\vp,i_0}$ is independent of $b_{\vp,i_0}^{\kappa_n}$ and has the same law as the $a_i^{\kappa_n}$'s. Thus, by \eqref{cv disque quantique} and \eqref{cv Skorokhod}, by taking another subsequence again denoted by $\kappa_n$ and by applying Skorokhod's representation theorem, we may assume that $b_{\vp,i_0}^{\kappa_n}$ and $a_{\vp,i_0}^{\kappa_n}$ converge a.s.\@ to some independent random variables $b_{\vp,i_0}\neq 0$ and $a_{\vp,i_0}$, where $a_{\vp,i_0}$ has the same law as $\mu^2_{h^2}(\mathbb{D})$. In particular,
\begin{equation}\label{cv aire quantique cut}
	\mu^{\sqrt{\kappa_n}}_{h^{\sqrt{\kappa_n}}}(\mathrm{Cut}_{\tau_{\vp,{i_0}}^{\kappa_n}}^{\kappa_{n}}) = (b_{\vp,i_0}^{\kappa_n})^2 a_{\vp,i_0}^{\kappa_n} \mathop{\longrightarrow}\limits_{n\to \infty}^{(\mathrm{a.s.)}}
		(b_{\vp,i_0})^2 a_{\vp,i_0}.
\end{equation}}

{Let us then prove that a.s.\@ there exists a random variable $\mathcal{Z}\in \mathbb{D}\cap \Q^2$ and a random subsequence $\mathcal{K}_n$ of $\kappa_n$ such that for all $n \ge 0$, we have $\mathcal{Z} \in \mathrm{Cut}_{\tau_{\vp,{i_0}}^{\mathcal{K}_n}}^{\mathcal{K}_n}$, so that in particular $ \mathrm{Cut}_{\tau_{\vp,{i_0}}^{\mathcal{K}_n}}^{\mathcal{K}_n}$ converges in the sense of Carathéodory towards $D^{\mathcal{Z},4}_u$ since  $\mathrm{Cut}_{\tau_{\vp,i_0}^{\mathcal{K}_n}}^{\mathcal{K}_n} = D^{\mathcal{Z},\mathcal{K}_n}_u$ and due to \eqref{cvCLE}.} 
Assume by contradiction that for all $z \in \mathbb{D} \cap \Q^2$, there exists $n_z \in \N$ such that for all $n \ge n_z$, we have $z \not \in {\mathrm{Cut}_{\tau_{\vp,i_0}^{\kappa_n}}^{\kappa_n}}$. Let $(D^{z_i,4}_u)_{i\ge 1}$ be an enumeration of the domains of the form $D^{z,4}_u$ for some $z \in \mathbb{D} \cap \Q^2$. By (\ref{cv disque quantique}) we know that $\mu^{\sqrt{\kappa}}_{h^{\sqrt{\kappa}}}(\mathbb{D})$ converges towards $\mu^2_{h^2}(\mathbb{D})$ as $\kappa \uparrow 4$. Moreover, by (\ref{cvDQ}), for all $z\in \mathbb{D}\cap \Q^2$, we also have the convergence of $\mu^{\sqrt{\kappa}}_{h^{\sqrt{\kappa}}}(D^{z,\kappa}_u)$ towards $\mu^2_{h^2}(D^{z,4}_u)$ as $\kappa \uparrow 4$. As a consequence, we have for all $k\ge 1$,
\begin{align}
\sum_{i=1}^k \mu_{h^2}^2(D^{z_i,4}_u)= \lim_{n\to \infty} \sum_{i=1}^k \mu^{\sqrt{\kappa_n}}_{h^{\sqrt{\kappa_n}}}(D^{z_i,\kappa_n}_u)\le &\lim_{n\to \infty} \mu^{\sqrt{\kappa_n}}_{h^{\sqrt{\kappa_n}}}(\mathbb{D})-\min \{(b^{{\kappa_n}}_j)^2 a^{{\kappa_n}}_j ; \  j\ge 1, b^{{\kappa_n}}_j >\vp\} \label{eq probleme de masse}
\\ 
&= \mu^2_{h^2}(\mathbb{D})-
\min \{(b_j)^2 a_j ; \  j\ge 1, \ b_j >\vp\}. \notag
\end{align}
The above inequality in (\ref{eq probleme de masse}) comes {from the fact that $\mathrm{Cut}_{\tau_{\vp,i_0}^{\kappa_n}}^{\kappa_n}\neq D^{z_i,\kappa_n}_u$ for all $i \in \lb 1 , k \rb$ for all $n$ large enough (since $z_i \not\in\mathrm{Cut}_{\tau_{\vp,i_0}^{\kappa_n}}^{\kappa_n}$ for $n$ large enough)} and from the fact that for $n$ large enough, the $D^{z_i,\kappa_n}_u$'s for $i \in \lb 1, k \rb$ are distinct. Indeed, if $i\neq j$, then we cannot have $D^{z_i,\kappa_n}_u=D^{z_j,\kappa_n}_u$ infinitely often, otherwise we would have $D^{z_i,4}_u=D^{z_j,4}_u$. Moreover, the left-hand side of (\ref{eq probleme de masse}) converges to $\mu^2_{h^2}(\mathbb{D})$ as $k\to \infty$, hence a contradiction. 

{
Finally, let us prove that a.s.\@ $\mathcal{Z} \in \mathrm{Cut}^{\kappa_n}_{\tau^{\kappa_n}_{\vp,i_0}}$ for all $n$ large enough.
This will end the proof since in this case we will have $\mathrm{Cut}^{\kappa_n}_{\tau^{\kappa_n}_{\vp,i_0}} = D^{\mathcal{Z}, \kappa_n}_u$ for all $n$ large enough, which ensures the desired Carathéodory convergence. To see this, it suffices to notice that, on the one hand, by \eqref{cvDQ}, we have the convergence of $\mu^{\sqrt{\kappa_n}}_{h^{\sqrt{\kappa_n}}}(D^{\mathcal{Z}, \kappa_n}_u)$ towards $\mu^{2}_{h^2}(D^{\mathcal{Z}, 4}_u)$ and on the other hand, we have the convergence \eqref{cv aire quantique cut}. Moreover, since the $a_i$'s have the same law as the inverse of an exponential random variable thanks to Theorem 5.5 of \cite{AHPS21} and to Theorem 1.1 of \cite{AdS22}, their law has a density, so that by \eqref{eq cv a et b}, a.s.\@ for all $z,z'\in \mathbb{D}\cap \Q^2$, if $D^{z,4}_u\neq D^{z',4}_u$ are distinct cut out domains, then $\mu^2_{h^2}(D^{z,4}_u)\neq \mu^2_{h^2}(D^{z',4}_u) $. Thus, a.s.\@  $\mathrm{Cut}^{\kappa_n}_{\tau^{\kappa_n}_{\vp,i_0}} $ cannot converge in the sense of Carathéodory to some other cut out domain $D^{z, 4}_u\neq D^{\mathcal{Z}, 4}_u$
along a subsequence. So by doing the same reasoning as in the previous paragraph, one can see that a.s.\@ we cannot have $\mathcal{Z}\not\in \mathrm{Cut}^{\kappa_n}_{\tau^{\kappa_n}_{\vp,i_0}}$ infinitely often (otherwise we would construct a random variable $\mathcal{Z}'\in \mathbb{D} \cap \Q^2$ a.s.\@ not in $D^{\mathcal{Z}, 4}_u$ and a random subsequence $\mathcal{K}'_n$ such that for all $n\ge 0$, $\mathcal{Z}'\in \mathrm{Cut}^{\mathcal{K}'_n}_{\tau^{\mathcal{K}'_n}_{\vp,i_0}}$ so that $\mathrm{Cut}^{\mathcal{K}'_n}_{\tau^{\mathcal{K}'_n}_{\vp,i_0}}$ would converge in the sense of Carathéodory towards $D^{\mathcal{Z}', 4}_u\neq D^{\mathcal{Z}, 4}_u$). }
\end{proof}

\paragraph{Step 3.} By applying Lemma \ref{lemme cut cv sous suite}, we assume that a.s.\@ for all $ i \in \lb 1, N_\vp(u) \rb$, the domain $\mathrm{Cut}^{{ \kappa}_n}_{\tau_{\vp,i}^{{\kappa}_n}}$ converges in the sense of Carath\'eodory. Let us prove that the convergence of the $\mathrm{Cut}^{{\kappa}_n}_{\tau_{\vp,i}^{\kappa_n}}$'s also holds in terms of quantum disks. {More precisely:

\begin{lemma}\label{lemme cvDQ3} Let  $ i \in \lb 1, N_\vp(u) \rb$. Let $z \in \mathbb{D}\cap \Q^2$. Let us work on the event where $z$ is in the almost sure limit of $\mathrm{Cut}^{\kappa_n}_{\tau_{\vp,i}^{\kappa_n}}$. If $f_{\kappa_n}$ (resp. $f_4$) is the unique conformal map from $\mathbb{D}$ to $D^{z,\kappa_n}_u$ (resp. $D^{z,4}_u$) such that $f_{\kappa_n}(0)=z$ and $f'_{\kappa_n}(0)>0$ (resp. $f_4(0)=z$ and $f'_4(0)>0$), then we have the convergence jointly with the above convergences
\begin{equation}\label{cvDQ3}
	\begin{pmatrix}
		h^{\sqrt{\kappa_n}} \circ f_{\kappa_n} +Q_{\sqrt{\kappa_n}} \log |f_{\kappa_n}'|\\
		\mu^{\sqrt{\kappa_n}}_{h^{\sqrt{\kappa_n}}} \circ f_{\kappa_n}\\
		\nu^{\sqrt{\kappa_n}}_{h^{\sqrt{\kappa_n}}_{\vert D^{z,\kappa_n}_u}} \circ f_{\kappa_n}
	\end{pmatrix}
	\qquad
	\mathop{\longrightarrow}\limits_{\kappa \uparrow 4}^{(\mathrm{d})} \qquad
	\begin{pmatrix}
		h^{2} \circ f_4 +Q_{2} \log |f'_4|\\
		\mu^{2}_{h^{2}} \circ f_4\\
		\nu^{2}_{h^{2}_{\vert D^{z,4}_u}} \circ f_4
	\end{pmatrix}
\end{equation}
where the first coordinate is endowed with the $H^{-1}_\mathrm{loc}$ topology and the two last ones with the weak convergence of measures respectively on $\mathbb{D}$ and $\partial \mathbb{D}$. 
\end{lemma}}
\begin{proof}
Recall the convergence (\ref{cvDQ}) and the fact that for $n$ large enough, by the Carath\'eodory convergence of $\mathrm{Cut}^{\kappa_n}_{\tau^{\kappa_n}_{\vp,i}}$, we have $D^{z,\kappa_n}_u = \mathrm{Cut}^{\kappa_n}_{\tau^{\kappa_n}_{\vp,i}}$. Moreover, $ \mathrm{Cut}^{\kappa_n}_{\tau^{\kappa_n}_{\vp,i}}$ is one of the cut-out or encircled domains of quantum boundary length $b_i^{\kappa_n}$ so that there exists a random integer $I_n$, which is a.s.\@ bounded due to (\ref{eq somme des aires}), such that $ \mathrm{Cut}^{\kappa_n}_{\tau^{\kappa_n}_{\vp,i}}$ is the quantum disk of quantum boundary length $b_{I_n}^{\kappa_n}$. Since for all $i \ge 1$ such that $b_i^{\kappa_n}> \vp$, conditionally on $b_i^{\kappa_n}$, the corresponding domain is a $\sqrt{\kappa_n}$-quantum disk, we deduce from (\ref{cv disque quantique}) that 
$$
\widetilde{h}^{\sqrt{\kappa_n}} \coloneqq {{\bf 1}_{z \in \mathrm{Cut}^{\kappa_n}_{\tau^{\kappa_n}_{\vp,i}}}} \left( h^{\sqrt{\kappa_n}} \circ f_{\kappa_n}+ Q_{\sqrt{\kappa_n}} \log |f_{\kappa_n}'|- \frac{2}{\sqrt{\kappa_n}} \log \left(\nu_{h^{\sqrt{\kappa_n}}_{\vert D^{z,\kappa_n}_u}}^{\sqrt{\kappa_n}} (\partial D^{z,\kappa_n}_u) \right)\right)
$$
is tight as $n \to \infty$ in  $H^{-1}_\mathrm{loc}(\mathbb{D})$ (and also the quantum area and length measures $\mu_{\widetilde{h}^{\sqrt{\kappa_n}}}^{\sqrt{\kappa_n}}$ and $\nu_{\widetilde{h}^{\sqrt{\kappa_n}}}^{\sqrt{\kappa_n}}$ for the weak convergence of measures).
Thus $\log \nu_{h^{\sqrt{\kappa_n}}_{\vert D^{z,\kappa_n}_u}}^{\sqrt{\kappa_n}} (\partial D^{z,\kappa_n}_u)$ is tight in $\R$ as $n \to \infty$. So $h^{\sqrt{\kappa_n}} \circ f_{\kappa_n}+ Q_{\sqrt{\kappa_n}} \log |f_{\kappa_n}'|$ is tight in $H^{-1}_\mathrm{loc}(\mathbb{D})$ and hence converges in $H^{-1}_\mathrm{loc}(\mathbb{D})$ towards $h^2 \circ f_4 + Q_{2} \log |f_4'|$. Let us take a subsequence again denoted by $\kappa_n \uparrow 4$ such that 
$$
\begin{pmatrix}
	h^{\sqrt{\kappa_n}} \circ f_{\kappa_n}+ Q_{\sqrt{\kappa_n}} \log |f_\kappa'|\\
	(\widetilde{h}^{\sqrt{\kappa_n}} , \mu_{\widetilde{h}^{\sqrt{\kappa_n}}}^{\sqrt{\kappa_n}},\nu_{\widetilde{h}^{\sqrt{\kappa_n}}}^{\sqrt{\kappa_n}}) \\
	\log \left(\nu_{h^{\sqrt{\kappa_n}}_{D^{z,\kappa_n}_u}}^{\sqrt{\kappa_n}} (\partial D^{z,\kappa_n}_u)\right)
\end{pmatrix}
\mathop{\longrightarrow}\limits_{n\to \infty}^{(\mathrm{a.s.})}
\begin{pmatrix}
	h^2 \circ f_4 + Q_{2} \log |f_4'|\\
	(\widetilde{h}^2, \mu^2_{\widetilde{h}^2}, \nu^2_{\widetilde{h}^2})\\
	Z
\end{pmatrix},
$$
where  $Z$ is some real random variable and $\widetilde{h}^2$ is the field corresponding to the parametrization by the unit disk of the quantum disk corresponding to $b^4_I$ for some random $I$ integer such that $b^4_I\ge \vp$. 
Then $\widetilde{h}^2+ Z = h^2 \circ f_4+ Q_{2} \log |f_4'|$. Therefore 
$$
e^Z=\nu^2_{\widetilde{h}^2+Z}(\partial \mathbb{D}) = \nu^2_{h^2 \circ f_4+ Q_{2} \log |f_4'|} (\partial \mathbb{D})
=\nu^2_{h^2_{|D^{z,4}_u}}(f_4(\partial \mathbb{D})) = \nu^2_{h^2_{|D^{z,4}_u}}(\partial D^{z,4}_u),
$$
where the third equality comes by definition of the equivalence relation on quantum surfaces. Thus, 
we have the convergence in law of $\log \nu_{h^{\sqrt{\kappa_n}}_{\vert D^{z,\kappa_n}_u}}^{\sqrt{\kappa_n}} (\partial D^{z,\kappa_n}_u)$ towards $\log \nu^2_{h^2_{|D^{z,4}_u}}(\partial D^{z,4}_u) $. As a consequence, 
$$\widetilde{h}^2 = h^2 \circ f_4+ Q_{2} \log |f_4'| - \log \nu^2_{h^2_{|D^{z,4}_u}}(\partial D^{z,4}_u). $$ 
The convergence of $
\nu^{\sqrt{\kappa_n}}_{h^{\sqrt{\kappa_n}}_{\vert D^{z,\kappa_n}_u}} \circ f_{\kappa_n}$ comes from 
the convergence of $\nu_{h^{\sqrt{\kappa_n}}_{\vert D^{z,\kappa_n}_u}}^{\sqrt{\kappa_n}} (\partial D^{z,\kappa_n}_u)$ and from the fact that
$$
\nu^{\sqrt{\kappa_n}}_{h^{\sqrt{\kappa_n}}_{\vert D^{z,\kappa_n}_u}} \circ f_{\kappa_n} =
\nu^{\sqrt{\kappa_n}}_{h^{\sqrt{\kappa_n}}_{\vert D^{z,\kappa_n}_u}} (\partial D^{z,\kappa_n}_u) \nu^{\sqrt{\kappa_n}}_{\widetilde{h}^{\sqrt{\kappa_n}}}.
$$
This ends the proof of (\ref{cvDQ3}).\end{proof}

\paragraph{Step 4.} Recall from \eqref{eq cv N} that $N_\vp(u)$ is the almost sure limit of $N^{\kappa_n}_\vp(u)$ for all $\vp \in \Q \cap \R_+^*$ and $u \in \Q \cap \R_+$. Let us identify the limit. {The main idea is to show that any cut out domain of quantum boundary length larger than $\vp$ is the limit of some cut out domain of quantum boundary length larger than $\vp$ so that $N_\vp(u) \ge N^4_\vp(u)$ and to use the convergence of the cut out domain of quantum boundary length larger than $\vp$ obtained in Lemma \ref{lemme cvDQ3} to obtain the reverse inequality.}
\begin{lemma}\label{lemme identification N epsilon}
	We have $N_\vp(u)=N^4_\vp(u)$ for all $u \in \Q \cap \R_+$.
\end{lemma} 
\begin{proof}
Recall that we have a sequence $\kappa_n \uparrow 4$ such that the convergences (\ref{cvCLE}), (\ref{cv disque quantique}), (\ref{cvDQ}), (\ref{eq cv a et b}), (\ref{eq cv N}) and (\ref{cvDQ3}) {from Lemma \ref{lemme cvDQ3}} hold almost surely ({for the last one, we again apply Skorokhod's representation theorem}).

By the proof of Proposition \ref{prop cv CLE}, we know that any cut out (resp. encircled) domain of the form $\mathrm{Cut}^{4}_{\tau_{i_0,\vp}^{4}}$ for some {$i_0 \le N^4_{\vp}(u)$} is the limit in the sense of Carath\'eodory of some cut out (resp. encircled) domain. Moreover, by (\ref{cvDQ}) we know that if $z \in \mathrm{Cut}^{4}_{\tau_{i_0,\vp}^{4}}$, then for $n$ large enough, we have 
\begin{equation}\label{eq inegalite aire quantique cut}
\mu^{\sqrt{\kappa_n}}_{h^{\sqrt{\kappa_n}}} (D^{z,\kappa_n}_u) \ge \mu^2_{h^2}( \mathrm{Cut}^{4}_{\tau_{i_0,\vp}^{4}})/2. 
\end{equation}
To prove that $D^{z,\kappa_n}_u$ converges in terms of quantum disks towards $\mathrm{Cut}^{4}_{\tau_{i_0,\vp}^{4}}$, we distinguish two cases:
\begin{itemize}
	\item[--] If $\mathrm{Cut}^{4}_{\tau_{i_0,\vp}^{4}}$ is a domain encircled by a loop, then we know by Proposition 3.1{6} of \cite{AG23} that the quantum boundary length of $\mathrm{Cut}^{4}_{\tau_{i_0,\vp}^{4}}$ is the limit of the quantum boundary length of $D^{z,\kappa_n}_u$ as $n \to \infty$. 
 \item[--] If $\mathrm{Cut}^{4}_{\tau_{i_0,\vp}^{4}}$ is a cut out domain, then {we proceed as follows. Let $b_{i,+}^{\kappa_n}$'s for $i\ge 1$ (resp.\@ the $(b_{i,+}^{\kappa_n})^2 a_{i,+}^{\kappa_n}$'s) stand for the $b_{i}^{\kappa_n}$'s (resp.\@ the $(b_{i}^{\kappa_n})^2 a_{i}^{\kappa_n}$'s) which correspond to a domain which is encircled by a loop, i.e.\@ the $b_{i,+}^{\kappa_n}$'s are the positive jumps of $S^{\kappa_n, \mu((4/\kappa_n)-1)}$, ordered so that $(b_{i,+}^{\kappa_n})_{i\ge 1}$ is non-increasing.}
 {By} {Lemma 3.14} of \cite{AG23} {we have} the convergence in law of the quantum boundary length of a loop {$\mathcal{L}_n$} surrounding a point chosen at random according to $\mu^{\sqrt{\kappa_n}}_{h^{\sqrt{\kappa_n}}}/\mu^{\sqrt{\kappa_n}}_{h^{\sqrt{\kappa_n}}}(\mathbb{D})$ towards a positive random variable {$\widetilde{b}$}. 
 {Besides, for all $\eta>0$,
 \begin{equation}\label{eq majoration somme des b i a i plus}
 \sum_{i \ge k} (b_{i,+}^{\kappa_n})^2 a_{i,+}^{\kappa_n} \le  \sum_{i\ge 1}  \left( {b_{i,+}^{\kappa_n}}\wedge \eta\right)^2a_{i,+}^{\kappa_n} + \sum_{i\ge k}(b_{i,+}^{\kappa_n})^2 a_{i,+}^{\kappa_n} {\bf 1}_{b_{i,+}^{\kappa_n}>\eta}
 .
 \end{equation}
 The first sum on the right hand side of \eqref{eq majoration somme des b i a i plus} can be rewritten
 $$
 \mu^{\sqrt{\kappa_n}}_{h^{\sqrt{\kappa_n}}} (\mathbb{D})
 \sum_{i\ge 1} \frac{a_{i,+}^{\kappa_n}}{\mu^{\sqrt{\kappa_n}}_{h^{\sqrt{\kappa_n}}}(\mathbb{D})} \left( {b_{i,+}^{\kappa_n}}\wedge \eta\right)^2 \le \mu^{\sqrt{\kappa_n}}_{h^{\sqrt{\kappa_n}}} (\mathbb{D}) \E\left(\left. \left(\nu^{\sqrt{\kappa_n}}_{h^{\sqrt{\kappa_n}}}(\mathcal{L}_n)\wedge \eta\right)^2 \right\vert (a_{i,+}^{\kappa_n})_{i\ge 1}\right).
 $$
 Then, note that by (\ref{cvDQ}), $ \mu^{\sqrt{\kappa_n}}_{h^{\sqrt{\kappa_n}}} (\mathbb{D})$ converges a.s.\@ and that by the convergence in law of $\nu^{\sqrt{\kappa_n}}_{h^{\sqrt{\kappa_n}}}(\mathcal{L}_n)$ towards $\widetilde{b}$ we have
 $$
 \E\left(\left(\nu^{\sqrt{\kappa_n}}_{h^{\sqrt{\kappa_n}}}(\mathcal{L}_n)\wedge \eta\right)^2\right) \mathop{\longrightarrow}\limits_{n\to \infty} \E \left( (\widetilde{b}\wedge \eta)^2 \right),
 $$
 which can be made arbitrarily small by letting $\eta \to 0$. Thus, by a union bound and by Markov's inequality, we deduce that for all $\delta>0$,
 \begin{equation}\label{eq limite premiere somme des b i a i plus}
  \lim_{\eta \to 0} \limsup_{n\to \infty}\P \left(\sum_{i\ge 1}  \left( {b_{i,+}^{\kappa_n}}\wedge \eta\right)^2a_{i,+}^{\kappa_n} >\delta\right) =0.
 \end{equation}
 Next, let us consider the second sum on the right-hand side of \eqref{eq majoration somme des b i a i plus}. By the proof of Lemma 3.13 in \cite{AG23} we know that the number of loops with quantum boundary length larger than $\eta$ is tight so that $N^{\kappa_n}_{\eta,+} \coloneqq\#\{i\ge 1 ; \ b^{\kappa_n}_{i,+}>\eta\}$ is tight and note that
 $$
 \sum_{i\ge k}(b_{i,+}^{\kappa_n})^2 a_{i,+}^{\kappa_n} {\bf 1}_{b_{i,+}^{\kappa_n}>\eta} = \sum_{i=k}^{N^{\kappa_n}_{\eta,+}} (b_{i,+}^{\kappa_n})^2 a_{i,+}^{\kappa_n}\le \sum_{i=k}^{N^{\kappa_n}_{\eta,+}} (b_{i}^{\kappa_n})^2 a_{i,+}^{\kappa_n}.
 $$
 Therefore, using \eqref{eq cv a et b}, and the fact that the $a_{i,+}^{\kappa_n}$'s are i.i.d.\@ with the same law as the $a_{i}^{\kappa_n}$'s, we obtain that the right-hand side in the above inequality is tight and if $L_{\eta,k}$ is a limit in distribution along a subsequence of the right-hand side of the above display, then one can see that $L_{\eta,k}$ goes to $0$ in probability as $k\to \infty$. As a result, for all $\delta>0$,
 \begin{equation}\label{eq limite deuxieme somme des b i a i plus}
 \lim_{k\to \infty} \limsup_{n\to \infty} \P\left( \sum_{i\ge k}(b_{i,+}^{\kappa_n})^2 a_{i,+}^{\kappa_n} {\bf 1}_{b_{i,+}^{\kappa_n}>\eta} >\delta\right) = 0.
 \end{equation}
 Combining \eqref{eq majoration somme des b i a i plus}, \eqref{eq limite premiere somme des b i a i plus} and \eqref{eq limite deuxieme somme des b i a i plus}, we get that for all $\delta>0$,
 $$
 \lim_{k\to \infty} \limsup_{n\to \infty} \P\left( \sum_{i=k}^\infty(b_{i,+}^{\kappa_n})^2 a_{i,+}^{\kappa_n} >\delta \right) =0.
 $$
}
As a consequence, by Proposition \ref{Prop 5.1 MSW}, and since the intensities of the Poisson point process of the positive and negative jumps of $\widetilde{S}^{1+4/\kappa_n}$ only differ by a multiplicative constant $\cos((1+4/\kappa_n)\pi)$ tending to $1$ as $n\to \infty$, we also have {for all $\delta>0$,
 $$
 \lim_{k\to \infty} \limsup_{n \to \infty} \P\left( \sum_{i \ge k} (b_{i,-}^{\kappa_n})^2 a_{i,-}^{\kappa_n} >\delta \right)=0,
 $$}
where the $b_{i,-}^{\kappa_n}$'s (resp.\@ the $(b_{i,-}^{\kappa_n})^2 a_{i,-}^{\kappa_n}$'s) stand for the $b_{i}^{\kappa_n}$'s (resp.\@ the $(b_{i}^{\kappa_n})^2 a_{i}^{\kappa_n}$'s) which correspond to a cut out domain, i.e.\@ the $b_{i,-}^{\kappa_n}$'s are the negative jumps of $S^{\kappa_n, \mu((4/\kappa_n)-1)}$. Therefore, {for all $\delta\in (0,1)$} we can choose $k\ge 0$ large enough so that for all $n$ {large enough,} {with probability at least $1-\delta$},
 $$
 \sum_{i\ge k} (b_{i}^{\kappa_n})^2 a_{i}^{\kappa_n}  <\mu^2_{h^2}( \mathrm{Cut}^{4}_{\tau_{i_0,\vp}^{4}})/2.
 $$
 Thus, for all $n$ large enough, {with probability at least $1-\delta$},
 {we cannot have $\mu^{\sqrt{\kappa_n}}_{h^{\sqrt{\kappa_n}}} (D^{z,\kappa_n}_u) = (b_{i}^{\kappa_n})^2 a_{i}^{\kappa_n} $ for some $i \ge k$ due to \eqref{eq inegalite aire quantique cut}, so that there exists $i<k$ such that $\mu^{\sqrt{\kappa_n}}_{h^{\sqrt{\kappa_n}}} (D^{z,\kappa_n}_u) = (b_{i}^{\kappa_n})^2 a_{i}^{\kappa_n} $. In other words, since the $b_i^{\kappa_n}$'s are ranked in the non-increasing order, the quantum boundary length of $D^{z,\kappa_n}_u$ is at least $b_k^{\kappa_n}$.} Thus, {given that $b_k^{\kappa_n}$ converges a.s.\@ towards $b_k$ by \eqref{eq cv a et b}}, for $n$ large enough, with probabililty at least $1-\delta$, the quantum boundary length of $D^{z,\kappa_n}_u$ is larger than {$b_k/2$}.
\end{itemize}
In both cases,  {for all $\delta\in (0,1)$, there exists $\eta>0$ small enough such that for all $n\ge0$,} $\nu_{h^{\sqrt{\kappa_n}}_{D^{z,\kappa_n}_u}}^{\sqrt{\kappa_n}}(D^{z,\kappa_n}_u)\ge \eta$ {with probability at least $1-\delta$}. {In particular, for all $\delta \in (0,1)$, there exists $\eta>0$ small enough such that for all $n\ge 0$, with probability at least $1-\delta$, there exists a random integer $I_{\eta,u,n}\le N^{\kappa_n}_\eta(u)$ such that $D^{z,\kappa_n}_u= \mathrm{Cut}^{\kappa_n}_{\tau^{\kappa_n}_{\eta,I_{\eta,u,n}}}$. Then, by the convergence (\ref{cvDQ3}) of the previous step (applied to $\mathrm{Cut}^{\kappa_n}_{\tau^{\kappa_n}_{\eta,I_{\eta,u,n}}}$), we deduce that $D^{z,\kappa_n}_u$ is tight as a quantum disk. In particular, after taking another subsequence of $\kappa_n$ and applying Skorokhod's representation theorem once more, we may assume that $D^{z,\kappa_n}_u$ converges a.s.\@ in terms of quantum disks, i.e.\@ for the topology of the convergence (\ref{cvDQ3}).}

{Since we already know by \eqref{cvDQ} (first coordinate) that $\mathrm{Cut}^{4}_{\tau_{i_0,\vp}^{4}}$ is the a.s.\@ limit of $D^{z,\kappa_n}_u$ in the sense of Carathéodory,} we deduce that
$$
	D^{z,\kappa_n}_u
	\mathop{\longrightarrow}\limits_{n \to \infty}^{{(\mathrm{a.s.})}}
	\mathrm{Cut}^{4}_{\tau_{i_0,\vp}^{4}}
$$
in terms of quantum disks. Since every cut out or encircled domain $\mathrm{Cut}^{4}_{\tau_{i_0,\vp}^{4}}$ for {$i_0 \le N^4_{\vp}(u)$} is such a limit, we get that $N_\vp(u) \ge N^4_\vp(u)$. Moreover, since by the previous step the domain $\mathrm{Cut}^{\kappa_n}_{\tau_{i_0,\vp}^{\kappa_n}}$ converges in terms of quantum disks, and since the points of a Poisson point process of intensity $dx/x^2$ are a.s.\@ different from $\vp$, we obtain that $N_\vp(u) \le N^4_\vp(u)$ and hence $N_\vp(u) =N^4_\vp(u)$.
\end{proof}

\paragraph{Step 5.} Let us prove that the cut-out and encircled domains converge in the right order.
\begin{lemma}\label{lemme cv Cut n}
	We have a.s.
	\begin{equation}\label{cv Cut n}
		\left(\mathrm{Cut}^{\kappa_n}_{\tau^{\kappa_n}_{\vp, i}} \right)_{i \in \lb 1 , N^{\kappa_n}_\vp(u) \rb}
		\mathop{\longrightarrow}\limits_{n\to \infty}
		\left( \mathrm{Cut}^{4}_{\tau^{4}_{\vp, i}} \right)_{i \in \lb 1, N^4_\vp(u) \rb},
	\end{equation}
	for the Carathéodory topology and in terms of quantum disks.
\end{lemma}
\begin{proof}
By the two previous steps, {i.e.\@ by Lemmas \ref{lemme cvDQ3} and \ref{lemme identification N epsilon},} we deduce that we have the a.s.\@ convergence of the domains $(\mathrm{Cut}^{\kappa_n}_{\tau_{i,\vp}^{\kappa_n}})_{1 \le i \le N^{\kappa_n}_\vp}$ towards some reordering of the family $(\mathrm{Cut}^{4}_{\tau_{i,\vp}^{4}})_{1 \le i \le N^{4}_\vp}$ for the Carath\'eodory topology and in terms of quantum disks.

Then \eqref{cv Cut n} is easy to see by looking at rational times $u_i$'s for $i\ge 1$ such that $N^4_\vp(u_i)=i$ for all $i\le N^4_\vp(u)$: at time $u_1$ there is only one possibility for the cut out or encircled domain, at time $u_2$ we have already identified the limit of the first region to be encircled or cut out, etc. Note that the convergence (\ref{cv Cut n}) holds for any $\vp \in \R_+^* \cap \Q$.
\end{proof}

\paragraph{Step 6.} Here we prove the identity (\ref{transfo de Doob}) of Proposition \ref{absolue continuité SLE4}. 

\begin{proof}[Proof of (\ref{transfo de Doob}) in Proposition \ref{absolue continuité SLE4}] Recall that we have the a.s.\@ the convergences (\ref{cvCLE}), (\ref{cv disque quantique}), (\ref{cvDQ}), (\ref{eq cv a et b}), (\ref{eq cv N}), (\ref{cvDQ3}) and also (\ref{cv Cut n}) {by Lemma \ref{lemme cv Cut n}}. We may also assume that the convergence (\ref{cv Skorokhod}) also holds almost surely by taking another subsequence again denoted by $\kappa_n$.
By the definition of $(\overline{L}^4_t+ \overline{R}^4_t)_{t\ge 0}$ in Lemma \ref{lemme cv Skorokhod}, to prove (\ref{transfo de Doob}), it is enough to check that for all $t\ge0$, on the event $E_t$ the process $(S^{4,\mu}_s)_{0\le s \le t}$ has the same law as $(\overline{S}^4_s)_{ 0\le s \le t} \coloneqq (\overline{L}^4_s+ \overline{R}^4_s)_{ 0\le s \le t}$ on the event $\{ \forall s \in [0,t], \ \overline{L}^4_s, \overline{R}^4_s>0\}$, where $\overline{L}^4$ and $\overline{R}^4$  are defined in Lemma \ref{lemme cv Skorokhod}. 
	In what follows, we drop the exponents $\mu((4/\kappa_n)-1)$ and $\mu$ above $S^{\kappa_n,\mu((4/\kappa_n)-1)}$ and $S^{4,\mu}$ to simplify the notation. Note that for all $i_0\ge 1$ for $\vp>0$ small enough, for all $i\le i_0$ and $n \ge 1$, the time $\tau^{\kappa_n}_{\vp,i}$ is finite and converges in law as $\kappa \uparrow 4$ by (\ref{cv Skorokhod}). Moreover, for all $u \ge 0$, by definition of $N^{\kappa_n}_\vp(u)$, we obtain that
	 \begin{equation}\label{eq encadrement ql}
	 	\tau^{\kappa_n}_{\vp, N^{\kappa_n}_\vp(u)} \le \mathrm{ql}^{\kappa_n}(u) < \tau^{\kappa_n}_{\vp, N^{\kappa_n}_\vp(u)+1}.
	 \end{equation}
	Thus, by (\ref{eq cv N}) and (\ref{cv Skorokhod}) by letting $n\to \infty$ and then $\vp \to 0$ for all rational $u\ge 0$, there is a non-decreasing process $(\widetilde{\mathrm{ql}}(u))_{u \in \Q \cap \R_+}$ such that for all $u \in \Q \cap \R_+$,
	$$
	\mathrm{ql}^{\kappa_n}(u)
	\mathop{\longrightarrow}\limits_{n\to \infty}
	\widetilde{\mathrm{ql}}(u).
	$$
	Actually, by density of the jumps of $\overline{S}^4$, one can see that $\widetilde{\mathrm{ql}}$ extends to a continuous increasing function so that we have a.s.\@ uniformly on compact sets
	$$
	\mathrm{ql}^{\kappa_n}
	\mathop{\longrightarrow}\limits_{n\to \infty}
	\widetilde{\mathrm{ql}}.
	$$
	Then, by the definition of the Skorokhod $J_1$ convergence, the above convergence implies that
	$$
	(S^{\kappa_n}_{\mathrm{ql}^{\kappa_n}(u)})_{u\ge 0}
	\mathop{\longrightarrow}\limits_{n \to \infty}^{(\mathrm{a.s.})}
	(\overline{S}^4_{\widetilde{\mathrm{ql}}(u)})_{u\ge 0}
	$$
	in the sense of Skorokhod. In particular, since the Skorokhod convergence implies the convergence of the jumps of size at least $\vp$ for all $\vp >0$, we deduce that a.s.\@ for all $u \in \Q \cap \R_+$, for all $\vp>0$, by (\ref{cv Cut n}) we have
	$$
	N^4_{\vp,-}(u) \coloneqq 
	 \# \{
	v \in [0,u]; \ -\Delta {S}^4_{{\mathrm{ql}}^4(v)} \ge \vp
	\}
	= \# \{
	v \in [0,u]; \ -\Delta \overline{S}^4_{\widetilde{\mathrm{ql}}(v)} \ge \vp
	\}.
	$$
	Since the left and the right terms are càdlàg in $u$, the above equality holds for all $u \ge 0$. We can check that $\widetilde{\mathrm{ql}}= \mathrm{ql}^4$ using the caracterization \eqref{eq longueur du tronc critique} of the quantum length of the trunk as a limit in probability: for all $u>0$,
	\begin{align*}
		\mathrm{ql}^4(u) &=
		2\lim_{\vp \to 0}
		\vp \# \{v \in [0,u]; \ -\Delta S^4_{\mathrm{ql}^4(v)} \in [\vp,2\vp]\}
		\\
		&=
		2\lim_{\vp \to 0}
		\vp (N^4_{\vp,-}(u)- N^4_{2\vp,-}(u)) \\
		&=
		2\lim_{\vp \to 0}
		\vp \# \{t \in [0,\widetilde{\mathrm{ql}}(u)]; \ -\Delta \overline{S}^4_t \in [\vp,2\vp]\} \\
		&= \widetilde{\mathrm{ql}}(u).
	\end{align*}
	The last equality of limits in probability stems from the definition of $\overline{S}^4$ as a Doob $h$-transform of Cauchy processes. Thus  a.s.\@ $\widetilde{\mathrm{ql}}= \mathrm{ql}^4$. 
	{Therefore $\mathrm{ql}^{\kappa_n}$ converges towards $\mathrm{ql}^4$ uniformly on compact sets. Using the fact that $\mathrm{ql}^4$ is increasing, it is elementary to prove that $(\mathrm{ql}^{\kappa_n})^{-1}$ converges towards $(\mathrm{ql}^4)^{-1}$ uniformly on compact subsets of $[0,\lim_{v\to \infty} \mathrm{ql}^4(v))$. But, by the convergence of the unexplored region for the Carathéodory topology (see \eqref{eq cv region inexploree}): for any sequence $u_n \to u$ of non-negative (possibly random) real numbers, a.s.
		$$
		\mathcal{D}^{\kappa_n}_{\mathrm{ql}^{\kappa_n}(u_n)} \mathop{\longrightarrow}\limits_{n\to \infty}
		\mathcal{D}^{4}_{\mathrm{ql}^{4}(u)}.
		$$
		We obtain that for all $t \in [0,\lim_{v\to \infty} \mathrm{ql}^4(v))$, a.s.\@ we have the convergence in the sense of Carathéodory
		\begin{equation}\label{eq cv region inexploree quantique}\mathcal{D}^{\kappa_n}_t = \mathcal{D}^{\kappa_n}_{\mathrm{ql}^{\kappa_n}((\mathrm{ql}^{\kappa_n})^{-1}(t))}
		\mathop{\longrightarrow}\limits_{n\to \infty}
		\mathcal{D}^{4}_{\mathrm{ql}^{4}((\mathrm{ql}^{4})^{-1}(t))} = \mathcal{D}^{4}_t.
	\end{equation}
	Besides, since
	$$
	\nu^{\sqrt{\kappa_n}}_{h^{\sqrt{\kappa_n}}_{\vert \mathcal{D}^{\kappa_n}}}(\partial \mathcal{D}^{\kappa_n}_t) = S^{\kappa_n}_t \mathop{\longrightarrow}\limits_{n\to \infty} \overline{S}^4_t,
	$$
	and since for all $n$, $\mathcal{D}^{\kappa_n}_t$ is a quantum disk of quantum boundary length $S^{\kappa_n}_t$, we deduce that $\mathcal{D}^{\kappa_n}_t$ is tight in terms of quantum disk. Thus, after taking another subsequence and applying Skorokhod's representation theorem once again, $\mathcal{D}^{\kappa_n}_t$ converges towards $\mathcal{D}^4_t$ in terms of quantum disks a.s.\@ for all $t \in [0,\lim_{v\to \infty} \mathrm{ql}^4(v))$. In particular, $S^4_t=\overline{S}^4_t$ almost surely, hence (\ref{transfo de Doob}).}
	\end{proof}
	
\paragraph{Step 7.} Finally, let us prove the remaining part of Proposition \ref{absolue continuité SLE4}.

\begin{proof}[Proof of Proposition \ref{absolue continuité SLE4}: law of the cut out domains and unexplored region.] In order to prove the statement on the cut out domains and on the {region remaining to be explored}, it suffices to check that for all $\vp>0$, if $F$ is a continuous bounded function on $\mathbb{D}([0,t])$ and $G$ is a continuous bounded function on the space of finite sequences of distributions on $\mathbb{D}$ seen modulo the transformation (\ref{covariance}) (formally it can be written as $\bigcup_{n\ge 0} (\mathcal{D}'(\mathbb{D}))^n$ where $\mathcal{D}'(\mathbb{D})$ is the space of distributions on $\mathbb{D}$ modulo the transformation (\ref{covariance}) for all Möbius transformation $\varphi$), then
	$$
	\E\left( F((S^4_s)_{0\le s \le t}) G\left(
	h_t,
	(h_{\vp,i})_{i\ge 1, \tau^4_{\vp,i}<t}
	\right)
	\right) =
	\E\left( F((S^4_s)_{0\le s \le t})G\left(
	h(S^4_t),
	(h(\Delta S^4_{\tau^4_{\vp,i}}))_{i\ge 1, \tau^4_{\vp,i}<t}
	\right)
	\right),
	$$
	where a representative of $h_t$ is $h^2 \circ f_z +2 \log \vert f_z'\vert$ where $z \in \mathcal{D}_t^4$ and $f_z: \mathbb{D} \to \mathcal{D}^4_t$ is the unique conformal map such that $f_z(0)=z$ and $f_z'(0)>0$, where for all $i\ge 1$, we have defined $h_{\vp,i}$ as the equivalence class of $h^2 \circ f_{z_i} +2 \log \vert f_{z_i}'\vert$ where $z_i\in \mathrm{Cut}^4_{\tau_{\vp,i}^4}$ and $f_{z_i}: \mathbb{D} \to \mathrm{Cut}^4_{\tau_{\vp,i}^4}$ is the unique conformal map such that $f_{z_i}(0)=z_i$ and $f_{z_i}'(0)>0$, and where $h(b)$ denotes the instance of the GFF of an independent quantum disk of quantum boundary length $b$ (conditionally on $b$). But the above equality holds clearly thanks to the convergences (\ref{cv Skorokhod}),  (\ref{cv Cut n}), {\eqref{eq cv region inexploree quantique}} and thanks to the second point of Proposition \ref{Prop 5.1 MSW} which states the same equality for $\kappa<4$.
\end{proof}

\section{The $\mathrm{SLE}^{\langle \mu \rangle}_4(-2)$ converges to the uniform exploration as $\mu \to \infty$}\label{section mu vers l'infini}

Recall that we proved Theorem \ref{prop de Markov disques quantiques} in the case of the $\mathrm{SLE}^{\langle \mu \rangle}_4$ exploration for $\mu \in \R$ in Subsection \ref{sous-section kappa vers 4}. In this section, we prove Theorem \ref{prop de Markov disques quantiques} in the case of the uniform exploration, which was defined in Subsection \ref{sous-section explo unif}. To this end, we let $\mu \to \infty$ to approximate the uniform exploration using the $\mathrm{SLE}^{\langle \mu \rangle}_4(-2)$ exploration. In this section, for convenience, we will explore the nested $\mathrm{CLE}_4$.

\subsection{From the Carath\'eodory convergence to the Markov property for quantum disks}\label{sous-section de la cv Caratheodory a la prop de markov quantique unif}

The fact that the uniform exploration can be constructed by letting $\mu \to \infty$ was first observed in the last paragraph of p.22 of \cite{WW13}. However, we will need a convergence in terms of Carath\'eodory topology. Recall the radial $\mathrm{SLE}^{\langle \mu \rangle}_4(-2)$ defined at the end of Subsection \ref{sous-section rappels CLE}. For all $\mu \in \R$, for all $z \in \mathbb{D}$, let $({\bf D}^{z,\mu}_{s})_{s\ge 0}$  be the radial exploration targeted at $z$ using the $\mathrm{SLE}^{\langle \mu \rangle}_4 (-2)$. More precisely, for every time $s\ge 0$, the domain ${\bf D}_s^{z,\mu}$ is the unexplored region containing $z$ for the radial $\mathrm{SLE}^{\langle \mu \rangle}_4(-2)$. For convenience, we do not stop the exploration when the loop arround $z$ is discovered, hence discovering the loops of the nested $\mathrm{CLE}_4$. The parametrization is chosen so that if $f_s: \mathbb{D} \to {\bf D}^{z,\mu}_s$ is the unique conformal map such that $f_s(0)=z$ and $f_s'(0)>0$, then $f'_s(0)=e^{-s}$. The positive number $f'_s(0)$ is called the conformal radius of ${\bf D}^{z,\mu}_s$ seen from $z$. 
For all $w\in \mathbb{D} \cap \Q^2 \setminus \{z\}$, we denote by $D^{z,w,\mu}_s$ the connected component containing $w$ of the complement in $\mathbb{D}$ of the path of the chordal $\mathrm{SLE}^{\langle \mu \rangle}_4(-2)$ towards $z$ until time $s$. Note that there is a time $\sigma^\mu_{z,w}$ at which the radial $\mathrm{SLE}^{\langle \mu \rangle}_4(-2)$ towards $z$ cuts $w$ out from $z$ or encircles $w$ with a loop which does not encircle $z$, {or encircles $z$ and not $w$}. Therefore, for all $s <\sigma^\mu_{z,w}$, we have $D^{z,w,\mu}_s = {\bf D}^{z,\mu}_s$ and for all $s \ge \sigma^\mu_{z,w}$, the domain $D^{z,w,\mu}_s = { D}^{z,w,\mu}_{\sigma^\mu_{z,w}}$ corresponds to the domain which is cut out or encircled by a loop.
Similarly, using the uniform exploration of the nested $\mathrm{CLE}_4$ defined in Subsection \ref{sous-section explo unif}, let $({\bf D}^z_{s})_{s\ge 0}$ be the exploration targeted at $z$, again parametrized so that the conformal radius of ${\bf D}^z_s$ is $e^{-s}$. More precisely, ${\bf D}^z_{s}$ is the unexplored region containing $z$ at time $s$. Note that for all $w\in \mathbb{D} \cap \Q^2 \setminus \{z\}$ there is a time $\sigma_{z,w}$ at which the exploration cuts $w$ out from $z$ or encircles $w$ with a loop which does not encircle $z$. We then set $D^{z,w}_s= {\bf D}^z_s$ for all $s <\sigma_{z,w}$, and we define $D^{z,w}_s$ for $s\ge \sigma_{z,w}$ as the domain which is cut out or encircled by a loop at time $\sigma_{z,w}$.

\begin{proposition}\label{prop mu vers l'infini}
	One has the convergence in distribution of
	$$
	\left(
	({\bf D}^{z,\mu}_{s})_{s\ge 0, z \in \mathbb{D} \cap \Q^2},
	(D^{z,w,\mu}_s)_{s\ge 0, z \in \mathbb{D} \cap \Q^2, w \in \mathbb{D} \cap \Q^2 \setminus \{z\}}
	\right)
	$$
	to
	$$
	\left(
	({\bf D}^z_{s})_{s\ge 0,z \in \mathbb{D} \cap \Q^2},
	(D^{z,w}_s)_{s\ge 0, z \in \mathbb{D} \cap \Q^2, w \in \mathbb{D} \cap \Q^2 \setminus \{z\}}
	\right)
	$$
	as $\mu \to \infty$ in terms of finite dimensional distributions for the Carath\'eodory topology. The convergence holds also in a stronger sense: if $(s_\mu)_{\mu\ge 0}$ converges towards $s$ as $\mu \to \infty$, then jointly with the above convergence, for all $z \in \mathbb{D} \cap \Q^2$, we have 
	$
	{\bf D}^{z,\mu}_{s_\mu} 
	\mathop{\longrightarrow}
	{\bf D}^z_s
	$
	as $\mu \to \infty$ in distribution.
\end{proposition}
The proof of Theorem \ref{prop de Markov disques quantiques} for $(Y_t)_{t\ge 0}$ then follows from roughly the same lines as the proof of Proposition \ref{absolue continuité SLE4}, although here we deal with the radial exploration instead of the chordal one. Let us only give a sketch of the proof:
\begin{proof}[Sketch of the proof of Theorem \ref{prop de Markov disques quantiques} for $(Y_t)_{t\ge 0}$ using Proposition \ref{prop mu vers l'infini}]
For all $z \in \mathbb{D}$, for all $\mu \in \R$, let us write $\mathrm{ql}_{z,\mu}(s)$ the quantum length of the trunk of the radial exploration $({\bf D}^{z,\mu}_{s})_{s\ge 0}$. More precisely, it is characterized (with our choice of the multiplicative constant) by
$$
\mathrm{ql}_{z,\mu}(s)= 2 \lim_{\vp \to 0} \vp \# \{v \in [0,s]; -\Delta \nu^2_{h^2_{\vert {\bf D}^{z,\mu}_v}}(\partial{\bf D}^{z,\mu}_v) \in [\vp,2\vp]\}.
$$
Let us write $\mathrm{ql}_{z,\mu}^{-1}$ the right-continuous inverse of $\mathrm{ql}_{z,\mu}$ and for all $t\ge 0$,
$$
\widehat{S}^{z,\mu}_t \coloneqq \nu^2_{h^2_{\vert {\bf D}^{z,\mu}_{\mathrm{ql}_{z,\mu}^{-1}(t)}}}(\partial{\bf D}^{z,\mu}_{\mathrm{ql}_{z,\mu}^{-1}(t)}).
$$
Let $Z_i$ for $i \ge 1$ be i.i.d.\@ random points of $\mathbb{D}$ of law $\mu^2_{h^2}/\mu^2_{h^2}(\mathbb{D})$ which are independent of $h^2$ and of the explorations. 
One can see using Theorem \ref{prop de Markov disques quantiques} for $(Y^\mu_t)_{t\ge 0}$, the relation between the quantum area and the quantum boundary length and the renewal property {(Proposition \ref{prop renewal et Markov conforme})} that the process $(\widehat{S}^{Z_i,\mu}_t)_{t\ge 0, i\ge 1}$ has the same distribution for all $\mu \in \R$ 
and that for all $i \ge 1$, conditionally on their quantum boundary lengths, ${\bf D}^{Z_i,\mu}_{\mathrm{ql}^{-1}_{Z_i,\mu}(t)}$ is chosen at random among the quantum disks corresponding to the unexplored regions of the exploration tree at height $t$. More precisely, if we denote by $a_j^\mu(t)$ for $j \ge 1$ the quantum areas of the unexplored quantum disks when we stop the exploration tree at height $t$ ranked in the non-increasing order, then conditionally on the exploration tree at height $t$ and on the $a^\mu_j$'s, the quantum surface ${\bf D}^{Z_i,\mu}_{\mathrm{ql}^{-1}_{Z_i,\mu}(t)}$ corresponds to the quantum disk of area $a_j^\mu(t)$ with probability $a_j^\mu(t)/(\sum_{k\ge 1} a_k^\mu (t))$. In particular, the quantum surface ${\bf D}^{Z_i,\mu}_{\mathrm{ql}^{-1}_{Z_i,\mu}(t)}$ is constant in distribution when $\mu$ varies.

 By the same reasoning as in the proof of Lemma \ref{lemme cvDQ}, we know that the convergence in Proposition \ref{prop mu vers l'infini} also holds for the underlying fields in terms of the convergence in the sense of distribution and for the quantum area measures with the weak convergence of measures. 
 
 Using that $\mu^2_{h^2}(\mathbb{D})<\infty$, it is easy to see that the number of encircled or cut out disks of quantum boundary length larger than some $\vp >0$ during the exploration ${\bf D}^{Z_i,\mu}$ until time $s$ is tight, {as in Lemma \ref{lemme N epsilon est tendu}}. As a result, using that $\widehat{S}^{Z_i,\mu}$ is constant in law, we obtain as in (\ref{eq encadrement ql}) that for all $u \in \Q \cap \R_+$, the family $(\mathrm{ql}_{Z_i, \mu}(u))_{\mu \ge 0}$ is tight. {But note that $\mathrm{ql}_{Z_i, \mu}$ is a non-decreasing function. Moreover, by density of the jumps of $\widehat{S}^{Z_i,\mu}$, we get that $\mathrm{ql}_{Z_i, \mu}$ is continuous. As a result, $\mathrm{ql}_{Z_i, \mu}$} is tight as $\mu \to \infty$ for the topology of the uniform convergence on compact sets and that the limit must be continuous and non-decreasing.
 
 Let $\mu_n \to \infty$. By taking a subsequence again denoted by $\mu_n$ and by applying Skorokhod's theorem, we may assume that the convergence of Proposition \ref{prop mu vers l'infini} holds a.s.\@ together with the convergences for all $i \ge 1$,
 $$
 \widehat{S}^{Z_i,\mu_n}
 \mathop{\longrightarrow}\limits_{n \to \infty}
 \overline{S}^{Z_i}
 \qquad
 \mathrm{ql}_{Z_i,\mu_n}
 \mathop{\longrightarrow}\limits_{n \to \infty}
 \widetilde{\mathrm{ql}}_{Z_i},
 $$
 where the first convergence holds for the $J_1$ Skorokhod topology and the second one holds uniformly on compact sets, where $\overline{S}^{Z_i}$ is a càdlàg process which has the same law as $\widehat{S}^{Z_i,\mu_n}$ for all $n$ and where $\widetilde{\mathrm{ql}}_{Z_i}$ is a continuous non-decreasing process from $\R_+$ to $\R_+$. Let $\widetilde{\mathrm{ql}}^{-1}_{Z_i}$ be its right-continuous inverse, which is increasing {since $\widetilde{\mathrm{ql}}_{Z_i}$ is continuous}.
 
 {By the convergence of $\mathrm{ql}_{Z_i,\mu_n}$ to $\widetilde{\mathrm{ql}}_{Z_i}$, we deduce that a.s.\@ for all $t \in \R_+$ such that $\widetilde{\mathrm{ql}}^{-1}_{Z_i}$ is continuous at $t$, the random variable $\mathrm{ql}^{-1}_{Z_i,\mu_n}(t)$ converges towards $\widetilde{\mathrm{ql}}^{-1}_{Z_i}(t)$ as $n\to \infty$.} As a consequence, by the convergence of Proposition \ref{prop mu vers l'infini}, we also have a.s.\@ for all $t \in \R_+$ such that $\widetilde{\mathrm{ql}}^{-1}_{Z_i}$ is continuous at $t$,
 $$
 {\bf D}^{Z_i,\mu_n}_{\mathrm{ql}^{-1}_{Z_i,\mu_n}(t)}
 \mathop{\longrightarrow}\limits_{n\to \infty}
 {\bf D}^{Z_i}_{\widetilde{\mathrm{ql}}^{-1}_{Z_i}(t)}
 $$
 in the sense of Carath\'eodory, and thus also in terms of quantum disks, since it is constant in law {as a quantum surface}. Therefore, {a.s.\@} for all $t \in \R_+$ such that $\widetilde{\mathrm{ql}}^{-1}_{Z_i}$ is continuous at $t$,
 $$
 \widehat{S}^{Z_i,\mu_n}_t \mathop{\longrightarrow}\limits_{n\to \infty}
 \nu^2_{h^2_{\vert  {\bf D}^{Z_i}_{\widetilde{\mathrm{ql}}^{-1}_{Z_i}(t) }}} \left(
 \partial {\bf D}^{Z_i}_{\widetilde{\mathrm{ql}}^{-1}_{Z_i}(t)  }
 \right).
 $$
 Hence for all $t \in \R_+$ such that $\widetilde{\mathrm{ql}}^{-1}_{Z_i}$ is continuous at $t$,
 \begin{equation}\label{eq nu egal S barre}
 \nu^2_{h^2_{\vert  {\bf D}^{Z_i}_{\widetilde{\mathrm{ql}}^{-1}_{Z_i}(t) }}} \left(
 \partial {\bf D}^{Z_i}_{\widetilde{\mathrm{ql}}^{-1}_{Z_i}(t)  }
 \right) = \overline{S}^{Z_i}_t,
 \end{equation}
 and the equality holds in fact for all $t\ge 0$ since both processes are càdlàg.

  We see that the quantum length $\mathrm{ql}_{z,\mu_n}$ converges to the quantum natural distance as $n \to \infty$ using the definition of the quantum natural distance $\mathrm{d}_{\mathrm{q}}(\partial \mathbb{D}, {\bf D}^z_s)$ as the limit in probability {(which was given in \eqref{eq distance quantique target})}
$$
\mathrm{d}_{\mathrm{q}}(\partial \mathbb{D}, {\bf D}^z_s) = 
2 \lim_{\vp \to 0} \vp \#\{ v\in [0,s], \ - \Delta \nu^2_{h^2_{\vert {\bf D}^z_v}} (\partial {\bf D}^z_v)  \in [\vp,2\vp]\}.
$$
The above definition is indeed consistent with the definition \eqref{eq distance quantique target} since the Cauchy process $\widetilde{S}^2$ in Proposition \ref{absolue continuité SLE4} is symmetric. From the above formula, it is easy to show, as in the end of the sixth step of the proof of Proposition \ref{prop mu vers l'infini}, that for all $s \ge 0$, we have $\mathrm{d}_{\mathrm{q}} (\partial \mathbb{D}, {\bf D}^{Z_i}_s) = \widetilde{\mathrm{ql}}_{Z_i}(s)$. { Indeed, using the fact that $\overline{S}^{Z_i}$ has the same law as $\widehat{S}^{Z_i, \mu}$ and the definition of $\widehat{S}^{Z_i, \mu}$, one has the convergence in probability
\begin{align*}
	\widetilde{\mathrm{ql}}_{Z_i}(s)&=2 \lim_{\vp \to 0} \vp \# \{ r\in [0,\widetilde{\mathrm{ql}}_{Z_i}(s)], \ -\Delta \overline{S}^{Z_i}_r \in [\vp,2 \vp] \}\\
	&= 2 \lim_{\vp \to 0} \vp \# \{v \in [0,s], \ -\Delta \nu^2_{h^2_{\vert {\bf D}^{Z_i}_v}} (\partial {\bf D}^{Z_i}_v) \in [\vp, 2\vp]\} \qquad \qquad \text{by \eqref{eq nu egal S barre}} \\
	&= \mathrm{d}_{\mathrm{q}}(\partial \mathbb{D}, {\bf D}^{Z_i}_s).
\end{align*}}
For all $z \in \mathbb{D}$, if $\Psi_z$ is the inverse of $s \mapsto \mathrm{d}_{\mathrm{q}}(\partial \mathbb{D}, {\bf D}^z_s)$, then we set for all $t\ge 0$,
$$
\widehat{S}^z_t \coloneqq \nu^2_{h^2_{\vert {\bf D}^{z}_{\Psi_z(t)}}}(\partial{\bf D}^{z}_{\Psi_z(t)}).
$$
By (\ref{eq nu egal S barre}) and since for all $s \ge 0$, we have $\mathrm{d}_{\mathrm{q}} (\partial \mathbb{D}, {\bf D}^{Z_i}_s) = \widetilde{\mathrm{ql}}_{Z_i}(s)$, we deduce that $\overline{S}^{Z_i}= \widehat{S}^{Z_i}$ and thus a.s.
\begin{equation}\label{eq cv taille de la region autour de z}
(\widehat{S}^{Z_i,\mu_n}_t)_{t \ge 0}
\mathop{\longrightarrow}\limits_{n \to \infty} (\widehat{S}^{Z_i}_t)_{t\ge 0}
\end{equation}
in the sense of Skorokhod.

Next, we consider the locally largest component of the $\mathrm{SLE}^{\langle \mu_n \rangle}_4(-2)$ exploration parametrized by the quantum length of the trunk, that we denote by $(\mathcal{C}^{\mu_n}_t)_{t\ge 0}$. Recall that its quantum boundary length is denoted by $(Y^{\mu_n}_t)_{t\ge 0}$. By Theorem \ref{prop de Markov disques quantiques} for $Y^{\mu_n}$, we know that the law of $(Y^{\mu_n}_t)_{t\ge 0}$ does not depend on $n$ and that conditionally on $(Y^{\mu_n}_{t'})_{0\le t' \le t}$, the cut out or encircled domains, together with $\mathcal{C}_t^{\mu_n}$, are independent $2$-quantum disks of boundary lengths corresponding to the negative and positive jumps of $(Y^{\mu_n}_{t'})_{0\le t' \le t}$ and to $Y^{\mu_n}_t$.

Let us recall from Subsection \ref{sous-section composante localement la plus grande} that the locally largest component along the uniform exploration of the $\mathrm{CLE}_4$ parametrized by the quantum natural distance is denoted by $(C_{\sigma(t)})_{t\ge 0}$ and that its quantum boundary length is denoted by $(\mathcal{Y}(\sigma(t)))_{t\ge 0}$. We set $\mathcal{C}_t\coloneqq C_{\sigma(t)}$ and {recall that} $Y_t= \mathcal{Y}(\sigma(t))$ for all $t\ge 0$.

Let $t\ge 0$.  Let $i$ such that $Z_i \in \mathcal{C}_t$. Then $\mathcal{C}_t = {\bf D}^{Z_i}_{\Psi_{Z_i}(t)}$. Therefore, for all $t' \in [0,t]$, we have
\begin{equation}\label{eq egalite composante localement la plus grande et target explo unif}
{\bf D}^{Z_i}_{\Psi_{Z_i}(t')} = \mathcal{C}_{t'} \qquad \text{and} \qquad \widehat{S}^{Z_i}_{t'} = Y_{t'}.
\end{equation}
Besides, by the Skorokhod convergence (\ref{eq cv taille de la region autour de z}), we know that for ${n}$ large enough, the process $(\widehat{S}^{{Z_i},\mu_n}_{t'})_{0\le t' \le t}$ does not have any negative jump at {any} time $t'\le t$ of size at least half of $\widehat{S}^{{Z_i},\mu_n}_{t'-}$. In other words, for ${n}$ large enough, for all $t' \in [0,t]$,
\begin{equation}\label{eq egalite composante localement la plus grande et target SLE mu}
{\bf D}^{{Z_i},\mu_n}_{\mathrm{ql}_{{Z_i},\mu_n}^{-1}(t')} =\mathcal{C}^{\mu_n}_{t'}\qquad \text{and} \qquad
\widehat{S}^{{Z_i},\mu_n}_{t'} = Y^{\mu_n}_{t'}.
\end{equation}
{Therefore, from \eqref{eq cv taille de la region autour de z} we deduce that $Y^{\mu_n}$ converges almost surely towards $Y$ for the $J_1$ topology of Skorokhod. In particular, $Y$ has the same law as $Y^{\mu_n}$. We obtain the law of the encircled, cut out and unexplored region} by reasoning again as in {Step 7} of the proof of Proposition \ref{absolue continuité SLE4} {using the fact that \eqref{eq egalite composante localement la plus grande et target SLE mu} and \eqref{eq egalite composante localement la plus grande et target explo unif} together with the convergence of the cut out and encircled regions stated in Proposition \ref{prop mu vers l'infini} entail the convergence of the cut out and encircled regions of quantum boundary length at least $\vp>0$ and the convergence of the unexplored region along the exploration of the locally largest component}.
\end{proof}

{\paragraph{Outline of the proof of Proposition \ref{prop mu vers l'infini}.} Let us give a short outline of the proof of Proposition \ref{prop mu vers l'infini}, which is performed in the next two subsections. Intuitively, the idea is that as $\mu \to \infty$ the small rotations at each time the exploration is not drawing a loop will be more and more important as $\mu \to \infty$, in a way that for each time the exploration starts to draw a loop, ``the starting point looks like a point sampled uniformly at random on the boundary of the unexplored region''. We first prove in Subsection \ref{sous-section cv mu vers l infini une branche} the convergence of one branch of the exploration. By target-invariance, it suffices to prove the convergence of the branch towards the origin. This exploration is obtained using the same ideas as in Section 2 of \cite{AHPS21}: we approximate the explorations by focusing on the ``macroscopic loops''. Then, between each time of appearance of a macroscopic loop, we look at the impact of the ``rotations'' on the starting point of the next macroscopic loop. This starting point ``becomes uniform'' as $\mu \to \infty$. Next, in Subsection \ref{sous-section convergence jointe des branches quand mu tend vers l infini}, we prove that the convergence of the exploration branches towards each point $z \in \mathbb{D} \cap \Q^2$ holds jointly, relying intuitively on the fact that ``they draw the same loops'': the loops of a CLE$_4$.}
\subsection{Convergence of the branch towards the origin}\label{sous-section cv mu vers l infini une branche}
The rest of this section is devoted to the proof of Proposition \ref{prop mu vers l'infini} and follows the same lines as Section 2 of \cite{AHPS21}. For this purpose, we recall a stronger topology on the space of families of domains as in \cite{AHPS21} and \cite{MS16}. Let $\mathcal{D}$ be the space of families $(D_s)_{s \ge 0}$ of simply connected subdomains of $\mathbb{D}$ which are increasing for the inclusion, such that for every $s\ge 0$, we have $0 \in D_s$, and if $f_s=f[D_s]$ is the unique conformal map from $\mathbb{D}$ to $D_s$ which sends $0$ to $0$ and such that $f'_s(0)>0$, then $f'_s(0)= e^{-s}$. The space $\mathcal{D}$ is endowed with the topology such that a sequence of domains $(D^n_s)_{s\ge 0} \in \mathcal{D}$ converges to $(D_s)_{s\ge 0}$ if and only if for all compact subset $K\subset \mathbb{D}$, for all $T>0$,
$$
\sup_{s \in [0,T]} \sup_{z \in K} \left\lvert f^n_s(z)-f_s(z) \right\rvert \mathop{\longrightarrow}\limits_{n\to \infty} 0,
$$
where $f^n_s=f[D^n_s]$ and $f_s = f[D_s]$. The space $\mathcal{D}$ is then metrizable and separable. We define a distance $d_{\mathcal{D}}$ on $\mathcal{D}$ by setting for all $D,\widetilde{D} \in \mathcal{D}$,
\begin{equation}\label{eq distance sur D}
d_{\mathcal{D}}((D_s)_{s\ge 0},(\widetilde{D}_s)_{s\ge 0})
\coloneqq \sum_{j,k\in \N} 2^{-j-k} \sup_{s \in [0,2^k]} \sup_{z \in (1-1/j) \mathbb{D}} \vert f_s(z)-\widetilde{f}_s(z) \vert,
\end{equation}
where $f_s=f[D_s]$ and $\widetilde{f}_s=f[\widetilde{D}_s]$ for all $s\ge 0$. Note that the convergence of $(D^n_s)_{s\ge 0}$ to $(D_s)_{s\ge0}$ in $\mathcal{D}$ implies the convergence of $D^n_s$ towards $D_s$ in the Carath\'eodory topology for every $s\ge 0$. 
As in \cite{AHPS21}, we will sometimes consider evolving domains $(D_s)_{s\ge 0}$ such that the conformal radius of $D_s$ is $e^{-s}$ for all $s \in [0,S]$ for some finite $S>0$ and such that for all $s \ge S$, $D_s=D_S$. In order to see $(D_s)_{s\ge 0}$ as an element of $\mathcal{D}$, we replace $D_s$ for $s\ge S$ by $f_S(e^{-(s-S)}\mathbb{D})$ where $f_S= f[D_S]$.

Let us recall the definition of (measure-driven) radial Loewner chains. If $\lambda$ is a measure on $[0,\infty) \times \partial \mathbb{D}$ whose marginal on $[0,\infty)$ is the Lebesgue measure, we define the radial Loewner equation driven by $\lambda$ by setting for all $ z \in \mathbb{D}$ and for all $s\ge 0$,
$$
g_s(z) = \int_{[0,s] \times \partial \mathbb{D}} g_u(z) \frac{w+g_u(z)}{w-g_u(z)} d\lambda(u,w)+z.
$$
It is known (see Proposition 6.1 of \cite{MS16}) that for any such $\lambda$, this ODE has a unique solution $s \mapsto g_s(z)$ for each $z \in \mathbb{D}$, defined until time $s_z\coloneqq \sup \{ s \ge 0; \ g_s(z) \in \mathbb{D}\}$. Moreover, if one defines $D_s = \{ z \in \mathbb{D}; \ s_z < s\}$ for all $s\ge 0$, then $(D_s)_{s\ge 0}$ is called the radial Loewner chain driven by $\lambda$. 

If one restricts to a measure of the form $\lambda(ds,A) = \delta_{W(s)}(A)ds$ with $W: \R_+ \rightarrow \partial \mathbb{D}$ a piecewise continuous function, then one recovers the more classical notion of Loewner chains.

\begin{remark}\label{remarque cv pilote}(Remark 2.1 of \cite{AHPS21})
Using Proposition 6.1 of \cite{MS16}, we know that the weak convergence of the driving measure implies the convergence in $\mathcal{D}$ of the radial Loewner chain. 
In particular, the convergence of the radial Loewner chain holds if we assume the  convergence in the Skorokhod $J_1$ topology of the driving functions.
\end{remark}

Let us then recall the direct construction of the radial $\mathrm{SLE}^{\langle\mu\rangle}_4(-2)$ (see \cite{AHPS21} Subsection 2.1.3 for the analogous definition for $\kappa'>4$ and see Subsection 4.1 of \cite{ASW19} for the case $\mu=0$). Let $(B_s)_{s\ge 0}$ be a standard Brownian motion and let {$(\ell^{0}_s)_{s\ge 0}$ be its local time at $\pi \Z$.} Set, for all $s\ge 0$,
$$
U_s=  \int_0^s \cot(B_u) du + \mu \ell^{0}_s,
$$
where the integral is not absolutely convergent but can be defined as the limit in probability of integrals of the form $\int_0^s \cot(B_u){\bf 1}_{d(B_u,\pi \Z)\ge \vp} du$ as $\vp \to 0$, in the same way as in \cite{WW13} in the chordal case. 
Let $\left(W_0^\mu\right)_s = e^{i(2B_s -U_s)}$. Then, if one considers the radial Loewner chain $({\bf D}^{0,\mu}_s)_{s\ge 0}$ in $\mathbb{D}$ driven by $W^\mu_0$, it corresponds to the radial $\mathrm{SLE}^{\langle\mu\rangle}_4(-2)$ targeted at zero (and starting from $1$). Let $\tau_0$ be the first time that $B$ reaches $\pm \pi$, which corresponds to the first time that $0$ gets encircled by a loop during the $\mathrm{SLE}^{\langle\mu\rangle}_4(-2)$ exploration (see Subsection 4.1 of \cite{ASW19} and Subsection 4.3 of \cite{S09} for more details). Let 
$$
\left({ D}^{0,\mu}_s\right)_{s\ge 0} = \left({\bf D}^{0,\mu}_{s\wedge \tau_0}\right)_{s\ge 0}.
$$
We also recall that, during each excursion of $B$ away from $\pi \Z$, the radial Loewner chain draws a loop of the nested $\mathrm{CLE}_4$ which does not contain $0$ (see again for example Subsection 4.1 of \cite{ASW19} in the less general case $\mu=0$ and \cite{S09} for more details). See Figure \ref{image explo radiale}.

\begin{figure}[h]
   \centering
   \includegraphics[scale=0.53]{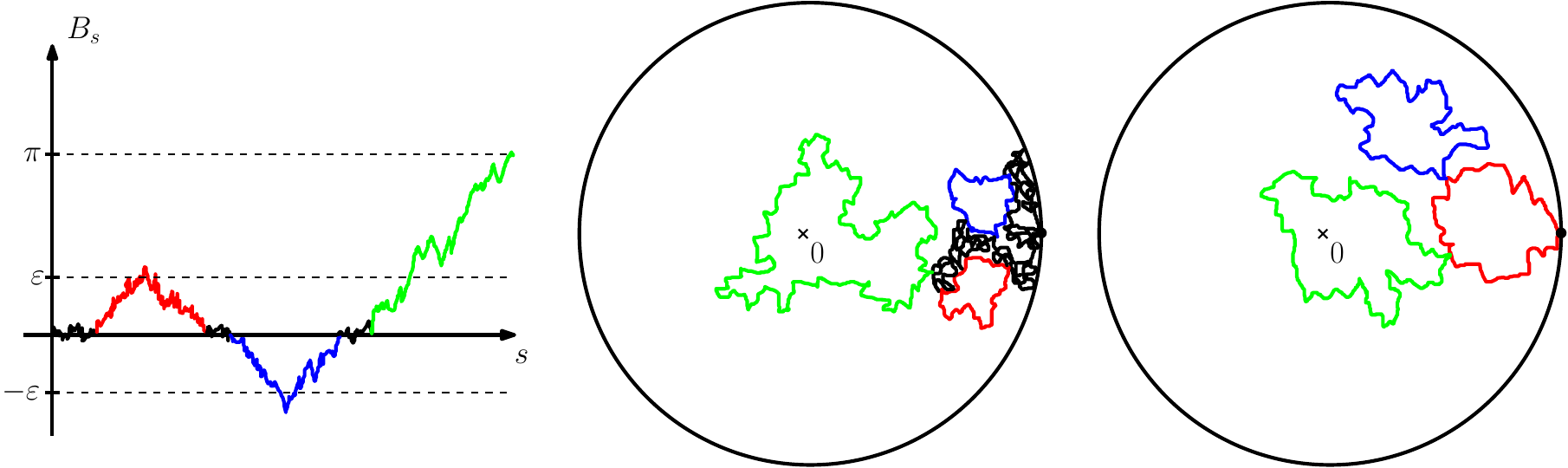}
   \caption{Left: the Brownian motion $(B_s)_{s\ge 0}$. Center: a radial exploration of the $\mathrm{CLE}_4$. The red and blue excursions of height at least $\vp$ give rise to two loops which do not encircle zero and the green part of the Brownian motion generates a loop which encircles zero. Right: approximation of the radial exploration. }
   \label{image explo radiale}
\end{figure}

Let us present an approximation of this radial $\mathrm{SLE}_4^{\langle \mu \rangle}(-2)$. Let $n \in \mathbb{N}$. The idea is to remove intervals of time where $B$ is making tiny excursions away from zero. We set $T_0^n\coloneqq0$ and for all $k\ge 0$,
$$
\left\{
\begin{array}{cl}
	R_{k+1}^n & = \inf \{ s\ge T_k^n; \ \lvert B_s \rvert \ge 1/2^n\}; \\
	S^n_{k+1} &= \sup \{s\le R^n_{k+1}; \ B_s=0\} ; \\
	T^n_{k+1} &= \inf \{ s\ge R^n_{k+1}; \ B_s =0\}.
\end{array}
\right.
$$
Call $e^n_k$ the $k$-th excursion associated to the interval $[S^n_k, T^n_k]$. Set $\Lambda^n \coloneqq \sup \{k; S^n_k \le \tau_0\}$ and
$$
l_k^n = T^n_k-S^n_k \ \forall k < \Lambda^n\ ; \ l^n_{\Lambda^n} = \tau_0-S^n_{\Lambda^n} \ ; \  L_k^n = \sum_{1\le j \le k} l^n_j \ \forall k \le \Lambda^n.
$$
Now we define for all $k \in \lb 1,  \Lambda^n \rb$,
$$
\forall k \in \lb 1,  \Lambda^n \rb, \ \forall s \in [ L^n_{k-1}, L^n_k), \qquad (W^{\mu,n}_0)_s= (W^{\mu}_0)_{S^n_k + (s-L^n_{k-1})},
$$
and set $D^{0,\mu, n}$ to be the radial Loewner chain with driving function $W^{\mu,n}_0$. This is defined up to time $\tau^{n}_0 = L^n_{\Lambda^n}$.

Let us recall an approximation of the uniform exploration of the nested $\mathrm{CLE}_4$. Sample $(X_k^n)_{k\ge 1}$ uniformly and independently on $\partial \mathbb{D}$. For all $s \in [0,\tau^n_0)$, set 
\begin{equation}\label{approximation exploration uniforme}
	\forall k \in \lb 1, \Lambda^n \rb, \ 
	\forall s \in [L^n_{k-1}, L^n_k), \qquad
	(W^n_0)_s= X^n_k \exp \left(
	i\left(2B_{s+S^n_k} -
	\int_{S_k^n}^{s+S_k^n} \cot(B_u) du
	\right)
	\right).
\end{equation}
{Recall  from Subsection \ref{sous-section de la cv Caratheodory a la prop de markov quantique unif} that $({\bf D}^0_{s})_{s\ge 0}$ is the uniform exploration targeted at $0$, again parametrized so that the conformal radius of ${\bf D}^0_s$ is $e^{-s}$. More precisely, ${\bf D}^0_{s}$ is the unexplored region containing $z$ at time $s$.} Let $D^{0,n}$ be the radial Loewner chain with driving function $W^n_0$. By the same argument as in Section 4 of  \cite{WW13} we know that $D^{0,n}$ converges in $\mathcal{D}$ in probability towards a uniform exploration $D^0$ targeted at $0$ in $\mathcal{D}$ and stopped at the time $\tau_0$ defined by $D^0_s= {\bf D}^0_{s \wedge \tau_0}$ for all $s \ge 0$
. Note that the time $\tau_0$ thus also corresponds to the first time at which $0$ gets encircled by a loop for the uniform exploration ${\bf D}^0$. Moreover, as in the beginning of Section 4 of \cite{WW13}, each loop drawn by ${\bf D}^0$ again corresponds to the excursion of $B$ away from $\pi \Z$ at the same time.

Let us show the convergence in law jointly in $z \in \Q^2 \cap \mathbb{D}$ as $\mu \to \infty$ of the $\mathrm{SLE}^{\langle\mu \rangle}_4(-2)$ branch towards $z$ to the uniform exploration branch towards $z$, together with the convergence of the $\mathrm{CLE}_4$ loops surrounding $z$ to the $\mathrm{CLE}_4$ loops surrounding $z$. We first show the convergence of the branch $({\bf D}^{0,\mu}_s)_{s\ge 0}$ targeted at $0$. To show this, it suffices to check the following statement:

\begin{proposition}\label{cv branche stoppee}
	The stopped exploration branch towards zero converges:
	$$
	(
	D^{0,\mu}
	, \tau_0,B)
	\mathop{\longrightarrow} \limits_{\mu \to \infty}^{(\mathrm{d})}
	(
	D^{0}
	, \tau_0,B),
	$$
	where the first coordinate converges in the space $\mathcal{D}$ and where $B$ is the Brownian motion from which the $D^{0,\mu}$'s are defined (and also $D^{0}$ via the approximation).
\end{proposition}

By the iterative definition of the exploration of the nested $\mathrm{CLE}_4$ towards zero, the convergence for all time follows immediately from Proposition \ref{cv branche stoppee}:

\begin{proposition}\label{prop cv branche vers zero}
	The whole exploration branch towards zero converges:
	$$
	({\bf D}^{0,\mu}, \tau_0,B)
	\mathop{\longrightarrow} \limits_{\mu \to \infty}^{(\mathrm{d})}
	({\bf D}^0, \tau_0,B),
	$$
	where the first coordinate converges in $\mathcal{D}$ and $B$ is the Brownian motion with which the $D^{0,\mu}$'s are defined. 		Notice that we also get in particular the convergence in the sense of Carathéodory of the domain ${\bf D}^{0,\mu}_{\tau_0}$ encircled by the outermost loop encircling $0$ towards ${\bf D}^{0}_{\tau_0}$.
\end{proposition}

For all $z \in \mathbb{D}\cap \Q^2$, we can define the space $\mathcal{D}_z$ of families of domains containing $z$ which are in $\mathcal{D}$ after applying the conformal map $w \mapsto (w-z)/(1-\bar{z}w)$, equipped by the topology induced by the topology of $\mathcal{D}$. We also denote by $\tau_z^\mu$ (resp. $\tau_z$) the first time at which $z$ is encircled by a loop in the radial $\mathrm{SLE}^{\langle \mu \rangle}_4(-2)$ targeted at $z$ (resp. in the uniform exploration $({\bf D}^z_s)_{s\ge 0}$). We also denote by $B^{z,\mu}$ (resp. $B^z$) the Brownian motion which defines $D^{z,\mu}$ (resp. $D^z$) composed with the conformal map $w \mapsto (w-z)/(1-\bar{z}w)$. The time $\tau^\mu_z$ (resp. $\tau_z$) is then the first time at which $B^{z,\mu}$ (resp. $B^z$) reaches $\pm \pi$. Note that as in the case $z=0$, we can build a coupling between the explorations using the same underlying Brownian motion so that $B^{z,\mu}=B^z$ and $\tau^\mu_z=\tau_z$ for all $\mu \in \R$. The convergence of the other branches then follows by conformal invariance as stated just below. We endow the space of continuous real functions $\mathcal{C}(\R_+,\R)$ with the topology of the uniform convergence on compact sets.

\begin{corollary}\label{cor cv branche z}
For all $z \in \mathbb{D}\cap \Q^2$,
\begin{equation}\label{eq cv branche vers z}
({\bf D}^{z,\mu}, \tau^\mu_z, B^{z,\mu}) 
\mathop{\longrightarrow}\limits_{\mu \to \infty}^{(\mathrm{d})}
({\bf D}^{z}, \tau_z, B^z) \qquad \text{in } \mathcal{D}_z \times \R \times \mathcal{C}(\R_+,\R),
\end{equation}
Moreover, we also have the convergence of the domain encircled by the outermost loop $D^{z,\mu}_{\tau^\mu_z}$ containing $z$ towards $D^z_{\tau_z}$ in the sense of Carath\'eodory jointly with the above convergence.
\end{corollary}

Exactly as in \cite{AHPS21}, our proof of Proposition \ref{cv branche stoppee} goes through the approximations $D^{0,\mu,n}$ and $D^{0,n}$. Recall the definition of the distance $d_\mathcal{D}$ in \eqref{eq distance sur D}. {The following lemma states that the approximation $D^{0,\mu,n}$ converges to $D^{0,\mu}$ as $n\to \infty$ uniformly in $n$. Its proof relies on Remark \ref{remarque cv pilote} and on the convergence of the driving functions.}
\begin{lemma}\label{lemme approximation uniforme}
	We have, a.s.\@ $\tau^n_0 \rightarrow \tau_0$ as $n\to \infty$ and
	$$
	\sup_{\mu \in \R} d_{\mathcal{D}}(D^{0,\mu, n}, D^{0,\mu}) \mathop{\longrightarrow}\limits_{n \to \infty}^\text{a.s} 0.
	$$
\end{lemma}

\begin{proof}
	The convergence of $ \tau^n_0$ stems from elementary properties of the Brownian motion. To prove the uniform convergence in $\mathcal{D}$, let us define the driving functions using the same Brownian motion $B$. Let $T>0$. Let us check that if $F: \partial \mathbb{D} \times [0,T] \to \R$ is a continuous function, bounded in absolute value by some constant $M>0$, then a.s.
	$$\int_0^T F((W^{\mu,n}_0)_{s\wedge \tau_0^n}, s )ds
	\mathop{\longrightarrow}\limits_{n\to \infty}
	\int_0^T F((W^{\mu}_0)_{s\wedge \tau_0}, s )ds,
	$$
	uniformly on $\mu \in \R$. This follows from the convergence of $ \tau^n_0$ since
	\begin{align*}
		\left\lvert
		\int_0^T F((W^{\mu,n}_0)_{s\wedge \tau_0^n}, s )ds
		-
		\int_0^T F((W^{\mu}_0)_{s\wedge \tau_0}, s)ds
		\right\rvert
		\le 2M(\tau_0-\tau^n_0).
	\end{align*}
	By Remark \ref{remarque cv pilote}, this gives the convergence for all $\mu \in \R$.
	To obtain the uniform convergence, we know by Remark \ref{remarque cv pilote} that the function which maps a measure $\lambda$ on $[0,\infty) \times \partial \mathbb{D}$ whose marginal on $[0,\infty)$ is the Lebesgue measure and whose marginal on $\partial \mathbb{D}$ is a probability measure to the corresponding element of $\mathcal{D}$ is continuous. Moreover, the space of such measures restricted to $\partial \mathbb{D} \times [0,T]$ is compact by compactness of $\partial \mathbb{D} \times [0,T]$. This enables to get the uniform convergence for the distance $d_{\mathcal{D}}$.
\end{proof}
{The lemma below is the key ingredient in the convergence towards the uniform exploration. It states the convergence of the approximation of the SLE$^{\langle \mu \rangle}_4(-2)$ exploration to the approximation of the uniform exploration. The main idea in the proof of this lemma is that the starting points of each loop become i.i.d.\@ uniform on the boundary when $\mu\to \infty$. This essentially stems from the fact that if $E$ is an exponential random variable, then $e^{-i\mu E}$ converges in distribution as $\mu\to \infty$ to a uniform random variable on the circle $\partial \mathbb{D}$.}
\begin{lemma}\label{lemme cv approximations}
	Fix $n \in \N$. Then
	$$
	(D^{0,\mu,n}, \tau^n_0,B)
	\mathop{\longrightarrow}\limits_{\mu\to \infty}^{(\mathrm{d})} (D^{0,n}, \tau^n_0,B),
	$$
	where the first coordinate converges in $\mathcal{D}$.
\end{lemma}

\begin{proof}
	We set for all $k \in \lb 1, \Lambda^n \rb$, $
	X_k^{\mu,n} = (W^\mu_0)_{S_k^{n}},
	$ 
	and ${\bf X}^{\mu,n} = (X^{\mu,n}_k)_{1 \le k \le \Lambda^n}$. For the uniform exploration we set
	${\bf X}^n = (X^n_k)_{1 \le k \le \Lambda^n}$. We also write $(e^n_k)_{k\ge 1}$ the family of the excursions {of $B$} of height larger than $1/2^n$. {Let us write also $(e_v)_{v\in \R_+}$ the Poisson point process of the excursions of $B$, which is indexed by the local time at $0$ and recall that we denote by $\tau(e_v)$ the duration of the excursion $e_v$. We define the point processes $(e^<_v)_{v\ge 0}$ and $(e^>_v)_{v\ge 0}$ of excursions of height smaller than $1/2^n$ and larger than $1/2^n$.}
	
	To prove the above lemma, it is enough to verify that
	\begin{equation}\label{cv vers des uniformes}
		{\bf X}^{\mu,n} \mathop{\longrightarrow}\limits_{\mu \to \infty}^{(\mathrm{d})} {\bf X}^n.
	\end{equation}
	Indeed, we know that $W_0^n$ is defined by (\ref{approximation exploration uniforme}) and moreover for all $k \in \lb 1, \Lambda^n \rb$, for all $s \in [L^n_{k-1}, L^n_k)$,
	\begin{align*}
		(W^{\mu, n}_0)_s &= (W^\mu_0)_{S^{n}_k}
		\exp \left(
		i\left(
		2B_{s+S^n_k} -2B_{S^n_k}- \int_{S^n_k}^{s+S^n_k} \cot(B_u)du - \mu (\ell^{0}_s -\ell^{0}_{S^n_k})\right)
		\right) \\
		&= X^{\mu,n}_k \exp\left(
		i\left(
		2B_{s+S^n_k} 
		- \int_{S^n_k}^{s+S^n_k} \cot(B_u)du
		\right)
		\right).
	\end{align*}
	Thus, if ones assumes (\ref{cv vers des uniformes}), then we have clearly the convergence of the driving function 
	$
	W^{\mu,n}_0
	$
	towards $W^n_0$ in the Skorokhod $J_1$ topology. This entails the convergence in $\mathcal{D}$ of $(D^{0,\mu,n}_s)_{s\ge 0}$ by Remark \ref{remarque cv pilote}.
	Let us now prove the convergence (\ref{cv vers des uniformes}). One can see that
	\begin{align}
	(W_0^\mu)_{S^n_1} &=
	\exp\left(i\left(
	 -  \int_0^{S^n_1} \cot(B_s)ds - \mu \ell^{0}_{S_1^n}
	\right)
	\right)\notag
	\\
	 &={\exp\left( i \left(- \sum_{v<\ell^{0}_{S_1^n}} \int_0^{\tau(e_v)} \cot(e_v(s))ds - \mu \ell^{0}_{S_1^n} \right) \right)}\label{eq W zero mu 1}
	\end{align}
	and that for all $k \ge 2$, 
	\begin{align}
	&(W_0^\mu)_{S^n_k} =
	(W_0^\mu)_{S^n_{k-1}} \exp\left(i\left(
	-\int_{S^n_{k-1}}^{S^n_k} \cot(B_s)ds -\mu (\ell^{0}_{S^n_k}-\ell^{0}_{T^n_{k-1}})
	\right)
	\right)\notag\\
	&={(W_0^\mu)_{S^n_{k-1}} \exp\left( i\left(-\int_0^{\tau(e^n_k)} \cot(e^n_k(s)) ds - \sum_{\ell^0_{T^n_{k-1}}<v<\ell^0_{S^n_k}} \int_0^{\tau(e_v)} \cot(e_v(s)) ds -\mu (\ell^{0}_{S^n_k}-\ell^{0}_{T^n_{k-1}}) \right)\right)}\label{eq W zero mu k}
	\end{align}
	{By the restriction property of Poisson point processes, the point processes $(e^<_v)_{v\ge 0}$ and $(e^>_v)_{v\ge 0}$ of excursions of height smaller than $1/2^n$ and larger than $1/2^n$ are independent Poisson point processes. Note also that the atoms of $(e^>_v)_{v\ge 0}$ are the $e^n_k$'s for $k\ge 1$ which appear at the local times $\ell^0_{S^n_k} = \ell^0_{T^n_k}$'s.} {In particular,} the $\ell^{0}_{S^n_k}-\ell^{0}_{T^n_{k-1}}$'s are i.i.d. exponential random variables whose parameter $\lambda_n>0$ only depends on $n$ (and not on $\mu$), {which are independent of $(e_v^<)_{v\ge 0}$ and of the excursions $(e^n_k)_{k\ge 1}$}. But one can see that if $E$ is an exponential random variable of parameter $\lambda_n>0$, then
	$$
	(E,e^{-i\mu E}) \mathop{\longrightarrow}\limits_{\mu \to \infty}^{(\mathrm{d})} (E,U_1^n),
	$$
	where $U^n_1$ is independent from $E$ and is a uniform random variable in $\partial \mathbb{D}$. {Taking \eqref{eq W zero mu 1} and \eqref{eq W zero mu k} into account, the above convergence yields
	$$
	\left(X^{\mu, n}_1, \frac{X^{\mu,n}_2}{X^{\mu,n}_1}, \ldots, \frac{X^{\mu, n}_{\Lambda^n}}{X^{\mu, n}_{\Lambda^n-1}} \right)
	\mathop{\longrightarrow}\limits_{\mu \to \infty}^{(\mathrm{d})} \left(U^n_1, \ldots, U^n_{\Lambda^n}\right),
	$$
	where $(U_j^n)_{j\ge 1}$ is a family of i.i.d.\@ uniform random variables in $\partial \mathbb{D}$ which is independent of $(B_s)_{s\ge 0}$, hence (\ref{cv vers des uniformes}).}
\end{proof}

\begin{proof}[Proof of Proposition \ref{cv branche stoppee}]
	This follows by combining Lemma \ref{lemme approximation uniforme}, Lemma \ref{lemme cv approximations} and the fact that $D^{0,n}$ converges towards the uniform exploration of the nested $\mathrm{CLE}_4$ targeted at zero. 
\end{proof}

\subsection{Joint convergence of the exploration branches}\label{sous-section convergence jointe des branches quand mu tend vers l infini}

Before proving Proposition \ref{prop mu vers l'infini}, let us show the convergence of the interiors of the loops drawn along the exploration. {The main idea of the proof of the lemma below is that the loops are drawn using the same excursions, the only difference being the ``starting points''.}
\begin{lemma}\label{lemme cv boucles}
	Let $\vp>0$. For all $i\ge 1$, let $e_{\vp, i}$ be the $i$-th excursion of $B$ of height at least $\vp$ away from $\pi \Z$ during an interval $[S_{\vp,i}, T_{\vp,i}]$. Let $\mathcal{L}^{0,\mu}_{\vp,i}$ (resp. $\mathcal{L}^{0}_{\vp,i}$) denote the loop drawn by ${\bf D}^{0,\mu}$ (resp. ${\bf D}^0$) during the interval $[S_{\vp,i}, T_{\vp,i}]$, so that the domain inside the loop is $\mathcal{B}^{0,\mu}_{\vp,i}= {\bf D}^{0,\mu}_{S_{\vp,i}} \setminus \overline{{\bf D}^{0,\mu}_{T_{\vp,i}}}$ (resp. $\mathcal{B}^0_{\vp, i} = {\bf D}^0_{S_{\vp, i}} \setminus \overline{{\bf D}^0_{T_{\vp, i}}}$). Then we have, jointly with the convergence of Proposition \ref{prop cv branche vers zero}, for all $\vp >0$,
	$$
	\left(
	\mathcal{B}^{0,\mu}_{\vp, i}
	\right)_{i\ge 1}
	\mathop{\longrightarrow}\limits_{\mu \to \infty}^{(\mathrm{d})}
	\left(
	\mathcal{B}^{0}_{\vp, i}
	\right)_{i\ge 1}
	$$
	in terms of finite dimensional distributions for the Carath\'eodory topology and for the Hausdorff topology for their closure.
\end{lemma}

\begin{proof}
	It suffices to prove the convergence in the sense of Carath\'eodory since the outermost loop surrounding a point $z$ is constant in law as $\mu \to \infty$ {given that it is the outermost loop of a $\mathrm{CLE}_4$ surrounding $z$}. Let $\mu_n \to \infty$. By Skorokhod's representation theorem, we assume that the convergence of Proposition \ref{cv branche stoppee} holds almost surely along $\mu_n$.
	
	Let $\vp>0$. It is enough to prove the convergence for the excursions before the time $\tau_0$. Let $i\ge 1$ such that $T_{\vp,i}<\tau_0$. By Proposition \ref{cv branche stoppee}, we know that a.s.
	\begin{equation}\label{cv D S Carath\'eodory}
		D^{0,\mu_n}_{S_{\vp,i}}
		\mathop{\longrightarrow}\limits_{n\to \infty}
		D^0_{S_{\vp,i}}
	\end{equation}
	in the sense of Carath\'eodory.
	
	Besides, let us denote by $g^{0,\mu_n}_{S_{\vp,i}}: D^{0,\mu_n}_{S_{\vp,i}} \to \mathbb{D}$ (resp. $g^{0}_{S_{\vp,i}}: D^{0}_{S_{\vp,i}} \to \mathbb{D}$) the unique conformal map such that $g^{0,\mu_n}_{S_{\vp,i}}(0)=0$ and $(g^{0,\mu_n}_{S_{\vp,i}})'(0)= e^{-s}$ (resp. $g^{0}_{S_{\vp,i}}(0)=0$ and $(g^{0}_{S_{\vp,i}})'(0)= e^{-s}$). Then by definition of the radial Loewner chain, we have
	\begin{equation}\label{eq egalite boucles a rotation pres}
	\frac{1}{(W^{\mu_n}_0)_{S_{\vp,i}}} 
	g_{S_{\vp, i}}^{0,\mu_n} \left(
	\mathcal{L}^{0,\mu_n}_{\vp,i}
	 \right)
	 = 
	 \frac{1}{X_{i,\vp}} g^0_{S_{\vp,i}} \left(
	 \mathcal{L}^0_{\vp,i}
	 \right),
	\end{equation}
	where $X_{i,\vp}$ is a uniform random variable in $\partial \mathbb{D}$ which comes from the approximation of the uniform exploration: {indeed, from the approximation of the uniform exploration \eqref{approximation exploration uniforme}, we know for all $i\ge 1$, that the loop $g^0_{S_{\vp,i}}(\mathcal{L}^0_{\vp,i})$ is drawn using the radial Loewner chain driven by
		$$
		s\mapsto X_{i,\vp} \exp\left(i \left( 2e_{i,\vp}(s) - \int_0^s \cot(e_{i,\vp}(s))\right) \right),
		$$
		where the $X_{i,\vp}$'s are i.i.d.\@ uniform random variables in $\partial \mathbb{D}$ which are independent of $B$; while by definition of the radial $\mathrm{SLE}^{\langle \mu_n \rangle}_4(-2)$ exploration, using that the local time is constant on excursion time intervals, we know that the loop $g_{S_{\vp, i}}^{0,\mu_n} \left(
		\mathcal{L}^{0,\mu_n}_{\vp,i}
		\right)$ is drawn using the radial Loewner chain driven by
		$$
		s \mapsto (W^{\mu_n}_0)_{S_{\vp,i}}  \exp\left(i \left( 2e_{i,\vp}(s) - \int_0^s \cot(e_{i,\vp}(s))\right) \right),
		$$
		hence \eqref{eq egalite boucles a rotation pres} since the driving functions only differ by their starting point.} One concludes using the fact that the $(W^{\mu_n}_0)_{S_{\vp,i}}$'s for $i\ge 1$ converge to independent uniform random variables on $\partial \mathbb{D}$ as in the end of the proof of Lemma \ref{lemme cv approximations}, so that
	$$
	g_{S_{\vp, i}}^{0,\mu_n} \left(
	\mathcal{B}^{0,\mu_n}_{\vp,i}
	\right)
	\mathop{\longrightarrow}\limits_{n\to \infty}
	g_{S_{\vp, i}}^{0} \left(
	\mathcal{B}^{0}_{\vp,i}
	\right)
	$$
	in the sense of Carath\'eodory in distribution jointly in $i\ge 1$. By taking (\ref{cv D S Carath\'eodory}) into account, this ends the proof.
\end{proof}

Let us conclude this section by ending the proof of Proposition \ref{prop mu vers l'infini}. 

\begin{proof}[Proof of Proposition \ref{prop mu vers l'infini}]
Let us first prove the convergence of the exploration branches. More precisely, we want the convergences (\ref{eq cv branche vers z}) to hold jointly in $z \in \mathbb{D} \cap \Q^2$. 

For all $\mu \in \R$, for all $z \neq w \in \mathbb{D} \cap \Q^2$, we define the time $\sigma_{z,w}^\mu$ (resp.\@ $\sigma_{z,w}$) at which the radial $\mathrm{SLE}^{\langle \mu \rangle}_4(-2)$ exploration (resp.\@ uniform exploration) towards $z$ separates $w$ from $z$ by 
$$
\sigma^\mu_{z,w} = \inf \{ s\ge0; \ w \not \in {\bf D}^{z,\mu}_s\} \qquad \text{(resp.}
\enskip 
\sigma_{z,w} = \inf \{ s\ge0; \ w \not \in {\bf D}^{z}_s\} \text{).}
$$ 
By the target invariance of the explorations {(see Proposition \ref{prop target invariance} for the $\mathrm{SLE}^{\langle \mu \rangle}_4(-2)$)} and by the renewal property {(see Proposition \ref{prop renewal et Markov conforme} for the $\mathrm{SLE}^{\langle \mu \rangle}_4(-2)$)}, it is enough to show that for all $z \neq w \in \mathbb{D} \cap \Q^2$ we have the convergence of the separation time
$$
\sigma^\mu_{z, w} \mathop{\longrightarrow}\limits_{\mu \to \infty}^{(\mathrm{d})}
\sigma_{z,w},
$$
jointly with the convergence of the branch ${\bf D}^{z,\mu}$. Actually, by conformal invariance, it is enough to prove that for all $w \in (\mathbb{D} \cap \Q^2)\setminus \{0\}$, 
\begin{equation}\label{cv temps de separation}
\sigma^\mu_{0, w} \mathop{\longrightarrow}\limits_{\mu \to \infty}^{(\mathrm{d})}
\sigma_{0,w},
\end{equation}
jointly with the convergence of the branch ${\bf D}^{0,\mu}$. The convergence of the separation time (\ref{cv temps de separation}) will be easier to prove in our context than in \cite{AHPS21}. One can first see that $\sigma^\mu_{0,w}$ is tight as $\mu \to \infty $. Indeed, for all $\mu\in \R$, let $I^\mu$ be the smallest integer $i\ge 1$ such that the $i$-th loop surrounding $0$ does not contain $w$. Then $\sigma^\mu_{0,w}$ is smaller than the time 
at which we draw the $I^\mu$-th loop. If we assumed by contradiction that $\sigma^{\mu_n}_{0,z} \to \infty$ as $n \to \infty$ for some sequence $\mu_n \to \infty$, then we would have $I^{\mu_n} \rightarrow \infty$ and then $w$ would be at distance zero from $0$ by the Hausdorff convergence of the domains inside the loops encircling $0$ of Lemma \ref{lemme cv boucles}, absurd. Similarly, $\sigma_{w,0}^\mu$ is also tight as $\mu \to \infty$. Now let $\mu_n \to \infty$. By taking a subsequence, again denoted by $\mu_n$, 
we may assume that
\begin{equation}\label{eq cv D tau B}
({\bf D}^{0,\mu_n}, \tau^{\mu_n}_0, B^{0,\mu_n}, {\bf D}^{w,\mu_n}, \tau_w^{\mu_n}, B^{w,\mu_n}) 
\mathop{\longrightarrow}\limits_{n \to \infty}^{(\mathrm{d})}
({\bf D}^{0}, \tau_0, B^0, \widetilde{\bf D}^w, \widetilde{\tau}_w, \widetilde{B}^w) 
\end{equation}
in $\mathcal{D} \times \R \times \mathcal{C}(\R_+, \R) \times \mathcal{D}_w \times \R \times \mathcal{C}(\R_+, \R)$, where $ (\widetilde{\bf D}^w, \widetilde{\tau}_w, \widetilde{B}^{w})$ has the same law as $({\bf D}^w, \tau_w, B^w)$, jointly with the convergence of the separation times 
\begin{equation}\label{eq cv separation sous-suite}
(\sigma^{\mu_n}_{0,w}, \sigma^{\mu_n}_{w, 0})
\mathop{\longrightarrow}\limits_{n \to \infty}^{(\mathrm{d})}
(\sigma^*_{0,w}, \sigma^*_{w, 0})
\end{equation}
towards some random variables $\sigma^*_{0,w}, \sigma^*_{w, 0}$. 

{In what follows, we prove that $\sigma_{0,w}^*=\sigma_{0,w}$ so that \eqref{cv temps de separation} holds. The main idea of the proof is to show that $\widetilde{\bf D}^w$ coincides up to a time-change with ${\bf D}^0$ until the time $\sigma_{0,w}^*$, which will imply the difficult part of the equality $\sigma_{0,w}^*=\sigma_{0,w}$, which is the inequality $\sigma_{0,w}^*\le\sigma_{0,w}$.}

For all $n \ge 0$, let $\varphi_n : [0, \sigma^{\mu_n}_{0,w}] \to [0, \sigma^{\mu_n}_{w, 0}]$ be the continuous increasing function such that for all $s \in [0, \sigma^{\mu_n}_{0,w}]$, we have $B_s ^{0,\mu_n}= B^{w,\mu_n}_{\varphi_n(s)}$ (given by construction of the branching exploration). One can also express $\varphi_n(s)$ as minus the logarithm of the conformal radius of $D^{0,\mu_n}_s$ seen from $w$. 
Since $\sigma^{\mu_n}_{w,0}$ converges in distribution, and thanks to (\ref{eq cv D tau B}), we may take another subsequence again denoted by $\mu_n$ such that for all $s \in [0,\sigma^*_{0,w})\cap \Q$, we have
\begin{equation}\label{eq cv phi}
\varphi_n(s) \mathop{\longrightarrow}\limits_{n \to \infty}^{(\mathrm{d})} \varphi^*(s)
\end{equation}
for some increasing process $(\varphi^*(t))_{t \in [0,\sigma^*_{0,w})\cap \Q}$ jointly with the convergences (\ref{eq cv D tau B}) and (\ref{eq cv separation sous-suite}). This implies that $B^0_t = \widetilde{B}^w_{\varphi^*(t)}$ for all $t \in [0,\sigma^*_{0,w})\cap \Q$, hence $\varphi^*$ is strictly increasing. 
Similarly, we can also assume that for all $s\in [0,\sigma^*_{w,0})\cap \Q$, jointly with (\ref{eq cv D tau B}), (\ref{eq cv separation sous-suite}) and (\ref{eq cv phi}), we have
\begin{equation}\label{eq cv psi}
	\varphi_n^{-1}(s) \mathop{\longrightarrow}\limits_{n \to \infty}^{(\mathrm{d})} \psi^*(s)
\end{equation}
for some increasing process $(\psi^*(s))_{s \in [0,\sigma^*_{w,0})\cap \Q}$. For all $s\in [0,\sigma^*_{w,0})\cap \Q$, we also have $B^0_{\psi^*(s)} = \widetilde{B}^{w}_s$. So $B^0_{\vert [0,\sigma^*_{0,w}]}$ is obtained as a time change of $\widetilde{B}^z_{\vert [0,\sigma^*_{w,0}]}$, which is itself a time change of $B^0_{\vert [0,\sigma^*_{0,w}]}$ Thus, $\varphi^*$ and $\psi^*$ can be extended to continuous increasing bijections between $[0, \sigma^*_{0,w}]$ and $[0,\sigma^*_{w,0}]$ such that $\varphi^*= (\psi^*)^{-1}$. As a result, the convergences (\ref{eq cv phi}) and (\ref{eq cv psi}) hold uniformly on the corresponding intervals. 

For all $\vp>0$, for all $i,n\ge 1$, let $[S^{0,\mu_n}_{\vp,i}, T^{0,\mu_n}_{\vp, i}]$ (resp. $[S^{0}_{\vp,i}, T^{0}_{\vp, i}]$, $[S^{w,\mu_n}_{\vp,i}, T^{w,\mu_n}_{\vp, i}]$, $[\widetilde{S}^{w}_{\vp,i}, \widetilde{T}^{w}_{\vp, i}]$) be the interval of the $i$-th excursion of $B^{0, \mu_n}$ (resp. $B^0$, $B^{w,\mu_n}$, $\widetilde{B}^w$) of height larger than $\vp$, during which the exploration ${\bf D}^{0,\mu_n}$ (resp. ${\bf D}^0$, ${\bf D}^{w,\mu_n}$, $\widetilde{\bf D}^w$) draws a loop encircling a domain $\mathcal{B}^{0,\mu_n}_{\vp,i}$ (resp. $\mathcal{B}^0_{\vp,i}$, $\mathcal{B}^{w,\mu_n}_{\vp,i}$, $\widetilde{\mathcal{B}}^w_{\vp,i}$).
We may assume by Lemma \ref{lemme cv boucles} that these domains converge: for all $\vp>0$, jointly with the previous convergences, 
$$
	\left(
	\mathcal{B}^{0,\mu_n}_{\vp, i},
	\mathcal{B}^{w,\mu_n}_{\vp,i}
	\right)_{i\ge 1}
	\mathop{\longrightarrow}\limits_{n \to \infty}^{(\mathrm{d})}
	\left(
	\mathcal{B}^{0}_{\vp, i},
	\widetilde{\mathcal{B}}^w_{\vp,i}
	\right)_{i\ge 1}
	$$
in the sense of Carathéodory and also for the Hausdorff distance when we consider their closure. 
Notice that, by the convergence of the Brownian motions, we also have the convergence of the times of start and end of the excursions of height larger than $\vp $. By Skorokhod's representation theorem, we assume that all these convergences hold almost surely.

If $s < \sigma_{0,w}$, then $w \in {\bf D}^0_s$, so that for $n$ large enough we have $w \in {\bf D}^{0,\mu_n}_s$, hence $s< \sigma_{0,w}^{\mu_n}$. Thus we have plainly $\sigma^*_{0,w} \ge \sigma_{0,w}$. Let us now prove that $\sigma^*_{0,w}\le\sigma_{0,w}$. Let $\vp>0$ small enough so that we can find $i\ge 1$ such that $ S^0_{\vp,i} < T^0_{\vp,i} < \sigma^*_{0,w}$. Let $s <S^0_{\vp,i} $. Then we have by definition of the exploration $\mathcal{B}^0_{\vp,i} \subset {\bf D}^0_s$ and $\widetilde{\mathcal{B}}^w_{\vp,i} \subset \widetilde{\bf D}^w_{\vp,i}$. But since $T^0_{\vp,i} < \sigma^*_{0,w}$ we have for all $n$ large enough $T^{0,\mu_n}_{\vp,i} < \sigma^{\mu_n}_{0,w}$. Thus, for all $n$ large enough, $\mathcal{B}^{0,\mu_n}_{\vp,i} = \mathcal{B}^{w,\mu_n}_{\vp,i}$ (given that $B^{0,\mu_n}_s=B^{w,\mu_n}_{\varphi_n(s)}$ for all $s\le \sigma^{\mu_n}_{0,w}$). Hence $\mathcal{B}^0_{\vp,i} = \widetilde{\mathcal{B}}^w_{\vp,i}$ a.s.\@ by the Hausdorff convergence. Therefore ${\bf D}^0_s \cap \widetilde{\bf D}^w_{\varphi^*(s)} \neq \emptyset$ almost surely. But then by taking some $z \in {\bf D}^0_s \cap \widetilde{\bf D}^w_{\varphi^*(s)} $, we see by the Carathéodory convergence that the limits of ${\bf D}^{0,\mu_n}_s= {\bf D}^{w,\mu_n}_{\varphi_n(s)}$ coincide, i.e. that ${\bf D}^0_s =\widetilde{\bf D}^w_{\varphi^*(s)} $ {since $d_{\text{Carathéodory},z}$ defined in \eqref{eq distance caratheodory} is indeed a distance on simply connected open domains containing $z$.} As a consequence, $s< \sigma_{0,w}$. By taking $\vp$ small enough and $i$ such that $S_{\vp,i}^0$ is close enough to $\sigma^*_{0,w}$ we obtain that $\sigma^*_{0,w}\le\sigma_{0,w}$, hence (\ref{cv temps de separation}). This concludes the proof 
since the domains $D^{z,w,\mu_n}_s$ for $s\ge 0$ correspond up to a time-change to the branch ${\bf D}^{w,\mu_n}$ stopped when $z$ and $w$ are separated.
\end{proof}

\section{Conformally invariant distance and quantum natural distance}\label{section Lamperti}
This section is devoted to the relationships between the distance $\mathrm{d}_{\mathrm{WW}}$ to the boundary of \cite{WW13} {recalled in Subsection \ref{sous-section explo unif}} and the quantum natural distance $\mathrm{d}_{\mathrm{q}}$ introduced by \cite{AHPS21} {which is characterized by \eqref{eq distance quantique}}. We prove Theorem \ref{th Lamperti} using Theorem \ref{prop de Markov disques quantiques}, which relates these two distances via a Lamperti type transform modulo a process with stationary increments.

{The main idea of the proof is that when one parametrizes the exploration by the Lamperti transform of the quantum distance, if we stop the exploration at time $t$, denote by $x$ the quantum boundary length of the unexplored region and if we divide by $x$ the process giving the future evolution of the quantum boundary length, then the process has the same law as the initial process. A crucial ingredient, coming from Theorem \ref{prop de Markov disques quantiques} is that, once properly normalized, the unexplored region is again a unit-boundary quantum disk.}
\begin{proof}[Proof of Theorem \ref{th Lamperti}]
{Recall from Subsection \ref{sous-section composante localement la plus grande} that $\mathcal{Y}(u)=\nu^2_{h^2_{\vert C_u}} (\partial C_u)$ where $(C_u)_{u\ge 0}$ is the locally largest component in the uniform exploration parametrized by the distance $\mathrm{d}_{\mathrm{WW}}$ to the boundary. Recall that $\sigma$ is the right-continuous inverse of $(\mathrm{d}_\mathrm{q} (\partial \mathbb{D}, C_u))_{u\ge 0}$ and for all $t\ge 0$ we have $Y_t= \mathcal{Y}(\sigma(t))$.} Let $\tau$ be defined by (\ref{Lamperti}) with $\alpha=  -1$. We know from Theorem \ref{prop de Markov disques quantiques} that
	$$
	({\mathcal{Y}}({\sigma(t)}))_{t\ge 0} \mathop{=}\limits^{(\mathrm{d})}(e^{\xi(\tau({\pi t }))})_{t\ge 0},
	$$
	 so that $
	(\log(\mathcal{Y}(\sigma(\tau^{-1}(t){/\pi}))))_{t\ge 0}
	$ has the same law as $\xi$ and we identify the two processes using an appropriate coupling.
By (\ref{eq distance quantique}) and by the fact that the Poisson point process (PPP) of loops $(\widehat{\gamma}_u)_{u\ge 0}$ of intensity $M$ introduced in Subsection \ref{sous-section explo unif} has an infinite intensity, we get that $u \mapsto \mathrm{d}_{\mathrm{q}}(\partial \mathbb{D}, C_u)$ is increasing, so that $t \mapsto \sigma(t)$ is continuous. Similarly, since the Lévy measure of $\xi$ is infinite, $t \mapsto \sigma(t)$ is increasing.

To end the proof, it is enough to show that the process $(\sigma(\tau^{-1}(t){/\pi}))_{t\ge 0}$ has stationary increments.

Each positive jump {of $(\mathcal{Y}(u))_{u\ge 0}$} at some time $u$ corresponds to the discovery of a simple loop $F_u^{-1}(\widehat{\gamma}_u)$ where $(\widehat{\gamma}_u)_{u\ge 0}$ is a PPP which has the same law as described in Subsection \ref{sous-section explo unif} and $F_u: C_u \to \mathbb{D}$ is the conformal transformation which is given by $F_u = \Phi\circ\widehat{F}_u^{z_{j_k}} \circ \Phi^{-1}$ when $u \in [\tau_{k-1}, \tau_k)$ using the notation introduced in Subsections \ref{sous-section explo unif} and \ref{sous-section composante localement la plus grande}, and where $\Phi: \mathbb{H} \to \mathbb{D}$ is a fixed conformal map.

 Let us denote by $(\mathcal{E}_t)_{t\ge 0}$ (resp. $(\mathcal{F}_u)_{u\ge 0}$) the augmented natural filtration of the process $(\mathcal{Y}(\sigma(t)))_{t\ge 0}$ (resp. {$(h^2,\widehat{\gamma}_u, F_u)_{u\ge 0}$}).

{Furthermore, for all $u\ge 0$, we define the quantum surface
$$\mathrm{QS}(u)\coloneqq \left(\mathbb{D}, h\circ F_{u}^{-1} + 2 \log |(F_u^{-1})'| - \log (\nu^2_{h^2}(\partial C_u))\right).$$}
{The main idea of the proof is to show that for all $t\ge 0$, the pair $(\mathrm{QS}(\sigma(\tau^{-1}(t){/\pi})), (\widehat{\gamma}_u)_{u\ge \sigma(\tau^{-1}(t){/\pi})})$ has the same law as $(\mathrm{QS}(0), (\widehat{\gamma}_u)_{u\ge 0})$ and to show that this entails the fact that $(\sigma(\tau^{-1}(t){/\pi}))_{t\ge 0}$ has stationary increments.}

By Theorem \ref{prop de Markov disques quantiques}, for all $t\ge 0$, conditionally on $(\mathcal{Y}(\sigma(s)))_{0 \le s \le t}$, the unexplored region $(C_{\sigma(t)},h^2_{|C_{\sigma(t)}})$ is a $2$-quantum disk of quantum boundary length $\nu^2_{h^2}(\partial C_{\sigma(t)})$.  In other words, conditionally on $(\mathcal{Y}(\sigma(s)))_{0 \le s \le t}$, the quantum surface 
	$$ 
	{\mathrm{QS}(\sigma(t))=}\left(\mathbb{D}, h\circ F_{\sigma(t)}^{-1} + 2 \log |(F_{\sigma(t)}^{-1})'| - \log (\nu^2_{h^2}(\partial C_{\sigma(t)}))\right)
	$$ 
	is a unit boundary $2$-quantum disk. But using the fact that $\tau^{-1}(t){/\pi}$ is an $(\mathcal{E}_t)_{t\ge 0}$-stopping time, we get that for all $t\ge 0$, conditionally on $\mathcal{E}_{\tau^{-1}(t){/\pi}}$, the quantum surface 
	\begin{equation}\label{eq disque quantique renormalisé}
	{\mathrm{QS}(\sigma(\tau^{-1}(t){/\pi}))}=\left(\mathbb{D}, h\circ F_{\sigma(\tau^{-1}(t){/\pi})}^{-1} + 2 \log |(F_{\sigma(\tau^{-1}(t){/\pi})}^{-1})'| - \log (\nu^2_{h^2}(\partial C_{\sigma(\tau^{-1}(t){/\pi})}))\right)
	\end{equation}
	is a unit boundary $2$-quantum disk.
	
	Besides, by (\ref{eq distance quantique}) we know that $u \mapsto \sigma^{-1}(u)=\mathrm{d}_{\mathrm{q}}(\partial \mathbb{D}, C_u)$ is adapted with respect to $(\mathcal{F}_u)_{u\ge 0}$. 
	Moreover, for all $u\ge 0$,
	$$\tau({\pi}\sigma^{-1}(u))=\int_0^{{\pi}\sigma^{-1}(u)} \frac{dt}{\mathcal{Y}(\sigma(t{/\pi}))}  {= \int_0^{\sigma^{-1}(u)} \frac{\pi dt}{\mathcal{Y}(\sigma(t))}}= {\pi}\int_0^u \frac{ \sigma^{-1}(dv)}{\mathcal{Y}(v)}.$$
	So $(\tau({\pi}\sigma^{-1}(u)))_{u\ge0}$ is $(\mathcal{F}_u)_{u\ge 0}$-adapted. In particular, we obtain that for all $t\ge 0$, the random variable $\sigma(\tau^{-1}(t){/\pi})$ is an $(\mathcal{F}_u)_{u\ge 0}$-stopping time. As a result, given that $(\widehat{\gamma}_u)_{u\ge 0}$ is independent of $h^2$, we get that for all $t\ge 0$, conditionally on $\mathcal{F}_{\sigma(\tau^{-1}(t){/\pi})}$, the point process $(\widehat{\gamma}_{\sigma(\tau^{-1}(t){/\pi})+u})_{u\ge 0}$ has the same law as $(\widehat{\gamma}_u)_{u\ge 0}$ i.e. it is a PPP of intensity $M$. In addition to that, by definition of $(F_u)_{u\ge 0}$ from $(\widehat{\gamma}_u)_{u\ge 0}$, we have that conditionally on $\mathcal{F}_{\sigma(\tau^{-1}(t){/\pi})}$, the pair $((\widehat{\gamma}_{\sigma(\tau^{-1}(t){/\pi})+u})_{u\ge 0}, (F_{\sigma(\tau^{-1}(t){/\pi})} \circ F_{\sigma(\tau^{-1}(t){/\pi})+u}^{-1})_{u\ge 0})$ has the same law as $((\widehat{\gamma}_u)_{u\ge 0}, (F_u^{-1})_{u\ge 0})$.
	
	{But notice that $(\mathrm{QS}(u))_{u\ge 0}$ is $(\mathcal{F}_u)_{u\ge 0}$-adapted, so that the quantum disk $\mathrm{QS}(\sigma(\tau^{-1}(t){/\pi}))$ is $\mathcal{F}_{\sigma(\tau^{-1}(t){/\pi})}$-measurable.} Thus, by taking the above observation on the quantum disk (\ref{eq disque quantique renormalisé}) into account, we have that conditionally on $\mathcal{E}_{\tau^{-1}(t){/\pi}}\subseteq \mathcal{F}_{\sigma(\tau^{-1}(t){/\pi})}$, the quantum disk {$\mathrm{QS}(\sigma(\tau^{-1}(t){/\pi}))$} from (\ref{eq disque quantique renormalisé}) is a unit boundary length $2$-LQG disk independent of the pair $$((\widehat{\gamma}_{\sigma(\tau^{-1}(t){/\pi})+u})_{u\ge 0}, (F_{\sigma(\tau^{-1}(t){/\pi})} \circ F_{\sigma(\tau^{-1}(t){/\pi})+u}^{-1})_{u\ge 0}),$$
	which has the same law as $((\widehat{\gamma}_u)_{u\ge 0}, (F_u^{-1})_{u\ge 0})$. In other words, conditionally on $\mathcal{E}_{\tau^{-1}(t){/\pi}}$, the quantum disk {$\mathrm{QS}(\sigma(\tau^{-1}(t){/\pi}))$} is a unit boundary length $2$-LQG disk independent of the process $(F_{\sigma(\tau^{-1}(t){/\pi})}(C_{\sigma(\tau^{-1}(t){/\pi})+u}))_{u\ge 0}$, which has the same law as $(C_u)_{u\ge 0}$.
	
Besides, for all $u\ge 0$,
\begin{align*}
&\log \mathcal{Y}({\sigma(\tau^{-1}(t){/\pi})+u}) - \log \mathcal{Y}({\sigma(\tau^{-1}(t){/\pi})})
\\
&= \log \left(\frac{\nu^2_{h^2}(F_{\sigma(\tau^{-1}(t){/\pi})}^{-1}(F_{\sigma(\tau^{-1}(t){/\pi})}(\partial C_{\sigma(\tau^{-1}(t){/\pi})+u})))}{\nu^2_{h^2}(\partial C_{\sigma(\tau^{-1}(t){/\pi})})} \right) \\
&=
\log \left(\nu^2_{h\circ F_{\sigma(\tau^{-1}(t){/\pi})}^{-1} + 2 \log |(F_{\sigma(\tau^{-1}(t){/\pi})}^{-1})'| - \log (\nu^2_{h^2}(\partial C_{\sigma(\tau^{-1}(t){/\pi})})} (F_{\sigma(\tau^{-1}(t){/\pi})}( \partial C_{\sigma(\tau^{-1}(t){/\pi})+u})) \right),
\end{align*}
where we used (\ref{covariance}) and the fact that $\nu^2_{h^2-\log b} = \nu^2_{h^2}/b$ for all $b>0$. So conditionally on $\mathcal{E}_{\tau^{-1}(t){/\pi}}$, the process $(\log \mathcal{Y}({\sigma(\tau^{-1}(t){/\pi})+u}) - \log \mathcal{Y}({\sigma(\tau^{-1}(t){/\pi})}))_{u\ge 0}$ has the same distribution as $(\log \mathcal{Y}(u))_{u\ge 0}$.

It is now a simple matter to check that $(\sigma(\tau^{-1}(t){/\pi}))_{t\ge 0}$ has stationary increments. Indeed, let $t\ge 0$. For all $u,s \ge 0$, set
$$
\widetilde{\mathcal{Y}}(u) \coloneqq \frac{\mathcal{Y}({\sigma(\tau^{-1}(t){/\pi})+u})}{\mathcal{Y}({\sigma(\tau^{-1}(t){/\pi})})} \qquad
\text{and}
\qquad
\widetilde{\xi}(s) \coloneqq \xi({t+s})- \xi(t).
$$
Define $\widetilde{\sigma}$ from $\widetilde{\mathcal{Y}}$ as $\sigma$ is defined from $\mathcal{Y}$. Notice that $\widetilde{\xi}$ is characterized by the equality
\begin{equation}\label{eq caractérisation xi tilde}
\forall s \ge 0, \qquad \log \left(\widetilde{\mathcal{Y}}\left({\widetilde{\sigma}\left({(1/\pi)}\int_0^s e^{\widetilde{\xi}(r)}dr\right)}\right) \right) = \widetilde{\xi}(s).
\end{equation}
Since $(\log \mathcal{Y}({\sigma(\tau^{-1}(t){/\pi})+u}) - \log \mathcal{Y}({\sigma(\tau^{-1}(t){/\pi})}))_{u\ge 0}$ has the same distribution as $(\log \mathcal{Y}(u))_{u\ge 0}$, we obtain that 
$$
\left((\log \mathcal{Y}({\sigma(\tau^{-1}(t){/\pi})+u}) - \log \mathcal{Y}({\sigma(\tau^{-1}(t){/\pi})}))_{u\ge 0}, \widetilde{\sigma}, \widetilde{\xi} \right)
$$ has the same law as $(\mathcal{Y}, \sigma, \xi)$. 
Besides, by (\ref{eq caractérisation xi tilde}), we have for all $s\ge 0$,
\begin{align*}
\log \left(\widetilde{\mathcal{Y}}\left({\widetilde{\sigma}\left({(1/\pi)}\int_0^s e^{\widetilde{\xi}(r)}dr\right)}\right) \right) &= \widetilde{\xi}(s) \\
&=\xi(t+s)-\xi(t) \\
&= \log \mathcal{Y}({\sigma(\tau^{-1}(t+s){/\pi})}) - \log  \mathcal{Y}({\sigma(\tau^{-1}(t){/\pi})}) \\
&= \log \widetilde{{\mathcal{Y}}}({\sigma(\tau^{-1}(t+s){/\pi})-\sigma(\tau^{-1}(t){/\pi})}).
\end{align*}
As a consequence, by density of the jumps of $\widetilde{\mathcal{Y}}${, by looking e.g.\@ at the positive jumps of $\widetilde{\mathcal{Y}}$ of size in $[\vp,\infty)$ and letting $\vp \to 0$, we obtain} that for all $s\ge 0$, 
$$\sigma(\tau^{-1}(t+s){/\pi})-\sigma(\tau^{-1}(t){/\pi}) = {\widetilde{\sigma}\left({(1/\pi)}\int_0^s e^{\widetilde{\xi}(r)}dr\right)}.$$
Using the fact that for all $s \ge 0$, we have $\tau^{-1}(s) = \int_0^s e^{\xi(r)} dr$, we conclude that $\sigma\circ \tau^{-1}$ has stationary increments.
\end{proof}

\section{Scaling limit of the distances to the boundary in $3/2$-stable maps}\label{section cartes}
In this section, we describe the scaling limit of the distance between large faces and the boundary in $3/2$ stable maps.
\subsection{Background on $3/2$-stable maps and the peeling exploration}\label{rappels épluchage}

Recall the definition of finite $3/2$-stable maps from Subsection \ref{section background}. If one conditions a $3/2$-stable map of perimeter $2\ell$ on having at least $n$ edges and lets $n\to \infty$, then one obtains as a local limit an infinite $3/2$-stable map of perimeter $2\ell$, whose law is written $\P^{(\ell)}_\infty$. The infinite $3/2$-stable map was studied in \cite{BCM} and we will rely on this work in the next subsection.

{Recall from Subsection \ref{section background} that if $\mathfrak{m}$ is a planar map, $\mathrm{d}^\dagger_{\mathrm{gr}}$ is the graph distance on the dual map of $\mathfrak{m}$ and that $\mathrm{d}^\dagger_{\mathrm{fpp}}$ is the first passage percolation distance, which is obtained by putting i.i.d.\@ parameter $1$ exponential random lengths on each edge of the dual map.} To study the distances $\mathrm{d}^\dagger_{\mathrm{gr}}$ and $\mathrm{d}^\dagger_\mathrm{fpp}$ on Boltzmann planar maps, a convenient tool is the peeling exploration. We only give a brief overview of its definition and we refer the interested reader to the lecture notes \cite{StFlour} and \cite{BuddPeeling}. It consists in discovering the faces of the map, starting from the root face $f_r$ by looking at each step what is behind an edge on the boundary of the explored region which is chosen according to a peeling algorithm. More precisely, we say that a map $\mathfrak{e}$ with a distinguished simple face different from $f_r$ called a hole is a \textit{submap} (with one hole) of $\mathfrak{m}$ if one can obtain $\mathfrak{m}$ by gluing a well chosen planar map in the hole of $\mathfrak{e}$ (see Figure \ref{image sous-carte}).

\begin{figure}[h]
   \centering
   \includegraphics[scale=0.95]{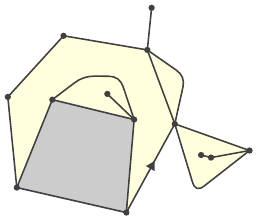}
	 \includegraphics[scale=0.95]{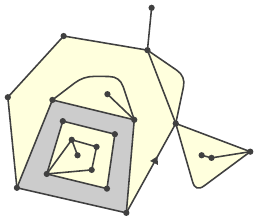}
	 \includegraphics[scale=1.05]{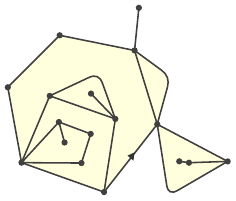}
   \caption{Left: a map $\mathfrak{e}$ with one hole which is a submap of the map $\mathfrak{m}$ on the right. The map $\mathfrak{u}$ which is in the hole of $\mathfrak{e}$ in the centre enables to recover $\mathfrak{m}$ after gluing the boundary of $\mathfrak{u}$ to the boundary of the hole.}
   \label{image sous-carte}
\end{figure}

A \textit{peeling algorithm} $\mathcal{A}$ is a function which takes a submap $\mathfrak{e}$ and gives back an edge in the boundary $\partial \mathfrak{e}$ of the hole of $\mathfrak{e}$. A (filled-in) peeling exploration of $\mathfrak{m}$ is an increasing sequence $(\overline{\mathfrak{e}}_n)_{n\ge 0}$ of submaps of $\mathfrak{m}$ starting from the submap $\overline{\mathfrak{e}}_0$ with only two faces which are the root face $f_r$ and a hole with the same degree, such that for all $n\ge 0$ the submap $\overline{\mathfrak{e}}_{n+1}$ is obtained from $\overline{\mathfrak{e}}_n$ as follows:
\begin{enumerate}
\item If the face $f$ which is on the other side of the peeled edge $\mathcal{A}(\overline{\mathfrak{e}}_n)$ is not in $\overline{\mathfrak{e}}_n$, then $\overline{\mathfrak{e}}_{n+1}$ is obtained by gluing $f$ to $\overline{\mathfrak{e}}_n$ onto the edge $\mathcal{A}(\overline{\mathfrak{e}}_n)$. We denote this case by $C_k$, where $k=\deg(f)/2$.
\item If the face $f$ on the other side of $\mathcal{A}(\overline{\mathfrak{e}}_n)$ is already in $\overline{\mathfrak{e}}_n$, then $f$ must be on the boundary of the hole of $\overline{\mathfrak{e}}_n$ and in this case $\overline{\mathfrak{e}}_{n+1}$ is obtained by identifying the two edges on the boundary of the hole. This creates at most two holes. If $\mathfrak{m}$ is infinite, then we choose to fill in the hole which contains a finite part of the map. If $\mathfrak{m}$ is finite, we choose here to fill in the hole which has the smallest degree and in case of equality we choose to fill in a hole chosen arbitrarily (so that in the finite case we explore the locally largest component). This case is written $G_{k,*}$ or $G_{*,k}$, where $2k$ is the degree of the hole which is filled-in, depending on whether the filled-in hole is on the left or on the right of the peeled edge {(we can speak of left and right since the map has an orientation)}. If there is not any hole, then the exploration stops.
\begin{figure}[h]
   \centering
	 \includegraphics[scale=1.2]{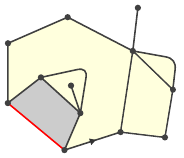}
	 \includegraphics[scale=1.05]{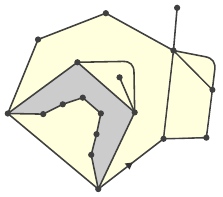}
	 \includegraphics[scale=1.05]{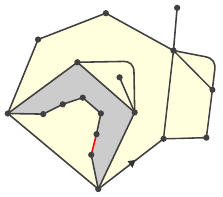}
	 \includegraphics[scale=1.05]{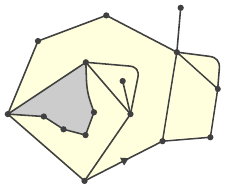}
   \caption{From left to right: two peeling steps of the map $\mathfrak{m}$ from Figure \ref{image sous-carte} starting from the submap $\mathfrak{e}$ on the left. The red edge on the left is the edge chosen by the peeling algorithm. The peeling of the red edge causes an event of type $C_4$. The peeling algorithm then chooses another edge on the boundary which is again in red. When we peel this edge, an event of type $G_{3,1}$ happens and we fill-in the hole which has the smallest perimeter with a map made in this case of only one edge.}
   \label{image épluchage}
\end{figure}
\end{enumerate}

When performing the exploration, one can keep track of the \textit{perimeter process} $(P(n))_{n\ge 0}$, where for all $n\ge 0$, the integer $P(n)$ is the half degree of the hole of $\overline{\mathfrak{e}}_n$. Notice that on an event of type $C_k$ for $k\ge 1$, the perimeter process increases by $k-1$ while on an event of type $G_{*,k}$ or $G_{k,*}$ for some $k\ge 0$, the perimeter process decreases by $k+1$.

A peeling algorithm which is particularly well suited to the study of the distance $\mathrm{d}_{\mathrm{fpp}}^\dagger$ is the \textit{uniform peeling algorithm}. Given a submap $\mathfrak{e}$, it chooses uniformly at random an edge in $\partial{e}$. Using this algorithm properly coupled with the exponential random lengths defining $\mathrm{d}_{\mathrm{fpp}}^\dagger$, by the elementary properties of exponential random variables, one can see that for all $n \ge 0$, assuming that the exploration has not stopped before, the fpp distance from $f_r$ to the hole $\mathfrak{u}_n$ of $\overline{\mathfrak{e}}_n$ can be written
\begin{equation}\label{eq distance fpp}
T({n+1})\coloneqq \mathrm{d}^\dagger_{\mathrm{fpp}}(f_r, \mathfrak{u}_n)= \sum_{i=0}^{n} \frac{\mathcal{E}_i}{2P(i)},
\end{equation}
where the $\mathcal{E}_i$'s are i.i.d. exponential random variables of parameter $1$. More precisely, $\overline{\mathfrak{e}}_n$ corresponds exactly to the closed ball around $f_r$ of radius $T(n)$ where we filled-in the holes which were filled-in during the exploration, and where $T(0)=0$ by convention. In particular, a face $f$ which is discovered at time $n$ (i.e. $f \in \mathrm{Faces}(\overline{\mathfrak{e}}_{n}) \setminus \mathrm{Faces}(\overline{\mathfrak{e}}_{n-1}$)) is at distance $T(n)$ from the root face $f_r$. When the explored map has the law $\P^{(\ell)}$ or $\P^{(\ell)}_\infty$, the $\mathcal{E}_i$'s are independent from the perimeter process. See e.g. Section 13.1 of \cite{StFlour} or Subsection 3.2 of \cite{Kam23} for details. 

For the dual graph distance $\mathrm{d}^\dagger_{\mathrm{gr}}$, an appropriate peeling algorithm is the \textit{peeling by layers algorithm} $\mathcal{A}_{\mathrm{layers}}$. Given a submap $\mathfrak{e}$ of $\mathfrak{m}$, we define the height of an edge $e\in \partial \mathfrak{e}$ as the dual graph distance between $f_r$ and the face next to $e$ which is not in the hole. Assume that there exists an integer $h\ge 0$ such that all the edges of $\partial \mathfrak{e}$ have heights in $\left\{h, h+1\right\}$ and that the edges at height $h$ form a connected interval of $\partial{\mathfrak{e}}$ (we denote this hypothesis by \textbf{(H)}). Then $\mathcal{A}_{\mathrm{layers}}(\mathfrak{e})$ is defined as the leftmost edge in $\partial \mathfrak{e}$ of height $h$ (see Figure \ref{image épluchage par couches}) and when all the edges of $\partial \mathfrak{e}$ are at the same height we choose one of them arbitrarily. Notice that $\overline{\mathfrak{e}}_0$ satisfies \textbf{(H)} and that if $\overline{\mathfrak{e}}_n$ satisfies \textbf{(H)} then with $\mathcal{A}_{\mathrm{layers}}$ the submap $\overline{\mathfrak{e}}_{n+1}$ also satisfies \textbf{(H)} (with $h$ increasing by zero or one). For all $n\ge 0$, let us denote by $H(n)$ the smallest height of the edges in $\partial \overline{\mathfrak{e}}_n$ when performing a peeling by layers exploration. The sequence $(H(n))_{n\ge 0}$ is called the \textit{height process}. Notice that $H(0)=0$ and that for all $n \ge 0$, we have $H(n+1)-H(n) \in \{0,1\}$. See e.g. Section 13.2 of \cite{StFlour} for more details.

\begin{figure}[h]
   \centering
	 \includegraphics[scale=1.2]{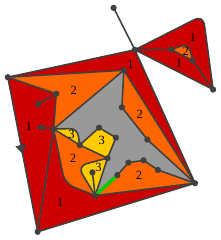}
   \caption{Illustration of the peeling by layers algorithm. The faces are colored according to their distance to the root face $f_r$. The edge chosen by the algorithm $\mathcal{A}_{\mathrm{layers}}$ is in green.}
   \label{image épluchage par couches}
\end{figure}

To be consistent with the notation of \cite{StFlour}, \cite{BBCK} and others, we add a subscript $\infty$ (resp. $*$) to the processes $(P(n))_{n\ge 0},(T(n))_{n\ge 0},(H(n))_{n\ge 0}$ under the law $\P_\infty^{(\ell)}$ (resp. $\P^{(\ell)}$). Recall that the peeling exploration under $\P^{(\ell)}$ satisfies a spatial Markov property in the sense that for all $n\ge 0$, conditionally on the explored region until time $n$, the maps which fill in the holes are independent Boltzmann maps, see e.g.\@ Proposition 4.6 of \cite{StFlour} for more details.

We finally recall from \cite{BBCK} the scaling limit of the perimeter process under $\P^{(\ell)}$ and $\P^{(\ell)}_\infty$ as $\ell \to \infty$. Let $(\Upsilon^\uparrow(t))_{t\ge 0}$ be a symmetric Cauchy process conditioned to stay positive and starting at $1$. More precisely, $(\Upsilon^\uparrow(t))_{t\ge 0}$ is the Doob $h$-transform of the Lévy process with no drift and no Brownian part, of Lévy measure {$(dx/|x|^2){\bf 1}_{x\neq 0}$} starting at one, using the function $h(x)= \sqrt{x}$. We recall from Proposition 6.6 of \cite{BBCK} the following convergences in distribution for the $J_1$ topology of Skorokhod:
\begin{align}
&\text{under } \P^{(\ell)}_\infty, \qquad \left(\frac{P_\infty(\lfloor \ell t \rfloor )}{\ell}\right)_{t\ge 0}
\mathop{\longrightarrow}\limits_{\ell \to \infty}^{(\mathrm{d})}
\left(\Upsilon^\uparrow(p_{\bf q} t)\right)_{t\ge 0}; \label{cvperimetre infini}\\
&\text{under } \P^{(\ell)}, \qquad \left(\frac{P_*(\lfloor \ell t \rfloor )}{\ell}\right)_{t\ge 0}
\mathop{\longrightarrow}\limits_{\ell \to \infty}^{(\mathrm{d})}
\left(X^{(-1)}(\pi p_{\bf q} t)\right)_{t\ge 0}.\label{cvperimetre fini}
\end{align}
Notice here that the scaling limit of the perimeter of the locally largest component (parametrized by the exploration time) corresponds already, up to a linear time change, to the quantum boundary length of the locally largest component during a uniform or $\mathrm{SLE}^{\langle \mu \rangle}_4(-2)$ exploration of a $\mathrm{CLE}_4$-decorated quantum disk, parametrized by the quantum natural distance or the quantum length of the trunk, by Theorem \ref{prop de Markov disques quantiques}.
\subsection{Metric growth for infinite $3/2$-stable maps with increasing perimeter}
Here we establish the scaling limit of $T_\infty$ and $H_\infty$ under $\P_\infty^{(\ell)}$ as $\ell \to \infty$. The scaling limit of $T_\infty$ is very easy to obtain:
\begin{proposition}\label{prop limite T infini}
Using the uniform peeling exploration, we have jointly with (\ref{cvperimetre infini}) the convergence in distribution for the topology of the uniform convergence on compact sets:
$$
\text{under } \P^{(\ell)}_\infty, \qquad \left(T_\infty(\lfloor \ell t \rfloor) \right)_{t\ge 0} \mathop{\longrightarrow}\limits_{\ell \to +\infty}^{(\mathrm{d})}  \left(\int_0^{t} \frac{1}{ 2  \Upsilon^\uparrow(p_{\bf q} s)} ds \right)_{t\ge 0}.
$$
\end{proposition}
\begin{proof}
Using (\ref{eq distance fpp}) and the fact that the $\mathcal{E}_n$'s are independent from $P_\infty$, we have for all $t\ge 0$,
$$
\E^{(\ell)}_\infty \left( \left. \sum_{n=0}^{\lfloor \ell t \rfloor -1} \left( \frac{\mathcal{E}_n -1}{ 2P_\infty(n)} \right)^2 \right \vert P_\infty \right)  = \sum_{n=0}^{\lfloor \ell t \rfloor -1} \frac{1}{ \left(2P_\infty(n) \right)^2 }\mathop{\longrightarrow}\limits_{\ell \to +\infty}^{(\P)} 0,
$$
where the convergence comes from (\ref{cvperimetre infini}). Moreover, again by (\ref{cvperimetre infini}),
$$
\text{under } \P^{(\ell)}_\infty, \qquad 
\left(  \sum_{n=0}^{\lfloor \ell t \rfloor -1}  \frac{1}{ 2P_\infty(n)}  \right)_{t\ge 0}
\mathop{\longrightarrow}\limits_{\ell \to \infty}^{(\mathrm{d})}
\left(\int_0^{t} \frac{1}{ 2  \Upsilon^\uparrow(p_{\bf q} s)} ds \right)_{t\ge 0}.
$$
This ends the proof since $t \mapsto \int_0^t 1/(2\Upsilon^\uparrow(p_{\bf q} s))ds$ is a continuous and increasing process.
\end{proof}

For the scaling limit of $H_\infty$, a $\log \ell$ factor appears:
\begin{proposition}\label{prop limite H infini}
For the peeling by layers exploration, we have the convergence in distribution jointly with (\ref{cvperimetre infini}) for the topology of the uniform convergence on compact sets:
$$
\text{under } \P^{(\ell)}_\infty, \qquad \left(\frac{H_\infty(\lfloor \ell t \rfloor)}{\log \ell} \right)_{t\ge 0} \mathop{\longrightarrow}\limits_{\ell \to +\infty}^{(\mathrm{d})}  \left(\int_0^{t} \frac{p_{\bf q}}{ 2  \Upsilon^\uparrow(p_{\bf q} s)} ds \right)_{t\ge 0}.
$$
\end{proposition}
To prove the above result, we will use the inverse $\theta$ of $H_\infty$ defined by
$$
\forall k \ge 0, \qquad \theta_k = \inf\left\{n \ge 0, H_\infty(n) \ge k\right\}.
$$
We will rely on the following lemma from \cite{BCM} {which gives an upper bound for the number of peeling steps needed to complete $k$ layers.}
\begin{lemma}\label{lemme theta k}(Corollary 2 of \cite{BCM})
For all $\vp \in (0,1)$, for all $k \in \N$, 
$$
\lim_{\ell \to \infty}
\P^{(\ell)}_\infty
\left( \frac{p_{\bf q} \log \ell}{2\ell} \frac{\theta_k}{k} \le 1+\vp \right) = 1.
$$
\end{lemma}
{We will also rely on the following result of \cite{BCM} to upperbound in expectation an interpolated version of the height process.} {We denote by $D_\infty(n)\in \lb 1, 2P_\infty(n) \rb$ the number of edges on the boundary of the explored region at time $n$ which are at height $H_\infty(n)$. Let $\vp'>0$.  Let $f:[0,1] \to [0,1]$ be a continuous non-increasing function such that $f(0)=1$, $f(1)=0$. For all $n\ge 0$, we set
	\begin{equation}\label{eq delta H f}
	\Delta H_\infty^f(n) \coloneqq H_\infty(n+1)+ f\left(\frac{D_\infty(n+1)}{2P_\infty(n+1)}\right)-H_\infty(n)- f\left(\frac{D_\infty(n)}{2P_\infty(n)}\right).
	\end{equation}
	Note that $\Delta H_\infty^f(n)\ge 0$ since either the peeling by layers algorithm completes a layer and in this case $H_\infty(n+1)= H_\infty(n)+1$ while $D_\infty(n+1)= 2P_\infty(n+1)$, or the height process remains constant and $D_\infty(n+1) \le D_\infty(n)$. }
\begin{lemma}\label{lemme 6 BCM}(Lemma 6 of \cite{BCM}) For all $\vp'>0$, {if $f$ is a twice continuously differentiable non-increasing function such that $f(0)=1$, $f(1)=0$, $f'(0)=f'(1)=f''(0)=f''(1)=0$ and $0\le -f'(x) \le 1+ \vp'$ for all $x \in [0,1]$, then} there exists a constant $C'>0$ such that for all $\ell \ge 1$ and $n\ge 0$,
{$$
		\E^{(\ell)}_\infty\left(\left.\Delta H^f_\infty(n)\right\vert P_\infty(n)\right) \le (1+3\vp')^3 \frac{p_{\bf q}}{2}  \frac{\log P_\infty(n) + C'}{P_\infty(n)}.$$}
\end{lemma}
{The proof of Proposition \ref{prop limite H infini}} uses the same techniques as the proof of Proposition 6 in \cite{BCM}, which gives a lower bound for $H_\infty(n)/(\log n)^2$ under $\P^{(1)}_\infty$ as $n \to \infty$, {and then we will use Lemma \ref{lemme 6 BCM}}. {Let us introduce some notation.} Let $C>0$. For all $\ell \ge 1$ and $k\ge 0$, set $t_k^{(\ell)} = \lfloor C \frac{k\ell}{\log \ell} \rfloor$. Notice that 
\begin{equation}\label{eq delta t k}
\Delta t^{(\ell)}_k \coloneqq t_{k+1}^{(\ell)} - t^{(\ell)}_k = (C+o(1)) \frac{\ell}{\log \ell}
\end{equation}
when $\ell \to \infty$ where the $o(1)$ is uniform in $k$. Note also that for all real $s>0$ we have 
\begin{equation}\label{eq t s log ell}
t^{(\ell)}_{\lfloor \frac{s \log \ell}{C} \rfloor} = \lfloor \ell s \rfloor (1+o(1))
\end{equation}
as $\ell \to \infty$ where the $o(1)$ is non-positive.

Let $\delta \in (0,1)$. For all $\ell\ge 1$ and $k\ge 0$, let
\begin{equation}\label{eq def eta}
\eta^{(\ell)}_k = \left(H_\infty(t^{(\ell)}_{k+1})-H_\infty(t^{(\ell)}_{k}) - C \frac{p_{\bf q} \ell}{2P_\infty(t_k^{(\ell)})}\right) {\bf 1}_{P_\infty(t^{(\ell)}_k)/\ell \in [\delta,1/\delta]}.
\end{equation}
{Let us show that the negative part of $\eta^{ (\ell)}_k$ is bounded with high probability using Lemma \ref{lemme theta k}.}
{\begin{lemma}\label{lemme eq eta borne}
	We have
	$$
		\lim_{\ell \to \infty} \inf_{k\ge 0} \P^{(\ell)}_\infty \left(\eta^{(\ell)}_k  \ge -2\right) =1.
	$$
\end{lemma}}
\begin{proof}
Take $\vp \in (0, 2\delta/(C p_{\bf q}))$. 
Let ${N}^{(\ell)}_k = \lfloor (1-\vp)(C p_{\bf q}\ell)/(2P_\infty(t^{(\ell)}_k)) \rfloor$. Then one can see that
$$
	\left\{\frac{P_\infty(t^{(\ell)}_k)}{\ell} \not\in [\delta, 1/\delta]\right\}
	\cup
	\left(
	 \left\{\frac{P_\infty(t^{(\ell)}_k)}{\ell} \in [\delta, 1/\delta]\right\} \cap \left\{ {N}^{(\ell)}_k \le H_\infty(t^{(\ell)}_{k+1})-H_\infty(t^{(\ell)}_{k}) \right\}
	\right) \\
	\subset \left\{ \eta^{(\ell)}_k \ge - 2\right\}.
$$
Besides, by (\ref{eq delta t k}), for all $\ell$ large enough, for all $k\ge 0$, on the event $\{{P_\infty(t^{(\ell)}_k)}/{\ell} \in [\delta, 1/\delta] \}$, we have
\begin{equation}\label{eq delta t k sur N}
\frac{\Delta t^{(\ell)}_k}{{N}^{(\ell)}_k} \ge \frac{2P_\infty(t^{(\ell)}_k)}{(1-\vp/2)p_{\bf q} \log P_\infty({t^{(\ell)}_k})}.
\end{equation}
We define for all $\ell \ge 1$ and $k\ge0$ the stopping times $\theta_{t_k^{(\ell)}, 0}  = t_k^{(\ell)}$ and for all $i\ge 1$, 
$$\theta_{t_k^{(\ell)}, i} = \inf \left\{n \ge 0 ; \ H_\infty(n+t_k^{(\ell)}) \ge H_\infty({t_k^{(\ell)}}) +i \right\}.$$
Then by Lemma \ref{lemme theta k}, for all $\gamma \in (0,1)$, for $\ell$ large enough, for all $k\ge 0$,
\begin{align*}
\P^{(\ell)}_\infty &\left( \left\{\frac{P_\infty(t^{(\ell)}_k)}{\ell} \in [\delta, 1/\delta]\right\} \cap \left\{ {N}^{(\ell)}_k \le H_\infty(t^{(\ell)}_{k+1})-H_\infty(t^{(\ell)}_{k})  \right\} \right) \\
 &=
\P^{(\ell)}_\infty \left( \left\{\frac{P_\infty(t^{(\ell)}_k)}{\ell} \in [\delta, 1/\delta]\right\} \cap \left\{ \frac{\theta_{t_k^{(\ell)},{N}^{(\ell)}_k} }{ {N}^{(\ell)}_k } \le  \frac{\Delta t_k^{(\ell)}}{ {N}^{(\ell)}_k} \right\} \right) \\
&\ge\P^{(\ell)}_\infty \left(  \left\{\frac{P_\infty(t^{(\ell)}_k)}{\ell} \in [\delta, 1/\delta]\right\} 
\cap \left\{ \frac{\theta_{t_k^{(\ell)},{N}^{(\ell)}_k} }{ {N}^{(\ell)}_k } \le \frac{2P_\infty(t^{(\ell)}_k)}{(1-\vp/2)p_{\bf q} \log P_\infty({t^{(\ell)}_k})}  \right\}   
 \right) \qquad \text{by (\ref{eq delta t k sur N})}\\
&\ge (1-\gamma) \P^{(\ell)}_\infty \left(\frac{P_\infty(t^{(\ell)}_k)}{\ell} \in [\delta, 1/\delta]  \right),
\end{align*}
the last inequality being uniform in $k$ since, under the event $\{{P_\infty(t^{(\ell)}_k)}/{\ell} \in [\delta, 1/\delta]\}$, the random variable ${N}^{(\ell)}_k$ is bounded above by $(Cp_q) / (2 \delta) $. Thus, by taking $\gamma$ arbitrarily small, {we obtain the statement of the lemma.}
\end{proof}
Let us now finish the proof of Proposition \ref{prop limite H infini}.
\begin{proof}[Proof of Proposition \ref{prop limite H infini}]
{Recall that $C>0$, $t_k^{(\ell)} = \lfloor C \frac{k\ell}{\log \ell} \rfloor$ and that $\delta \in (0,1)$.} Let $s\ge 0$. Let 
$$\mathcal{H}_{\delta,{C,}\ell} (s)= 
\frac{1}{\log \ell} \sum_{k=0}^{\lfloor (s \log \ell)/C \rfloor} (H_\infty(t^{(\ell)}_{k+1})-H_\infty(t^{(\ell)}_k)) {\bf 1}_{\forall n \in \lb t_k^{(\ell)}, t_{k+1}^{(\ell)} \rb , \  P_\infty(n)/\ell \in [\delta, 1/\delta]} .
$$
By (\ref{eq t s log ell}), we have 
\begin{equation}\label{minoration hauteur}
\mathcal{H}_{\delta,{C,} \ell}(s) \le \frac{1}{\log \ell} H_\infty(\lfloor \ell s \rfloor)
\end{equation}
Moreover, by (\ref{eq t s log ell}) and (\ref{cvperimetre infini}), for all $\upsilon >0$, with probability $1-o(1)$ as $\delta \to 0$, for all $\ell$ large enough,
\begin{equation}\label{majoration hauteur}
\frac{1}{\log \ell} H_\infty(\lfloor \ell s \rfloor ) \le \mathcal{H}_{\delta,{C,} \ell}(s + \upsilon).
\end{equation}
{Recall the definition of $\eta^{(\ell)}_k$ given in \eqref{eq def eta}.} One can write
\begin{align*}
 \mathcal{H}_{\delta, {C,}\ell} (s) = &\frac{1}{\log \ell}\sum_{k=0}^{\lfloor (s \log \ell)/C\rfloor}
C \frac{p_{\bf q} \ell}{2P_\infty(t_k^{(\ell)})}  {\bf 1}_{\forall n \in \lb t^{(\ell)}_k, t^{(\ell)}_{k+1} \rb , \ P_\infty(n)/\ell \in [\delta, 1/\delta]} 
\\
&+\frac{1}{\log \ell}\sum_{k=0}^{\lfloor (s \log \ell)/C\rfloor} \eta^{(\ell)}_k {\bf 1}_{\forall n \in \lb t^{(\ell)}_k, t^{(\ell)}_{k+1} \rb , \ P_\infty(n)/\ell \in [\delta, 1/\delta]} .
\end{align*}
Consider the first term. 
\begin{align}
{X_{\delta, C,\ell}(s)} &\coloneqq\frac{1}{ \log \ell} \sum_{k=0}^{ \lfloor (s\log \ell )/ C  \rfloor}  C
 \frac{p_q \ell }{ 2 P_\infty({t_k^{(\ell)}})} {\bf 1}_{\forall n \in \lb t^{(\ell)}_k, t^{(\ell)}_{k+1} \rb , \ P_\infty(n)/\ell \in [\delta, 1/\delta]}\notag \\
&= \frac{1}{\log \ell} \int_0^{ \lfloor ({s\log \ell })/{ C } \rfloor +1} C \frac{p_q \ell }{ 2 P_\infty({t_{\lfloor v \rfloor}^{(\ell)}})}  {\bf 1}_{ \forall n \in \lb t^{(\ell)}_{\lfloor v \rfloor}, t^{(\ell)}_{\lfloor v \rfloor +1} \rb , \ {P_\infty(n) }/{ \ell} \in [\delta, \frac{1}{\delta}]} dv \notag\\
&=O\left(\frac{1}{\log \ell}\right)+ \int_0^{s }  \frac{p_q \ell}{ 2 P_\infty\left({t_{\lfloor ({u \log \ell })/{ C} \rfloor}^{(\ell)}}\right)} 
{\bf 1}_{\forall n \in \lb t_{\lfloor ({u \log \ell })/{ C} \rfloor}^{(\ell)}, t_{\lfloor ({u \log \ell })/{ C} \rfloor+1}^{(\ell)} \rb , \  {P_\infty\left(n\right)}/{ \ell} \in [\delta, \frac{1}{ \delta}]} du \notag \\
&\mathop{\longrightarrow}\limits_{\ell \to + \infty}^{(\mathrm{d})} { \mathcal{H}_\delta(s) \coloneqq}\int_0^s \frac{p_{\bf q}}{2\Upsilon^\uparrow(p_{\bf q} u)} {\bf 1}_{\Upsilon^\uparrow(p_{\bf q}u) \in [\delta, 1/\delta]} du \label{eq convergence premier terme H delta}
\end{align}
{jointly with (\ref{cvperimetre infini})} by (\ref{eq t s log ell}) and by dominated convergence. {By the $O(1/\log \ell)$ term in the third line, we mean a random variable which is bounded in absolute value by $\widetilde{C}/\log \ell$, where $\widetilde{C}$ is a constant depending on $C$ and $\delta$.}

For the second term, {note that since $H_\infty$ is non-decreasing, the negative part of $\eta^{(\ell)}_k$ is upperbounded as follows: for all $\ell\ge 1, k\ge 0$,
\begin{equation}\label{eq majoration eta moins}
\left( \eta^{(\ell)}_k\right)_-\le 2{\bf 1}_{\eta^{(\ell)}_k \ge -2} + {\bf 1}_{\eta^{(\ell)}_k <-2}  C \frac{p_{\bf q} \ell}{2P_\infty(t_k^{(\ell)})}{\bf 1}_{P_\infty(t^{(\ell)}_k)/\ell \in [\delta,1/\delta]} \le 2 +  {\bf 1}_{\eta^{(\ell)}_k <-2}  C \frac{p_{\bf q}}{2\delta},
\end{equation}
hence for all $\ell\ge 1, k\ge 0$,}
$$
\E^{(\ell)}_\infty \left((\eta^{(\ell)}_k)_-\right) \le 2 + \P^{(\ell)}_\infty\left(\eta^{(\ell)}_k < -2\right) C \frac{p_{\bf q}}{2\delta}.
$$
Hence, by {Lemma \ref{lemme eq eta borne}},
$$
{\limsup_{\ell \to \infty} }\sup_{k\ge 0} \E^{(\ell)}_\infty \left((\eta^{(\ell)}_k)_-\right) \le 2.
$$
{As a result, if we set
$$
Y_{\delta, C,\ell}(s) \coloneqq \frac{1}{\log \ell}\sum_{k=0}^{\lfloor (s \log \ell)/C\rfloor} \left(\eta^{(\ell)}_k\right)_- {\bf 1}_{\forall n \in \lb t^{(\ell)}_k, t^{(\ell)}_{k+1} \rb , \ P_\infty(n)/\ell \in [\delta, 1/\delta]} ,
$$
then
$$
\limsup_{\ell \to \infty} \E^{(\ell)}_\infty \left(Y_{\delta, C,\ell}(s) \right) \le \frac{2s}{C} \qquad {\text{and}} \qquad \forall \ell \ge 1, \quad \mathcal{H}_{\delta,C,\ell}(s) \ge X_{\delta,C,\ell}(s)- Y_{\delta,C,\ell}(s).
$$
Moreover, by \eqref{eq majoration eta moins}, one can see that the sequence of non-negative random variables $(Y_{\delta, C, \ell}(s))_{\ell \ge 1}$ is upperbounded by $(s/C)(2+ Cp_{\bf q}/(2\delta))$. Let $(\ell_m)_{m\ge 0}$ be an increasing sequence of integers. By taking a subsequence again denoted by $(\ell_m)_{m\ge 0}$, we may assume that $Y_{\delta,C, \ell_m}(s)$ converges in distribution jointly with \eqref{cvperimetre infini} and \eqref{eq convergence premier terme H delta}. By Skorokhod's representation theorem, we may assume that the random variables $\mathcal{H}_{\delta,C,\ell}(s)$, $X_{\delta,C,\ell}(s)$ and $Y_{\delta,C,\ell}(s)$ are defined on the same probability space and that the convergence \eqref{eq convergence premier terme H delta} holds almost surely together with the convergence of $Y_{\delta,C, \ell_m}(s)$ towards a non-negative random variable $Y_{\delta,C}(s)$. Note that since $Y_{\delta,C, \ell_m}$ is bounded it also converges in $\mathrm{L}^1$. As a result, a.s.
\begin{equation}\label{eq lowerbound H delta}
	\liminf_{m \to \infty} \mathcal{H}_{\delta,C, \ell_m} (s) \ge   \mathcal{H}_\delta(s)  -Y_{\delta,C}(s) \qquad \text{and} \qquad \E \left( Y_{\delta,C} \right) \le \frac{2s}{C}.
\end{equation}}

Let us then give an upper bound in expectation of $\mathcal{H}_{\delta,{C,}\ell}(s)$. {Let $\vp'>0$. Let $f: [0,1]\to [0,1]$ satisfying the assumptions of Lemma \ref{lemme 6 BCM}. Recall the definition of $\Delta H^f_\infty(n)$ in \eqref{eq delta H f}. By summing over $n \in \lb t^{(\ell)}_k, t^{(\ell)}_{k+1} -1\rb$, using that $\Delta H_\infty^f(n) \ge 0$ for all $n\ge 0$, we get
\begin{align*}
	\E^{(\ell)}_\infty &\left(\left(H_\infty(t^{(\ell)}_{k+1}) - H_\infty(t^{(\ell)}_k) + f\left(\frac{D_\infty(t^{(\ell)}_{k+1})}{2P_\infty(t^{(\ell)}_{k+1})} \right) -f\left(\frac{D_\infty(t^{(\ell)}_{k})}{2P_\infty(t^{(\ell)}_{k})} \right) \right){\bf 1}_{\forall n \in \lb t_k^{(\ell)}, t_{k+1}^{(\ell)} \rb , \  P_\infty(n)/\ell \in [\delta, 1/\delta]}\right)\\
	&\le\sum_{n=t^{(\ell)}_k}^{t^{(\ell)}_{k+1}-1} \E^{(\ell)}_\infty\left( \Delta H^f_\infty(n) {\bf 1}_{P_\infty(n) /\ell \in [\delta,1/\delta]} \right).
\end{align*}}
{Thus, using Lemma \ref{lemme 6 BCM} and the fact that $0 \le f \le 1$, we deduce that} for all $\ell \ge 1$ and $k\ge 0$,
\begin{align*}
\E^{(\ell)}_\infty &\left( \left(H_\infty(t^{(\ell)}_{k+1})-H_\infty(t^{(\ell)}_k)\right) {\bf 1}_{\forall n \in \lb t_k^{(\ell)}, t_{k+1}^{(\ell)} \rb , \  P_\infty(n)/\ell \in [\delta, 1/\delta]} \right)\\
 &\le 2+ (1+3\vp')^3 \frac{p_{\bf q}}{2} \sum_{n=t^{(\ell)}_k}^{t^{(\ell)}_{k+1} -1}
\E^{(\ell)}_\infty\left(\frac{\log P_\infty(n)+C'}{P_\infty(n)} {\bf 1}_{P_\infty(n)/\ell \in [\delta, 1/\delta]}
\right).
\end{align*}
As a result, by (\ref{eq t s log ell}),
$$
	\E^{(\ell)}_\infty\left( \mathcal{H}_{\delta,{C},\ell}(s) \right)
	\le \frac{1}{\log \ell} \left( 2\left\lfloor \frac{s \log \ell}{C} \right\rfloor+2+ (1+3\vp')^3 \frac{p_{\bf q}}{2} \sum_{n=0}^{\lfloor s \ell \rfloor -1} \E^{(\ell)}_\infty\left(\frac{\log P_\infty(n)+C'}{P_\infty(n)} {\bf 1}_{P_\infty(n)/\ell \in [\delta, 1/\delta]}\right)
	\right).
$$
Thus, using (\ref{cvperimetre infini}), by dominated convergence as $\ell \to \infty$, and then by letting $\vp' \to 0$, we get that for all $\delta >0$, 
\begin{equation}\label{eq majoration esp delta}
	\limsup_{\ell \to \infty} \E^{(\ell)}_\infty \left( \mathcal{H}_{\delta,{C},\ell}(s) \right)
	\le  \E \left(\int_0^{s} \frac{p_{\bf q}}{ 2  \Upsilon^\uparrow(p_{\bf q} v)} {\bf 1}_{\Upsilon^\uparrow(p_{\bf q} v) \in [\delta, 1/\delta]}dv \right) {+ \frac{2s}{C} = \E \left( \mathcal{H}_\delta(s) \right) + \frac{2s}{C}.}
\end{equation}
Besides, by the lower bound \eqref{eq lowerbound H delta} and by Fatou's lemma, 
$$
{\E \mathcal{H}_{\delta}(s) - \frac{2s}{C} \le \E \mathcal{H}_\delta(s) - \E Y_{\delta,C}(s) \le \E\left(\liminf_{m\to \infty} \mathcal{H}_{\delta, C,  \ell_m}(s) \right)\le}
\liminf_{m\to \infty} \E \mathcal{H}_{\delta,{C}, \ell_m}(s),
$$
{so that combining this with \eqref{eq majoration esp delta} we get
	\begin{equation}\label{eq controle esperance H delta}
	\E \mathcal{H}_{\delta}(s) - \frac{2s}{C} \le \liminf_{m\to \infty} \E \mathcal{H}_{\delta,C, \ell_m}(s) \le 
	\limsup_{m\to \infty} \E \mathcal{H}_{\delta,C, \ell_m}(s) \le \E \mathcal{H}_\delta(s) + \frac{2s}{C}.
	\end{equation}
}
Moreover, notice that for all $\vp'>0$, by dominated convergence and by (\ref{eq lowerbound H delta}),
\begin{align*}
\E &\left((\mathcal{H}_\delta(s)-\mathcal{H}_{\delta,{C},  \ell_m}(s){- Y_{\delta,C}(s)})_+\right) \\
&\le \vp' + \E\left( (\mathcal{H}_\delta(s)-\mathcal{H}_{\delta,C, \ell_m}(s){- Y_{\delta,C}(s)}){\bf 1}_{\mathcal{H}_\delta(s)-\mathcal{H}_{\delta,{C}, \ell_m}(s){- Y_{\delta,C}(s)}>\vp'}\right) \\
&\le \vp' +\E\left( \mathcal{H}_\delta(s){\bf 1}_{\mathcal{H}_\delta(s)-\mathcal{H}_{\delta,{C}, \ell_m}(s){- Y_{\delta,C}(s)}>\vp'}\right) \mathop{\longrightarrow}\limits_{m \to +\infty} \vp'.
\end{align*}
Now, since for all $x\in \R$, $|x|=-x+2x_+$, {using the second part of \eqref{eq lowerbound H delta}, the above upper bound and \eqref{eq controle esperance H delta}} we get that 
{
	\begin{align*}
		&\limsup_{m\to \infty} \E\left\vert\mathcal{H}_\delta(s) - \mathcal{H}_{\delta,C, \ell_m}(s)\right\vert \\&\le 
		\E(Y_{\delta,C}(s)) + \limsup_{m\to \infty}\E\left\vert\mathcal{H}_\delta(s) - \mathcal{H}_{\delta,C, \ell_m} (s)- Y_{\delta,C}(s)\right\vert \\
		&\le \frac{2s}{C} + 2 \limsup_{m\to \infty}\E\left(\left(\mathcal{H}_\delta(s) - \mathcal{H}_{\delta,C, \ell_m} (s)- Y_{\delta,C}(s)\right)_+\right) - \liminf_{m\to \infty} \E\left(\mathcal{H}_\delta(s) - \mathcal{H}_{\delta,C, \ell_m} (s)- Y_{\delta,C}(s)\right)\\
		&\le \frac{6s}{C} .
	\end{align*}
	By taking $C$ large and then $\delta$ small,} using  (\ref{minoration hauteur}) and (\ref{majoration hauteur}), {we get that for any increasing sequence of integers $(\ell_m)_{m\ge 0}$ the sequence of random variables $H_\infty(\lfloor \ell_m s \rfloor )/\log \ell_m$ converges in distribution towards $\int_0^s p_{\bf q}/(2 \Upsilon^\uparrow(p_{\bf q}u))du$ along a subsequence, jointly with (\ref{cvperimetre infini}).} This concludes the proof. 
\end{proof}
\subsection{Large $3/2$-stable maps and $\mathrm{CLE}_4$-decorated quantum disks}\label{sous-section cartes et CLE}

We first derive the analogues of Propositions \ref{prop limite T infini} and \ref{prop limite H infini} for finite $3/2$-stable maps by absolute continuity. {Recall from Subsection \ref{rappels épluchage} that $H_*$ is the height process under $\P^{(\ell)}$ and that $T_*$ records the fpp distance from the root face to the boundary of the unexplored region along the uniform peeling exploration under $\P^{(\ell)}$.}

\begin{corollary}\label{corollaire limite des hauteurs}
For the uniform peeling, jointly with (\ref{cvperimetre infini}) we have the convergence
$$
\text{under } \P^{(\ell)}, \qquad
\left(T_*({\lfloor \ell t\rfloor }) \right)_{t\ge 0} \mathop{\longrightarrow}\limits_{\ell \to + \infty}^{(\mathrm{d})} \left(\int_0^t {1\over 2X^{(-1)}(\pi p_{\bf q}s)} ds \right)_{t\ge 0}.
$$
For the peeling by layers, jointly with (\ref{cvperimetre infini}) we have the convergence
$$
\text{under } \P^{(\ell)}, \qquad
\left({H_*({\lfloor \ell t\rfloor }) \over \log \ell}\right)_{t\ge 0} \mathop{\longrightarrow}\limits_{\ell \to +\infty}^{(\mathrm{d})} \left(\int_0^t {p_{\bf q}\over 2X^{(-1)}(\pi p_{\bf q}s)} ds \right)_{t\ge 0}.
$$

\end{corollary}

Then the convergences (\ref{cv rayon fpp}) and (\ref{cv rayon graphe}) are clearly a consequence of the above corollary.

\begin{proof}
{Let us prove the statement for $H_*$ (the same proof works for $T_*$ using Proposition \ref{prop limite T infini}).} We will mimic the proof of Proposition 6.6 of \cite{BBCK}, but let us write the details. 

By {Equation (39) of} Proposition 6.3 of \cite{BBCK}, for all $\ell \ge 1$, the Radon-Nikodym derivative of the law of the processes $(P_*,H_*)$ until some time $n\ge0$ under $\P^{(\ell)}$ with respect to the law of $(P_\infty,H_\infty)$ until time $n$ under $\P^{(\ell)}_\infty$ is given by $f^\uparrow(\ell)/f^\uparrow(P_\infty(n))$ times the indicator function that $P_\infty$ does not make any negative jump of size larger than its half, where $f^\uparrow (\ell) \coloneqq h^\uparrow(\ell)/(W^{(\ell)} c_{\bf q}^{-\ell})$, {where $W^{(\ell)}$ is the partition function of planar maps of perimeter $2\ell$, $c_{\bf q}$ appears in \eqref{eq critique de type deux}} and $h^\uparrow(\ell)= 2\ell 2^{-\ell} \binom{2\ell}{\ell}$. 

{More precisely, let $T>0$. Denote by $\mathbb{D}([0,T])$ the set of real-valued càdlàg functions on $[0,T]$ equipped with the $J_1$ Skorokhod topology. Let $\vp>0$. Set
$$
U_\vp \coloneqq \left\{(x(t))_{t\in [0,T]} \in \mathbb{D}([0,T]); \ x(T)<\vp \ \text{or} \ \exists t \in (0,T], \ x(t)< \frac{x(t-)}{2}\right\}.
$$
Note that the closure of $U_\vp$ satisfies the inclusion
$$
\overline{U_\vp} \subseteq \left\{(x(t))_{t\in [0,T]} \in \mathbb{D}([0,T]); \ x(T)\le \vp \ \text{or} \ \exists t \in (0,T], \ x(t)\le \frac{x(t-)}{2}\right\}.
$$
Let $F: \mathbb{D}([0,T], \R)^2 \to \R$ be a continuous bounded function vanishing on the open set $\mathbb{D}([0,T]) \times U_\vp$. Then,
\begin{align*}
	\E^{(\ell)} F &\left( \left( \frac{H_*(\lfloor \ell t \rfloor)}{\log \ell} \right)_{t \in [0,T]}, \left( \frac{P_*(\lfloor \ell t \rfloor)}{\ell} \right)_{t \in [0,T]} \right) \\&= \E^{(\ell)}_\infty \left( \frac{f^\uparrow(\ell)}{f^\uparrow(P_\infty(\lfloor \ell T \rfloor))} F\left( \left( \frac{H_\infty(\lfloor \ell t \rfloor)}{\log \ell} \right)_{t \in [0,T]}, \left( \frac{P_\infty(\lfloor \ell t \rfloor)}{\ell} \right)_{t \in [0,T]} \right)\right)\\
	&\mathop{\longrightarrow}\limits_{\ell \to \infty}
	\E \left( \left(\Upsilon^\uparrow (p_{\bf q} T)\right)^{-5/2} F \left( \left(\int_0^{t} \frac{p_{\bf q}}{ 2  \Upsilon^\uparrow(p_{\bf q} s)} ds\right)_{t \in [0,T]}, (\Upsilon^\uparrow(p_{\bf q}t))_{t\in [0,T]} \right) \right),
\end{align*}
}
where the convergence comes from Proposition \ref{prop limite H infini}, together with the convergence in distribution of $f^\uparrow(\ell)/f^\uparrow(P_\infty(\lfloor \ell T\rfloor))$ towards $\Upsilon^\uparrow(p_{\bf q}t)^{-5/2}$ which comes from (\ref{eq critique de type deux}) {and \eqref{cvperimetre infini}} {and from the fact that $h^\uparrow(\ell) \sim 2 \sqrt{\ell/\pi}$ as $\ell \to \infty$.}

 Finally, identifying the Radon-Nikodym derivative of the law of $(X^{(-1)}(\pi p_{\bf q}s))_{0\le t \le T}$ with respect to the law of $(\Upsilon^\uparrow(p_{\bf q} t))_{0\le t \le T}$ by Proposition 4.1 of \cite{BBCK}, {one obtains the identity
 \begin{align*}
 	\E &\left( \left(\Upsilon^\uparrow (p_{\bf q} T)\right)^{-5/2} F \left( \left(\int_0^{t} \frac{p_{\bf q}}{ 2  \Upsilon^\uparrow(p_{\bf q} s)} ds\right)_{t \in [0,T]}, (\Upsilon^\uparrow(p_{\bf q}t))_{t\in [0,T]} \right) \right)\\
 	&=\E \left(  F \left( \left(\int_0^{t} \frac{p_{\bf q}}{ 2  X^{(-1)}(\pi p_{\bf q}s)} ds\right)_{t \in [0,T]}, (X^{(-1)}(\pi p_{\bf q}t))_{t\in [0,T]} \right) \right).
 \end{align*}
 It is now elementary but tedious to get the desired convergence in law. Let us still give the details using the Portmanteau theorem. Let $A_\vp$ be the closed subset of $\mathbb{D}([0,T])$ defined by
 $$
 A_\vp \coloneqq \left\{(x(t))_{t \in [0,T]} \in \mathbb{D}([0,T]); \ x(T) \ge 2\vp \text{ and } \forall t \in [0,T], \ x(t)\ge \frac{1+\vp}{2} x(t-)\right\}.
 $$
 Let $B$ be a closed subset of $\mathbb{D}([0,T])^2$ and let $B^\delta$ be its $\delta$-neighborhood in $\mathbb{D}([0,T])^2$ for some distance which induces the $J_1$ Skorokhod topology. Note that $(B\cap(\mathbb{D}([0,T]) \times A_\vp) )\cap ((\mathbb{D}([0,T])^2 \setminus B^\delta)\cup (\mathbb{D}([0,T]) \times\overline{U_\vp})) = \emptyset $, so that by Urysohn's lemma, there exists a continuous function $F: \mathbb{D}([0,T])^2 \to [0,1]$ which vanishes on the closed set $((\mathbb{D}([0,T])^2 \setminus B^\delta)\cup (\mathbb{D}([0,T]) \times\overline{U_\vp}))$ and is $1$ on the closed set $B\cap(\mathbb{D}([0,T]) \times A_\vp) $. Let us write
 $$
 (X_\ell,Y_\ell) \coloneqq \left(\left( \frac{H_*(\lfloor \ell t \rfloor)}{\log \ell} \right)_{t \in [0,T]}, \left( \frac{P_*(\lfloor \ell t \rfloor)}{\ell} \right)_{t \in [0,T]}\right) 
 $$ and
 $$ (X,Y)\coloneqq  \left( \left(\int_0^{t} \frac{p_{\bf q}}{ 2  X^{(-1)}(\pi p_{\bf q}s)} ds\right)_{t \in [0,T]}, (X^{(-1)}(\pi p_{\bf q}t))_{t\in [0,T]} \right).
 $$
 Then, by the above convergence,
 one can see that
 \begin{align*}
 	\limsup_{\ell \to \infty} \P^{(\ell)}\left( (X_\ell, Y_\ell) \in B \right)&\le\limsup_{\ell \to \infty} \P^{(\ell)} (Y_\ell \not\in A_\vp) + \limsup_{\ell\to \infty}\P^{(\ell)}\left( (X_\ell,Y_\ell) \in B\cap(\mathbb{D}([0,T]) \times A_\vp)  \right)\\
 	&\le \limsup_{\ell \to \infty} \P^{(\ell)} (Y_\ell \not\in A_\vp)  + \limsup_{\ell \to \infty} \E^{(\ell)} F(X_\ell,Y_\ell) \\
 	&= \limsup_{\ell \to \infty} \P^{(\ell)} (Y_\ell \not\in A_\vp)  + \E F(X,Y) \\
 	&\le \limsup_{\ell \to \infty} \P^{(\ell)} (Y_\ell \not\in A_\vp)  + \P( (X,Y)\in B^\delta \cap (\mathbb{D}([0,T])\times (\mathbb{D}([0,T]) \setminus \overline{U_\vp}))).
 \end{align*}
 Moreover, it is easy to see that $\sup_{\ell \to \infty} \P^{(\ell)} (Y_\ell \not\in A_\vp)\to 0$ and $\P(Y \in \overline{U}_\vp)\to 0$ as $\vp \to 0$ so that
 $$
 \limsup_{\ell \to \infty} \P^{(\ell)}\left( (X_\ell, Y_\ell) \in B \right) \le \P((X,Y) \in B^\delta).
 $$
 Finally, by letting $\delta \to 0$, we conclude that $\limsup_{\ell \to \infty} \P^{(\ell)}\left( (X_\ell, Y_\ell) \in B \right) \le \P((X,Y) \in B).$ This concludes the proof by the Portmanteau theorem.}
\end{proof}

Let us now consider the whole exploration tree in order to prove Theorem \ref{gros théorème}: instead of filling-in the holes, we run new explorations inside. Let us focus on the dual graph distance $\mathrm{d}^\dagger_{\mathrm{gr}}$, using the peeling by layers algorithm.

Recall {from Section 2.2 of \cite{BBCK}} the recursive definition of the discrete perimeter cell-system $\left( (P_u(k))_{k \ge 0} \right)_{u\in \U}$ associated with a branching peeling exploration of a finite map $\mathfrak{m}$, where $\U = \bigcup_{n \ge 0} (\N^*)^n$ is the Ulam tree. The process $P_\emptyset$ is the (half-)perimeter process of the locally largest cycle. We denote by $(N_j)_{j\ge 1}$ the times of negative jumps of $P_\emptyset$ (usually ranked in the non-increasing order of jump sizes and in the decreasing order of jump times if the jump sizes are equal). Each one corresponds to an event of type $G_{k_1,k_2}$, i.e. to a splitting of the distinguished hole of half-perimeter $P_\emptyset(N_j)$ into two holes of half-perimeters $P_\emptyset(N_j +1)$ and $\Delta_-P_\emptyset(N_j) := P_\emptyset(N_j)- P_\emptyset(N_j+1)-1 \le P_\emptyset(N_j +1)$. We then define $(P_j(k))_{k\ge 0}$ as the perimeter process of the locally largest cycle during the exploration of the submap filling the second hole. We define $P_{uj}$ from $P_u$ and the time of negative jump $N_{uj}$ in the same way for all $u \in \U$. The birth time $B_u$ of the cell $u \in \U$ is defined recursively by $B_\emptyset =0$ and for all $u \in \U, j \in \N^*$, we set $B_{uj}=B_u + N_{uj}$. 
For all $u \in \U$, let $\cal{C}_u$ be the cycle (i.e. the boundary of the hole) created during the birth of the cell $u$, with by convention $\cal{C}_\emptyset$ being the boundary of the root face. For all $u\in \U$, each non-negative jump $P_u(k+1)-P_u(k)\ge 0$ of the perimeter of the cell $u$ at some time $k$ such that $P_u(k)\ge1$ corresponds to the discovery of some face $f$ in the branch $u$. The time of exploration of $f$ is defined by 
\begin{equation}\label{eq temps d'exploration}
n(f)\coloneqq B_{u}+k.
\end{equation} 
Let $X$ be a self-similar Markov process such that $0$ is an absorbing state which tends a.s.\@ towards zero or is a.s.\@ absorbed at zero after a finite time. Let $P_x$ be the law of $X$ starting from $x \in \R_+$. We say that $(X_u)_{u\in \U}$ is a (continuous) cell-system associated with $X$ starting at $x\ge 0$ if $X_\emptyset$ follows the law $P_x$ and for all $u \in \U$, conditionally on $X_u$, if we write $((x_j,\beta_{uj}))_{j\ge 1}$ the family of negative jumps of $X_u$, where $x_j>0$ is the size of the jump and $\beta_{uj}$ is the time of jump, ordered in the non increasing lexicographic order ($x_j\ge x_{j+1}$ and if $x_j=x_{j+1}$ then $\beta_j>\beta_{j+1}$), then the $X_{uj}$ for $j\ge 1$ are independent of laws $P_{x_j}$. For all $u\in\U$, we set $\zeta_u = \inf\{t\ge 0, \ X_u(t)=0\}$. Let $\left( (\cal{X}_u(t))_{t\ge 0} , (\beta_{uj})_{j\ge 1},b_u, \zeta_u\right)_{u \in \U}$ be the continuous cell-system associated with the self-similar Markov process $X^{(-1)}$.

Recall the definition of the $J_1$ Skorokhod topology on the space of families of càdlàg functions indexed by $\U$ denoted by $\mathbb{D}(\R_+, \R)^\U$ : $(f_u^{(n)}(t))_{t\ge 0, u \in \U}$ converges towards $(f_u(t))_{t \ge 0, u \in \U}$ if for all finite subset $F \subset \U$, $(f_u^{(n)}(t))_{t\ge 0, u \in F}$ converges towards $(f_u(t))_{t \ge 0, u \in F}$ in the space $\mathbb{D}(\R_+, \R^F)$ for the $J_1$ Skorokhod topology.

{The following lemma proves that the scaling limit of the discrete cell-system is given by the continuous cell-system.}
\begin{lemma}\label{lemme prérimètre cellule}
The rescaled di{s}crete cell-system associated with $\mathfrak{M}^{(\ell)}$ converges in distribution to the continuous cell-system as a random variable with values in $\mathbb{D}(\R_+, \R)^\U$ equipped with the $J_1$ Skorokhod topology:
\begin{equation}\label{périmètre cellule}
\text{under }\P^{(\ell)}, \qquad \left( \left( \frac{P_u ( \ell t) }{ \ell}\right)_{t\ge 0} \right)_{u \in \U} \mathop{\longrightarrow}\limits_{\ell \to +\infty}^{(\mathrm{d})} \left( \left( \cal{X}_u(\pi p_{\bf q}t)\right)_{t \ge 0} \right)_{u\in \U}.
\end{equation}
\end{lemma}
\begin{proof}
The proof is the same as the proof of Lemma 17 in \cite{BCK}. The first step is to show that the convergence \eqref{cvperimetre fini} holds in the space $\mathbb{D}([0,\infty])$ equipped with Skorokhod's $J_1$ topology. To prove this, it suffices to control $P_*(\ell t)/\ell$ for large $t\ge 0$ uniformly in $\ell$ using the supermartingale $(f^\uparrow(P_*(n)))_{n\ge 0}$ under $\P^{(\ell)}$ coming from Proposition 6.4 of \cite{BBCK}, where $f^\uparrow$ is defined in the proof of Corollary \ref{corollaire limite des hauteurs}. Then \eqref{périmètre cellule} stems from the spatial Markov property and from the fact that the $\mathcal{X}_u$'s for $u \in \U$ have a.s.\@ no common jump times, using Proposition 2.2 page 338 of \cite{JS87}. 
\end{proof}

At this point, we are in position to prove the main result of this section.

\begin{proof}[Proof of Theorem \ref{gros théorème}]
We first extend the convergence (\ref{périmètre cellule}) to the height processes. For every $u \in \U$, we can define $H_u$ the height process of the locally largest cycle during the (filled-in) exploration of the submap of $\mathfrak{m}$ contained in the hole created at the birth time of the cell $u$. Notice that $H_\emptyset  = H_*$. Analogously, in the continuous setting, we write for all $u \in \U$, for all $t\ge 0$,
$$
\cal{H}_u(t) = \int_0^t \frac{p_{\bf q}}{2 \cal{X}_u(\pi p_{\bf q} s)} ds.
$$
A direct consequence of Corollary \ref{corollaire limite des hauteurs} and Lemma \ref{lemme prérimètre cellule} is the following convergence towards a family of continuous increasing processes which holds jointly with (\ref{périmètre cellule}):
\begin{equation}\label{hauteur cellule}
\text{under } \P^{(\ell)}, \qquad \left( \left( \frac{H_u ( \ell t) }{\log \ell}\right)_{t\ge 0} \right)_{u \in \U} \mathop{\longrightarrow}\limits_{\ell \to +\infty}^{(\mathrm{d})} \left( \left( \cal{H}_u(t)\right)_{t \ge 0} \right)_{u\in \U}.
\end{equation}
The topology for which the above convergence holds is the product topology over $\U$ where each factor is equipped with the topology of the uniform convergence on compact sets.

Yet, the processes $H_u$ for $u \in \U$ only describe the height in the map encircled by $\mathcal{C}_u$. In order to describe the height in the whole map $\mathfrak{m}$, we introduce some additional notation. In the discrete setting, for all $u \in \U$, let $H(\cal{C}_u)$ be the minimum of the heights (i.e. the distances to the root face) of the edges of the cycle $\cal{C}_u$. We also define $\widetilde{H}_u$ by induction: $\widetilde{H}_\emptyset = 0$ and for all $u\in \U, j \ge 1$, $\widetilde{H}_{uj} = \widetilde{H}_u + H_u(N_{uj})$. In the continuous setting, we define the birth heights $(b^{\cal{H}}_u)_{u \in \U}$ by $b^{\cal{H}}_\emptyset = 0$ and for all $u \in \U$, for all $j \ge 1$, $b_{uj}^\cal{H} = b^{\cal{H}}_u +\cal{H}_u\left({\beta_{uj}}/({\pi p_{\bf q}})\right)$.

Notice that by (\ref{périmètre cellule}) and (\ref{hauteur cellule}), we have jointly with (\ref{périmètre cellule}) and (\ref{hauteur cellule}) the convergence for the product topology
\begin{equation}\label{hauteurs de naissance}
\text{under } \P^{(\ell)}, \qquad
\left( \frac{H_u(N_{uj})}{\log \ell}
\right)_{u\in \U, j\ge 1}
\mathop{\longrightarrow}\limits_{\ell \to \infty}^{(\mathrm{d})}
\left(
\cal{H}_u\left(\frac{\beta_{uj}}{\pi p_{\bf q}}\right)
\right)_{u \in \U, j\ge 1}.
\end{equation}
Since the peeling algorithm which is used in the exploration is the peeling by layers, for all $u\in \U, j \ge 1$, every edge of $\cal{C}_{uj}$ is at distance $H_u(N_{uj})$ or $H_u(N_{uj})+1$ from $\mathcal{C}_u$. Therefore, for all $u \in \U$, 
we have $
\widetilde{H}_{u} \le H(\cal{C}_u)  \le  \widetilde{H}_{u} + |u|,
$, 
where $|u|$ denotes the length of the word $u$. Thus (\ref{hauteurs de naissance}) entails that, jointly with (\ref{périmètre cellule}) and (\ref{hauteur cellule}),
\begin{equation}\label{eq cv hauteur cycle}
\left( \frac{H(\cal{C}_u)}{\log \ell}\right)_{u \in \U}
\mathop{\longrightarrow}\limits_{\ell \to + \infty}^{(\mathrm{d})} 
\left(b^{\cal{H}}_u \right)_{u \in \U}
\end{equation}
Note that the positive jumps of the $P_u$'s for $u \in \U$ correspond to the (half) perimeters of the faces of $\mathfrak{m}$ minus one (without the root face). 

Besides, notice that by the Markov property of Theorem \ref{prop de Markov disques quantiques}, the cell system 
$$\left((Y_u(t))_{t\ge 0}\right)_{u \in \U} \coloneqq \left( \left( \mathcal{X}_u(\pi t)\right)_{t \ge 0} \right)_{u\in \U}$$
describes the quantum boundary lengths along the exploration tree of the $\mathrm{CLE}_4$ decorated $2$-LQG disk, both for the uniform exploration and for the $\mathrm{SLE}_4^{\langle \mu \rangle} (-2)$ exploration, parametrized by the quantum natural distance for the first one and by the quantum length of the trunk for the second one. In particular, the positive jumps correspond the quantum boundary lengths of the loops, so that the quantum boundary lengths of the loops are all distinct (note that one can also see the fact that the quantum boundary lengths of the loops are distinct as a direct consequence of Theorem 1.2 of \cite{AG23} which gives the law of the quantum boundary lengths of the loops). 
 Moreover, one can see that for all $\vp >0$ there is a finite number of loops of quantum boundary length larger than $\vp$. It can be seen for instance as a consequence of Theorem \ref{prop de Markov disques quantiques} which shows that one can write 
$
\mu^2_{h^2}(\mathbb{D}) \ge  \sum_{j\ge 1} b_j^2 a_j,
$ where the $a_j$'s are i.i.d. of the same law as $\mu^2_{h^2}(\mathbb{D})$ and independent of the $b_j$'s which are the quantum boundary lengths of the loops. 

Then the desired result is a straightforward consequence of the convergences (\ref{périmètre cellule}), (\ref{hauteur cellule}) and (\ref{eq cv hauteur cycle}). Indeed, by Skorokhod's representation theorem, take a probability space where these convergences hold almost surely. Take $\vp>0$. Let $F\subset \U$ be the finite set of $u$'s such that $Y_u$ has at least one jump of size larger than $\vp$. The convergence (\ref{périmètre cellule}) entails the scaling limit of the degrees of the faces to the quantum boundary lengths of the loops and the scaling limit of the exploration times $n(f_i^{(\ell)})$ for $i\ge 1$ towards the quantum natural distances or quantum lengths of the trunk. Finally, the convergences (\ref{hauteur cellule}) and (\ref{eq cv hauteur cycle}) give the scaling limit of the height towards the Lamperti transform.
\end{proof}

Note that one can also obtain a ``slicing at heights'' of $3/2$-stable maps as in \cite{BCK} or \cite{BBCK}. For all $r\ge 1$, let $\mathrm{Ball}^\dagger_r (\mathfrak{m})$ be the ball of radius $r$ centred at $f_r$ in $\mathfrak{m}$. This ball is the submap of $\mathfrak{m}$ {whose faces are the faces} of $\mathfrak{m}$ which are at height at most $r$ {and the adjacency relation between these faces is the same as in $\mathfrak{m}$, except that two faces at distance $r$ to the root face which were adjacent in $\mathfrak{m}$ are no longer adjacent in the ball.}  This ball has several holes of boundaries $\mathfrak{h}_1(\mathfrak{m},r), \mathfrak{h}_2(\mathfrak{m},r)\ldots$ ranked in the non-increasing order of perimeter. For all $i \ge 1$ and $r\ge 0$, let $L_i(r)$ be the (half-)perimeter of the hole $\mathfrak{h}_i(\mathfrak{m},r)$.

Let us recall the notions of self-similar growth-fragmentation. See Subsections 2.2 and 3.1 of \cite{BBCK} or Section 2 of \cite{Ber17} for more details. If $(\cal{X}_u, b_u, \zeta_u)_{u \in \U}$ is a continuous cell system associated with a self-similar Markov process $X$, for all $t\ge 0$ we set 
$$
{\bf X}(t) = \left\{\!\!\{ \cal{X}_u(t-b_u) ; \ u \in \U, \ b_u \le t< b_u + \zeta_u \right\}\!\!\}
$$ 
where the notation $\left\{\!\!\{ \cdots \right\}\!\!\}$ refers to multiset. By Theorem 2 in \cite{Ber17}, the elements of ${\bf X}(t)$ can be ranked in the non-increasing order and form a sequence that tends to $0$, say $X_1(t) \ge X_2(t) \ge \ldots \ge 0$. If ${\bf X}(t)$ has only finitely many elements, say $n$, then we set $X_{n+1}(t)=X_{n+2}(t)= \ldots = 0$. The process $({\bf X}_t)_{t\ge 0}$ is called the growth-fragmentation process associated with $X$.

When $p\ge 1$, let $\ell_{p}^\downarrow$ be the space of non-increasing sequences $(x_i)_{i\ge 0}$ such that $\sum_{i\ge 1} x_i^{p} <\infty$. Let $\mathbb{D}(\R_+, \ell_{p}^\downarrow)$ be the space of càdlàg functions from $\R_+$ to $\ell^\downarrow_{p}$ equipped with the $J_1$ Skorokhod topology. Let ${\bf X}^{(0)} = ((X^{(0)}_i(t))_{i\ge 1})_{t\ge 0}$ 
be the growth-fragmentation associated to the self-similar process $X^{(0)}$. 
Then, using exactly the same proof as the proof of Theorem 6.8 of \cite{BBCK}, we obtain the following scaling limit: 
\begin{proposition}\label{proposition saucisson}
In the space $\mathbb{D}(\R_+, \ell_{5/2}^\downarrow)$,
$$
\text{Under } \P^{(\ell)}, \qquad
\left(\left(\frac{L_i(s \log \ell)}{\ell}\right)_{i\ge 1}\right)_{s\ge 0}
\mathop{\longrightarrow}\limits_{\ell \to \infty}^{(\mathrm{d})}
\left({\bf X}^{(0)}(2 \pi s) \right)_{s\ge 0}.
$$
\end{proposition}

\section{Discussion}\label{section discussion}
\subsection{Application to $O(2)$-decorated maps}\label{sous-section cartes deco}
Here we describe briefly how our results transfer to some critical $O(2)$ loop-decorated planar maps. Let us first recall the definition of a rigid loop-decorated map and the definition of Boltzmann $O(n)$ loop-decorated random planar maps for $n \in (0,2]$. 

A \textit{loop-decorated map} $(\mathfrak{m}, {\bf \mathcal{L}})$ is a (rooted bipartite planar) map $\mathfrak{m}$ together with a loop configuration ${\bf L}= (\mathcal{L}_1, \ldots, \mathcal{L}_k)$ of disjoint unoriented simple closed paths on the dual map $\mathfrak{m}^\dagger$. Furthermore, the loops $\mathcal{L}_i$ do not go through the root face of $\mathfrak{m}$ and are rigid in the sense that they only visit quadrangles and they enter and exit the quadrangles through opposite sides. See Figure \ref{image carte boucles}.

\begin{figure}[h]
	\centering
	\includegraphics[scale=0.65]{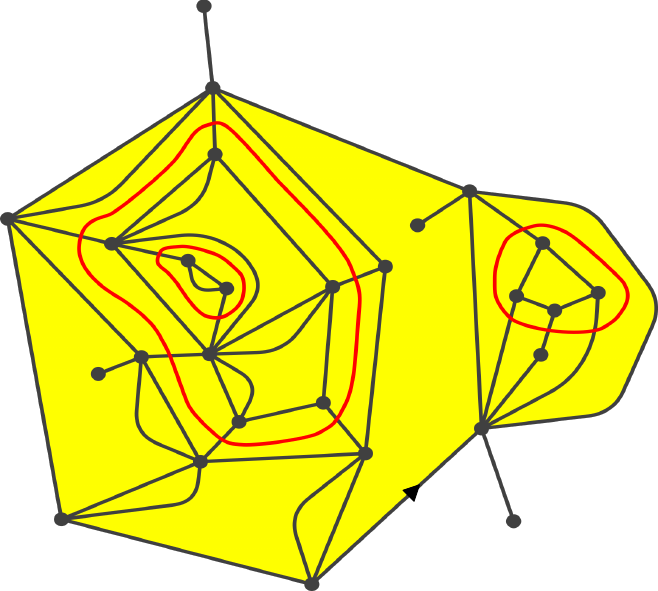}
	\includegraphics[scale=1.42]{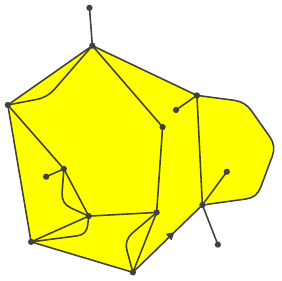}
	\caption{Left: a loop-decorated planar map. The loops, which cross quadrangles, are in red. Right: its gasket, obtained by removing the interior of the outermost loops.}
	\label{image carte boucles}
\end{figure}

By convention, the distances $\mathrm{d}_\mathrm{gr}^\dagger$ and $\mathrm{d}_{\mathrm{fpp}}^\dagger$ are defined on a loop-decorated map by adding the rule that all the quadrangles of each loop are identified to one point, i.e.\@ are at distance zero.

For all $\ell \ge 1$, we introduce the set $\mathcal{LM}^{(\ell)}$ of loop-decorated maps whose root face has degree $2\ell$. 
{Let $\widetilde{\bf q} = (\widetilde{q}_k)_{k\ge 1}$ be a sequence of non-negative real numbers and two real numbers $h,n\ge 0$. We define the weight of a loop decorated map $(\mathfrak{m}, {\bf L})\in \mathcal{LM}^{(\ell)}$ by setting
$$
w_{\widetilde{\bf q}, h,n} (\mathfrak{m}, {\bf L}) = \left(\prod_{\mathcal{L} \in {\bf L}} n h^{\vert \mathcal{L} \vert}\right) \left(\prod_{f \in \mathrm{Faces}(\mathfrak{m})\setminus(\{f_r\}\cup \bigcup_{\mathcal{L} \in {\bf L}} \mathcal{L})} \widetilde{q}_{\deg(f)/2}\right),
$$
where $\vert \mathcal{L} \vert $ is the length of the loop $\mathcal{L}$, i.e.\@ its number of quadrangles. For all $\ell\ge 1$, we recall that the partition function of $O(n)$ loop-decorated maps with triple $(\widetilde{\bf q}, h,n)$ of perimeter $2 \ell$ is defined by
$$
F^{(\ell)}(\widetilde{\bf q}, h,n) = \sum_{(\mathfrak{m}, {\bf L}) \in \mathcal{LM}^{(\ell)}} w_{\widetilde{\bf q}, h,n}(\mathfrak{m}, {\bf L}).
$$
The triple $(\widetilde{\bf q}, h,n)$ is said admissible if $F^{(\ell)}(\widetilde{\bf q}, h, n)<\infty$ for all $\ell \ge 1$. The \textit{Boltzmann probability measure} on $O(n)$ loop-decorated maps $\widetilde{\P}^{(\ell)}$ on $\mathcal{LM}^{(\ell)}$ is then defined by setting 
\begin{equation}\label{eq loi de Boltzmann carte a boucles}
\forall (\mathfrak{m}, {\bf L}) \in \mathcal{LM}^{(\ell)}, \qquad \widetilde{\P}^{(\ell)}(\{(\mathfrak{m}, {\bf L})\}) = \frac{w_{\widetilde{\bf q}, h, n}(\mathfrak{m}, {\bf L})}{F^{(\ell)}(\widetilde{\bf q}, h, n)}.\end{equation}

Let us now recall from \cite{BBG12} the gasket decomposition of a loop-decorated planar map. See also Sections 1.2.1 and 3.1 of \cite{B18} for a presentation of this decomposition. The gasket of a loop-decorated map $(\mathfrak{m}, {\bf L}) \in \mathcal{LM}^{(\ell)}$ is a planar map $\mathfrak{g} \in \mathcal{M}^{(\ell)}$ obtained by removing the interior of the outermost loops of $(\mathfrak{m}, {\bf L})$. See the right-hand side of Figure \ref{image carte boucles}. If we define the weight sequence ${\bf q}= (q_k)_{k\ge 1}$ by setting for all $k\ge 1$,
\begin{equation}\label{eq lien q qtilde}
	q_k = \widetilde{q}_k + n h^{2k} F^{(k)}(\widetilde{\bf q}, h, n),
\end{equation}
then for all $\mathfrak{g}\in \mathcal{M}^{(\ell)}$,
$$
w_{\bf q}(\mathfrak{g})  = \sum_{\substack{(\mathfrak{m}, {\bf L}) \in \mathcal{LM}^{(\ell)}\\ \mathfrak{g} \text{ is the gasket of } (\mathfrak{m}, {\bf L})}}w_{\widetilde{\bf q}, h, n}(\mathfrak{m}, {\bf L}),
$$
where we recall that $w_{\bf q}$ is defined in Subsection \ref{section background}. Recall also from Subsection \ref{section background} that $W^{(\ell)}$ is the partition function of random maps with perimeter $2\ell$.
In particular, summing over $(\mathfrak{m}, {\bf L})\in \mathcal{LM}^{(\ell)}$, we deduce that $W^{(\ell)} = F^{(\ell)}(\widetilde{\bf q}, h,n)$ for all $\ell \ge 1$, and that $\bf q$ is admissible if and only if the triple $(\widetilde{\bf q}, h,n)$ is admissible.

As in \cite{B18}, we say that the triple $(\widetilde{\bf q}, h,n)$ is \textit{non-generic critical} if it is admissible, $n>0$ and $h = 1/c_{\bf q} $. Note that by \eqref{eq lien q qtilde}, a weight sequence ${\bf q} = (q_k)_{k\ge 1}$ is the weight sequence of the gasket of a non-generic critical $O(n)$ loop decorated map if and only if $\bf q$ is admissible and
\begin{equation}\label{eq gasket}
	\forall k\ge 1, \qquad \qquad q_k-n c_{\bf q}^{-2k} W^{(k)} \ge 0.
\end{equation}
Moreover, the associated triple $(\widetilde{\bf q}, h,n)$ of $O(n)$ loop-decorated maps is given by \eqref{eq lien q qtilde} and $h = 1/c_{\bf q}$. By Theorem 1.1 of \cite{Kam24}, see also Equation (1.7) of  \cite{AdSH24} which was obtained independently for the particular case of quadrangulations, we know that if ${\bf q} = (q_k)_{k\ge 1}$ is the weight sequence of the gasket of a non-generic critical $O(2)$ loop-decorated map, then there exists a slowly varying function $L_{\bf q} : \R_+ \to \R_+$ such that $1= O(L_{\bf q}(\ell))$, $L_{\bf q}(\ell) = O(\log \ell)$ and
\begin{equation}\label{eq slowly varying function}
W^{(\ell)} \mathop{\sim}\limits_{\ell \to \infty } c_{\bf q}^{\ell+1} \frac{L_{\bf q}(\ell) }{2\ell^{2}},
\end{equation}
where we recall that $L_{\bf q}$ is a slowly varying function if for all $\lambda>0$, we have $L_{\bf q}(\lambda x) \sim L_{\bf q}(x)$ as $x \to \infty$. When this slowly varying function $L_{\bf q}$ is constant, we get \eqref{eq critique de type deux} with $a=2$. In other words, in our setting, when $L_{\bf q}$ is constant, the gasket of a critical $O(2)$ loop-decorated map is a $3/2$-stable maps. We will make this assumption in the rest of this subsection so that \eqref{eq critique de type deux} holds with $a=2$.
} 
Such a sequence does exist{: in fact, }as shown in Remark 1 of \cite{B18} one can take for all $k\ge 1$
$$
q_k = \frac{1}{(3\pi)^{k-1}} \left(\frac{2}{\pi} \frac{1}{(2k-3)(2k-1)} + {\bf 1}_{k=1}\right).
$$
Using the above-described gasket decomposition {and more precisely \eqref{eq lien q qtilde},} one can construct the critical $O(2)$ Boltzmann measure $\widetilde{\P}^{(\ell)}$ on $\mathcal{LM}^{(\ell)}$ for all $\ell \ge 1$ {defined in \eqref{eq loi de Boltzmann carte a boucles}} in a recursive way as follows. Let $\mathfrak{M}^{(\ell)}$ be a random variable in $\mathcal{M}^{(\ell)}$ of law $\P^{(\ell)}$. Then, independently, on each face $f$ of $\mathfrak{M}^{(\ell)}$, with probability $(2 W^{(p)})/(c_{\bf q}^{2p} q_p)$ where $p=\deg(f)/2$, we glue quadrangles along the edges around $f$ hence forming a rigid loop and we glue in this loop an independent loop-decorated map of law $\widetilde{\P}^{(p)}$. The fact that $\widetilde{\P}^{(\ell)}$ is well defined comes from the admissibility of $\bf q$, from (\ref{eq gasket}) and from \cite{BBG12} or from Theorem 1 of \cite{B18}. 
The map $\mathfrak{M}^{(\ell)}$ is the gasket of the obtained loop-decorated map. See Figure \ref{image carte boucles}.

\paragraph{Exploration of loop-decorated maps and signed cell-systems}{As observed in \cite{B18}, we can extend the peeling exploration of the gasket to a peeling exploration of the loop-decorated map.  More precisely, one can extend the exploration and the cell-system defined in Subsection \ref{sous-section cartes et CLE} to explore the whole critical $O(2)$ loop-decorated map (not only its gasket) as follows: each time we observe a non-negative jump of size $p-1\ge 0$ of the perimeter process, we know that with probability $(2 W^{(p)})/(c_{\bf q}^{2p} q_p)$ it corresponds to the discovery of a loop $\mathcal{L}$ of size $2p$ (and not of a face of degree $2p$). In this case, we run a new exploration of the map encircled by the loop $\mathcal{L}$. We thus obtain a discrete cell-system $((\widetilde{P}_u(k))_{k\ge 0})_{u \in \mathbb{U}}$ which is defined as follows: $\widetilde{P}_\emptyset = P_\emptyset$ and $\widetilde{B}_\emptyset = 0$. Then, assume that the cell $\widetilde{P}_u$ and its birth time $\widetilde{B}_u$ are defined. Let $\widetilde{N}_{uj}$ for $j\ge 1$ be the times of negative jumps and also the non-negative jumps corresponding to the discovery of a loop, in the lexicographic order (using the non-increasing order of absolute value of the sizes of the jumps). We then set $\widetilde{B}_{uj}= \widetilde{B}_u+\widetilde{N}_{uj}$ for all $j\ge 1$ and we define $\widetilde{P}_{uj}$ as the perimeter process of the exploration of the map filling the hole in the case of a negative jump or the map encircled by the loop in the case of a discovery of a loop. In the case of a discovery of a loop, by convention we change the sign of $\widetilde{P}_{uj}$ (i.e. if $\widetilde{P}_u$ is positive, then $\widetilde{P}_{uj}$ is minus the perimeter process, while if $\widetilde{P}_u$ is negative, then $\widetilde{P}_{uj}$ is the perimeter process).

Similarly, as in Subsection 3.1 of \cite{dS23}, we can define the continuous signed cell-system as follows. Let $X$ be a self-simialar Markov process such that $0$ is an absorbing state which tends a.s.\@ towards zero or is a.s.\@ absorbed at zero after a finite time. Let $P_x$ be the law of $X$ starting from $x \in \R_+$. We say that $(\overline{X}_u)_{u\in \U}$ is a continuous signed cell-system associated with $X$ starting at $x\ge 0$ if $\overline{X}_\emptyset$ follows the law $P_x$ and for all $u \in \U$, conditionally on $\overline{X}_u$, if we write $((\overline{x}_j, \overline{\beta}_{uj}))_{j\ge 1}$ the family of negative and positive jumps of $\overline{X}_u$, where $\overline{x}_j \in \R^*$ is the size of the jump and $\overline{\beta}_{uj}$ is the time of the jump, ordered in the non-increasing lexicographic order (the $\overline{x}_j$'s are compared using their absolute values), then the $\overline{X}_{uj}$'s are independent and have the law of $\pm X$ under $P_{\vert \overline{x}_j\vert}$, where the sign is determined as follows: the sign of $\overline{X}_{uj}$ is minus the sign of $\overline{X}_u$ when $\overline{x}_j>0$ and is the same sign when $\overline{x}_j<0$. We define inductively the birth times by setting $\overline{b}_\emptyset = 0$ and $\overline{b}_{uj} = \overline{b}_u + \overline{\beta}_{uj}$. The lifetime of $\overline{X}_u$ is $\overline{\zeta}_u \coloneqq \inf\{t\ge 0, \ \overline{X}_u(t)= 0\}$.

For all $\alpha\in \R$, let $((\overline{\mathcal{X}}^{(\alpha)}_u(t))_{t\ge 0}, (\overline{\beta_{uj}})_{j\ge 1}, \overline{b}_u, \overline{\zeta}_u)_{u\in \U}$ be the continuous signed cell-system associated with the self-similar Markov process $X^{(\alpha)}$. Let us denote by $\overline{\bf X}^{(\alpha)}$ the self-similar signed growth-fragmentation process associated with this cell-system, which is defined for all $t\ge 0$ by
$$
\overline{\bf X}^{(\alpha)}(t) = \left\{\!\!\!\left\{ \overline{\mathcal{X}}^{(\alpha)}_u(t-\overline{b}_u) ; \ u \in \U, \ \overline{b}_u \le t< \overline{b}_u + \overline{\zeta}_u \right\}\!\!\!\right\}.
$$ 
The self-similar signed growth-fragmentation $\overline{\bf X}^{(-1)}$ was studied in details in \cite{AdS22}. As in Section 4 of \cite{AdS22}, we arrange the elements of $\overline{\bf X}^{(\alpha)}(t) $ in the non-increasing order of absolute values.}

\paragraph{Scaling limit of the distances to the root.} Using this gasket decomposition, one can see that Theorem \ref{gros théorème} and Proposition \ref{proposition saucisson} have their counterparts in terms of critical $O(2)$ loop-decorated planar maps.

 {In order to extend Theorem \ref{gros théorème} to critical $O(2)$ loop-decorated maps}, it suffices to replace the faces of the $3/2$-stable map by the loops of the loop-decorated map and the non-nested $\mathrm{CLE}_4$ by a nested $\mathrm{CLE}_4$. {Recall the uniform exploration of the nested CLE$_4$ from Subsection \ref{sous-section explo unif}. The definition of the quantum natural distance defined in \eqref{eq distance quantique target} can be extended beyond the time $\tau^z$, which enables to define the quantum distance $\mathrm{d}_{\mathrm{q}}(\partial \mathbb{D}, \mathcal{L})$ to any loop $\mathcal{L}$ of the nested CLE$_4$. The distance $\mathrm{d}_{\mathrm{Lamperti}}$ can be extended in the same way.}
{
	\begin{corollary}
		For all $\ell \ge 1$, let $(\mathfrak{M}^{(\ell)},{\bf L}^{(\ell)})$ of law $\widetilde{\P}^{(\ell)}$. Let $(\mathcal{L}_i^{(\ell)})_{i\ge 1}$ be the collection of loops of $(\mathfrak{M}^{(\ell)},{\bf L}^{(\ell)})$ ranked in the non-increasing order of size. Let $(n(\mathcal{L}_i^{(\ell)}))_{i\ge 1}$ be their exploration times (i.e. the birth time of their corresponding cells in the discrete cell-system described above) using the peeling by layers algorithm. Let $(\mathcal{L}_i)_{i\ge 1}$ be the collection of loops of the nested $\mathrm{CLE}_4$ ranked in the decreasing order of quantum boundary length. Let us denote by $\nu^2_{h^2}(\mathcal{L})$ the quantum boundary length of a loop $\mathcal{L}$. Then, for the product topology,
		$$
		\left(\frac{\vert\mathcal{L}_i^{(\ell)}\vert}{2\ell}, \frac{n(\mathcal{L}_i^{(\ell)})}{p_{\bf q} \ell}, \frac{2\mathrm{d}^\dagger_\mathrm{gr}(f_r, \mathcal{L}_i^{(\ell)})}{ \log \ell}\right)_{i\ge 1}
		\mathop{\longrightarrow}\limits_{\ell \to \infty}^{(\mathrm{d})}
		\left(\nu^2_{h^2}(\mathcal{L}_i), \mathrm{d}_{\mathrm{q}}(\partial \mathbb{D}, \mathcal{L}_i), \mathrm{d}_{\mathrm{Lamperti}}(\partial \mathbb{D}, \mathcal{L}_i)\right)_{i\ge 1}.
		$$
	\end{corollary}
	The same theorem can be stated for the $\mathrm{SLE}_4^{\langle \mu \rangle}(-2)$ exploration, replacing the quantum natural distance by the quantum length of the trunk. For the fpp distance, the factor $\log \ell$ is replaced by a factor $1/p_{\bf q}$ and the peeling by layers exploration is replaced by the uniform peeling exploration. }
	
	{The above corollary is a consequence of Theorem \ref{gros théorème}, of the fact that the maps encircled by the outermost loops are independent of law $\widetilde{\P}^{ (\ell)}$ conditionally on their perimeter, of the Markov property for the uniform exploration of the nested CLE$_4$  and of Theorem \ref{prop de Markov disques quantiques} which states in particular that the encircled domains are independent $2$ quantum disks conditionally on their quantum boundary lengths. 
		
	The only difficulty is to see why in the limit all the positive jumps give rise to a loop while it is not the case in the exploration of the loop-decorated map. This comes from the fact that the jumps of the Cauchy process are symmetric. In other words, the terms $q_k$ and $2c_{\bf q}^{-2k} W^{(k)}$ are ``of the same order'' so that the weights of the faces $\widetilde{q}_k$ are ``negligible'' compared with $2c_{\bf q}^{-2k} W^{(k)}$. More precisely, as explained in Subsection 6.4 of \cite{BBCK}, the perimeter process is absolutely continuous with respect to the random walk with step distribution $\nu$ given by $\nu(k) =q_{k+1} c_{\bf q}^k$ and $\nu(-k-1) = 2W^{(k)} c_{\bf q}^{-k-1}$ for all $k\ge 0$. Here, by \eqref{eq lien q qtilde}, each non-negative jump of size $k-1\ge 0$ corresponds to the discovery of a loop with probability $2W^{(k)}/(c_{\bf q}^{2k} q_k) = \nu(-k-1)/\nu(k-1)$ but by Proposition 5.10 of \cite{StFlour} we know that $\nu([k,\infty))\sim p_{\bf q}/k \sim \nu((-\infty,-k])$ as $k\to \infty$. Thus, the PPP of the positive jumps of the Cauchy process appearing as the scaling limit of the random walk is the scaling limit of the non-negative jumps corresponding to the discovery of a loop.}

{For Proposition \ref{proposition saucisson}, one has to add a sign $(-1)^{N(\mathfrak{h})}$ to the perimeter of each hole $\mathfrak{h}$, where $N(\mathfrak{h})$ is the number of loops surrounding $\mathfrak{h}$ and one has to replace the growth-fragmentation ${\bf X}^{(0)}$ by the signed growth fragmentation $\overline{\bf X}^{(0)}$. Notice that this signed growth-fragmentation corresponds up to a Lamperti transform to the signed growth-fragmentation $\overline{\bf X}^{(-1)}$ studied in \cite{AdS22}. 

For all $r\ge 1$, let $\mathrm{Ball}^\dagger_r ((\mathfrak{m}, {\bf L}))$ be the ball of radius $r$ centred at $f_r$ in the loop decorated map $(\mathfrak{m}, {\bf L})$ for the distance $\mathrm{d}^\dagger_{\mathrm{gr}}$. This ball has several holes of boundaries $\mathfrak{h}_1(\mathfrak{m},r), \mathfrak{h}_2(\mathfrak{m},r)\ldots$ ranked in the non-increasing order of perimeter. For all $i \ge 1$ and $r\ge 0$, let $L_i(r)$ be $(-1)^{N(\mathfrak{h}_i)}$ times the (half-)perimeter of the hole $\mathfrak{h}_i(\mathfrak{m},r)$. One obtains:
\begin{proposition}
	In the space $\mathbb{D}(\R_+, \ell_{2}^\downarrow)$,
	$$
	\text{Under } \widetilde{\P}^{(\ell)}, \qquad
	\left(\left(\frac{L_i(\lfloor s \log \ell\rfloor)}{\ell}\right)_{i\ge 1}\right)_{s\ge 0}
	\mathop{\longrightarrow}\limits_{\ell \to \infty}^{(\mathrm{d})}
	\left(\overline{\bf X}^{(0)}(2 \pi s) \right)_{s\ge 0}.
	$$
\end{proposition}}

{The proof goes along the same lines as the proof of Theorem 6.8 of \cite{BBCK}.} The convergence towards the growth-fragmentation holds here in $\mathbb{D}(\R_+, \ell^\downarrow_2)$ since the harmonic functions associated to the exploration of Boltzmann planar maps which give rise to the martingales of Proposition 6.4 in \cite{BBCK} are replaced by the harmonic function $h^\downarrow_\mathfrak{p}$ for $\mathfrak{p}=1$ in Equation (5) of \cite{B18}, which is constant equal to $1$.

{We believe that the above results have their counterparts when the slowly varying function $L_{\bf q}$ in \eqref{eq slowly varying function} is not constant but one has to adapt the techniques to deal with this case.}

\subsection{Why the author thinks that the conformally invariant distance is not the Lamperti transform of the quantum distance}\label{sous-section heuristique}
We finally argue why we believe that the conformally invariant distance to the boundary $\mathrm{d}_{\mathrm{WW}}$ of \cite{WW13} is not the Lamperti transform $\mathrm{d}_{\mathrm{Lamperti}}$ of the quantum distance.

Let us consider, as before, a unit boundary length $2$-quantum disk $(\mathbb{D}, h^2)$ on which we draw an independent (non-nested) $\mathrm{CLE}_4$ using the uniform exploration. Let $Z$ be an independent random point of law $\mu^2_{h^2}/\mu^2_{h^2}(\mathbb{D})$ and let $\mathcal{L}$ be the loop surrounding $Z$. By definition of the uniform exploration, conditionally on $h^2$ and on $Z$, the distance $\mathrm{d}_\mathrm{WW}(\partial \mathbb{D},\mathcal{L})$ is an exponential random variable of parameter $\lambda$ for some constant $\lambda>0$.

{Indeed, since $Z$ is independent from the uniform exploration, this distance has the same law as $\mathrm{d}_\mathrm{WW}(\partial \mathbb{D},\mathcal{L}^z)$, where $\mathcal{L}^z$ is the loop which encircles a deterministic point $z\in \mathbb{D}$. But coming back to the definition of the uniform exploration on $\mathbb{H}$ in Subsection \ref{sous-section explo unif}, the distance to the loop surrounding $z \in \mathbb{H}$ corresponds to the time $\tau^z$ when a loop of the PPP of loops $(\widehat{\gamma}_u)_{u\ge 0}$ encircles $z$, which is an exponential random variable since it is the time of the first atom of the restriction of the PPP $(\widehat{\gamma}_u)_{u\ge 0}$ to the loops surrounding $z$.}

 In particular, if we stop the exploration at distance $u>0$, conditionally on not having encircled $Z$ before time $u$, the distance $\mathrm{d}_\mathrm{WW}(\partial \mathbb{D},\mathcal{L})-u$ from $\mathcal{L}$ to the boundary of the unexplored region is still an exponential random variable of parameter $\lambda$. On the contrary, if we stop the exploration at Lamperti transformed quantum distance $t>0$, then, conditionally on not having encircled $Z$ yet, the point $Z$ lies in each unexplored quantum disk with probability proportional to its quantum area. So the quantum disk containing $Z$ is {in some sense weighted} by its area, so that the absence of memory for $\mathrm{d}_\mathrm{Lamperti}(\partial \mathbb{D}, \mathcal{L})$ should be ruled out.

Nevertheless, one could start the uniform exploration from a loop of the $\mathrm{CLE}_4$ and measure the quantum boundary lengths of the loops it discovers, which should provide a way to measure the $\mathrm{d}_\mathrm{Lamperti}$-distance from a given loop. Since $\mathrm{d}_\mathrm{Lamperti}$ is the scaling limit of the distances to the boundary in $3/2$-stable maps, it is natural to conjecture that such an extension of $\mathrm{d}_\mathrm{Lamperti}$ is a measurable function of the $\mathrm{CLE}_4$-decorated $2$-LQG and is a distance, which would be the scaling limit of $3/2$-stable maps.

\paragraph{Acknowledgements.}
I am deeply indebted to Matthis Lehmkuehler for his guidance in the exploration of CLEs as well as in the powder snow of Les Diablerets, especially for his help for the proof of Proposition \ref{prop cv CLE}. I also thank Wendelin Werner for his advice. I am grateful to my supervisors Cyril Marzouk and Nicolas Curien for their invaluable support. I thank all the anonymous referees for their reading and for their extremely valuable and numerous comments and suggestions which improved greatly the paper.

\bibliographystyle{alpha}
\bibliography{distances_CLE4_nouveau}

\end{document}